\documentclass[11pt,reqno]{amsart}
\usepackage{amssymb}\usepackage{cancel}
\usepackage{esint}
\usepackage{amsfonts}
\usepackage{mathrsfs}
\usepackage{amsmath}
\usepackage{graphicx}
\usepackage{hyperref}
\usepackage{float}
\usepackage{epstopdf}
\usepackage{color}
\usepackage{bm}
\usepackage{comment}
\usepackage{soul}
\usepackage{esint}
\usepackage{amsthm}

\allowdisplaybreaks
\newtheorem{theorem}{Theorem}[section]
\newtheorem{proposition}[theorem]{Proposition}
\newtheorem{lemma}[theorem]{Lemma}
\theoremstyle{definition}
\newtheorem{corollary}[theorem]{Corollary}
\newtheorem{remark}[theorem]{Remark}

\newtheorem{example}[theorem]{Example}
\newtheorem{definition}[theorem]{Definition}

\def\diam{\mathrm{diam}}

\def\mcD{\mathcal{D}}
\def\mcE{\mathcal{E}}
\def\mcF{\mathcal{F}}
\def\msG{\mathscr{G}}

\def\USC{\mathcal{USC}}

\def\sn{\stackrel{n}{\sim}}
\def\sm{\stackrel{m}{\sim}}

\def\mbEn{\mathbb{E}^{(n)}}
\def\mbEmn{\mathbb{E}^{(m,n)}}
\def\mbPn{\mathbb{P}^{(n)}}
\def\mbPmn{\mathbb{P}^{(m,n)}}
\def\pn{p^{(n)}}

\def\phin{\phi^{(n)}}

\numberwithin{equation}{section}

\begin{document}
\title[Dirichlet forms on unconstrained Sierpinski carpets in $\mathbb{R}^3$]{Dirichlet forms on unconstrained Sierpinski carpets in $\mathbb{R}^3$}

\author{Shiping Cao}
\address{Department of Mathematics, University of Washington, Seattle 98105, USA}
\email{spcao@uw.edu}
\thanks{}

\author{Hua Qiu}
\address{Department of Mathematics, Nanjing University, Nanjing, 210093, P. R. China.}
\thanks{The first author was supported by a grant from the Simons Foundation Targeted Grant (917524) to the Pacific Institute for the Mathematical Sciences. The second author was supported by the National Natural Science Foundation of China, grant 12071213, and the Natural Science Foundation of Jiangsu Province in China, grant BK20211142.}
\email{huaqiu@nju.edu.cn}


\date{}


\maketitle

\begin{abstract}
We prove the existence of a strongly local, regular, self-similar Dirichlet form with a sub-Gaussian heat kernel estimate on an unconstrained Sierpinski carpet in $\mathbb{R}^3$. In the setting under consideration, the walk dimension $d_W$ and the Hausdorff dimension $d_H$ always satisfy the inequality that $d_H>d_W$.
\end{abstract}

\tableofcontents

\section{Introduction}\label{sec1}
In this paper, we prove the existence of a self-similar Dirichlet form with a sub-Gaussian heat kernel estimate on an unconstrained Sierpinski carpet in $\mathbb{R}^3$ ($\USC^{(3)}$). \vspace{0.2cm}

\noindent(\textbf{\emph{Unconstrained Sierpinski carpets in $\mathbb{R}^3$ ($\mathcal{USC}^{(3)}$)}}). We let $\square:=[0,1]^3\subset \mathbb{R}^3$, which is the unit cube in $\mathbb{R}^3$, and let $\msG$ be the \textit{group of self-isometries} on $\square$. 

Let $k\geq 3$ and $8+12(k-2)+6(k-2)^2\leq N\leq k^3-1$. Let $\{\Psi_i\}_{1\leq i\leq N}$ be a finite set of similarities with the form $\Psi_i(x)=\frac x k+ c_i$, $c_i\in\mathbb{R}^3$. Assume the following holds:
	
\noindent\emph{(Non-overlapping).} For $i\neq j$, the Lebesgue measure (on $\mathbb{R}^3$) of $\Psi_i(\square)\cap \Psi_j(\square)$ is $0$.
	
\noindent\emph{(Face included).} $\partial\square\subset \bigcup_{i=1}^N \Psi_i(\square)\subset\square$.
	
\noindent\emph{(Strong connectivity).} $\bigcup_{i=1}^N \Psi_i(\square)$ is connected and has no local cut-point.
	
\noindent\emph{(Symmetry).} $\Gamma\big(\bigcup_{i=1}^N \Psi_i(\square)\big)=\bigcup_{i=1}^N \Psi_i(\square)$ for any $\Gamma\in \msG$.
	
Then, we call the unique compact subset $K\subset \square$ satisfying
\[K=\bigcup_{i=1}^N \Psi_i(K)\]
an \textit{unconstrained Sierpinski carpet} in $\mathbb{R}^3$ ($\USC^{(3)}$). See Remarks \ref{remarka}--\ref{remarkd} in Section \ref{sec2}.\vspace{0.2cm}

The class $\USC^{(3)}$ is a natural higher dimensional analog of the planar unconstrained Sierpinski carpets ($\USC$) studied by the authors in an earlier work \cite{CQ1}, originated from the work of Kusuoka and Zhou \cite{KZ}. The adjective ``unconstrained'' means that we allow cells to live off the $1/k$-grids of $\square$ so that the intersection of two cells can be a line segment of irrational length or a rectangle of irrational area. As a consequence of higher dimensions, we expect the diffusion processes on $\USC^{(3)}$ to be transient (more precisely, we expect $d_H>d_W$ and we can only say that a process is transient on an infinite blowup of an $\USC^{(3)}$), and much difficulty arises. \vspace{0.2cm}

In history, most diffusion processes on self-similar sets that were constructed are pointwisely recurrent.

The area was initiated by Kusuoka \cite{kus}, Goldstein \cite{G} and Barlow-Perkins \cite{BP}, who constructed the diffusion process on the Sierpinski gasket and deeply studied the properties. The probabilistic construction was extended to nested fractals \cite{Lindstrom}. A purely analytical method, which becomes the standard point of view, was introduced by Kigami \cite{ki1,ki2}, where he defined Dirichlet forms as monotone limits of energy forms on discrete graphs and introduced a much broader class named post-critically finite (p.c.f.) self-similar sets. Readers can find  the construction of the Dirichlet forms in \cite{B,ki3,s}. Also see books \cite{CF,FOT} for background knowledge about Dirichlet forms.

A more difficult study was the construction of locally symmetric diffusion processes  on the Sierpinski carpets by Barlow-Bass \cite{BB1}. Deep studies on such processes were proceeded in a sequence of papers \cite{BB2,BB3,BB4,BB5,BBKT}, including the heat kernel estimates \cite{BB3,BB5} and a celebrated uniqueness theorem \cite{BBKT}. The fractals are non-p.c.f. since two neighboring cells may intersect on a line segment, so the method of graph approximations does not work easily. Still, in \cite{KZ}, Kusuoka and Zhou introduced a beautiful framework of defining self-similar Dirichlet forms on a large class of fractals satisfying some geometric conditions (A1)-(A4). But to fulfill the story, one still need verify a few more conditions, for example in \cite{KZ}, Kusuoka and Zhou used the ``Knight move'' argument (see condition (KM) in \cite{KZ}) by Barlow-Bass (see \cite{BB1,BB5}) to prove the Harnack inequality; in \cite{CQ1}, the authors used the strongly recurrent condition to verify that all the Poincar\'e constants are comparable (see condition (B) in \cite{CQ1}).

It is known that on p.c.f. self-similar sets and on planar Sierpinski carpets, the diffusion process can hit a single point. There are few natural examples by direct construction (of course, we can construct rich examples by taking the product or modifying the spaces properly, see \cite{B2} for example), that are known to be transient or not pointwisely recurrent on fractals with sub-Gaussian heat kernel estimates, i.e. for some constants $c_1$-$c_4>0$,
\begin{equation}\label{eqn11}
\frac{c_1}{t^{d_H/d_W}}e^{-c_2(\frac{d(x,y)^{d_W}}{t})^{\frac{1}{d_W-1}}}\leq p_t(x,y)\leq \frac{c_3}{t^{d_H/d_W}}e^{-c_4(\frac{d(x,y)^{d_W}}{t})^{\frac{1}{d_W-1}}},\ \forall 0<t\leq 1.
\end{equation}
In \cite{BB5}, Barlow and Bass constructed the diffusion processes on (higher dimensional) generalized Sierpinski carpets. We also need to mention the recent work of Murugan \cite{M1}, where a class of recurrent but not pointwisely recurrent diffusion processes with sub-Gaussian heat kernel estimates were constructed on snow balls via quasi-symmetry and circle packing.\vspace{0.2cm}

In this paper, we provide another class of fractals that has transient diffusion processes with sub-Gaussian heat kernel  estimates. This class contains irrationally ramified fractals as introduced in \cite{ki4}. We state our theorem precisely below.\vspace{0.2cm}

\noindent(\textbf{\emph{Standard self-similar Dirichlet forms (SsDF)}}). Let $K$ be a $\USC^{(3)}$, and $\mu$ be the normalized Hausdorff measure restricted on $K$. Let $(\mcE,\mcF)$ be a regular conservative irreducible symmetric Dirichlet form on $L^2(K,\mu)$. We call $(\mcE,\mcF)$ a \textit{standard self-similar Dirichlet form} on $K$ if
\begin{equation}\label{eqn1}
	\mcF\cap C(K)=\big\{f\in C(K):\  f\circ \Psi_i\in\mcF \text{ for any } i\in \{1,\dots, N\}\big\},
\end{equation}
and there exists $\rho>0$ such that
\begin{equation}\label{eqn2}
	\mcE(f,f)=\rho^{-1}\sum_{i=1}^N \mcE(f\circ \Psi_i, f\circ \Psi_i) \text{ for any } f\in\mcF\cap C(K).	
\end{equation}
We abbreviate such Dirichlet forms to SsDF and write $\mcE(f):=\mcE(f,f)$ for short later.

\begin{theorem}\label{thm1}
	Let $K$ be a $\USC^{(3)}$. Then there is an SsDF $(\mcE,\mcF)$ on $K$ such that a sub-Gaussian estimate (\ref{eqn11}) holds for the associated heat kernel with $d_H>d_W$. 
\end{theorem}

We will use the framework of Kusuoka-Zhou \cite{KZ}, which was modified a little in \cite{CQ1}. As pointed out, the main difficulty, that was not considered in previous papers \cite{KZ,CQ1,CQW}, arises from the transience of the process. 

On the other hand, during the past decades, equivalent characterizations for heat kernel estimates are extensively developed, which provide many useful tools and insights for us to study diffusion processes on concrete fractals. In a sequence of papers \cite{AB,BB6,GHL2}, the Gaussian type heat kernel estimates are shown to be equivalent to the Poincar\'e inequality and the cut-off Sobolev inequality, where the latter can be verified with the mean escaping time estimate \cite{AB,GHL2}. Deep applications on the stability of Harnack inequalities were also studied \cite{BCM,BM}. Moreover, people are still working on establishing easier characterizations. One of the well-known conjectures, named the resistance conjecture (\cite[Remark 3.17(1)]{B3}, \cite[p.1495]{GHL2}), states that the Poincar\'e inequality and the resistance estimate (between a ball and the complement of its neighborhood) are enough to derive the two-sides heat kernel estimate. Though the conjecture remains open so far, in \cite{M2}, Murugan solved it for the $d_H<d_W+1$ case, and we will need his result to finish the story of this paper. A key observation of Murugan's paper \cite{M2} is that one dimensional objects, for example a curve, are always hittable by the diffusion as long as $d_H<d_W+1$ and the Poincar\'e inequality that matches $d_W$ holds. Our approach in this paper also relies on a similar fact, and thus, the situation for higher dimensional (greater than $3$) analogs to $\USC$ are still not fully understood.\vspace{0.2cm}

Finally, we briefly introduce the main idea of the paper. 

Just as in the authors' previous work dealing with planar $\mathcal{USC}$ \cite{CQ1} (see also \cite{CQW}), the most difficult part about the existence is to show $R_n\gtrsim N^{-n}\lambda_n$, where $R_n$ are resistance constants and $\lambda_n$ are Poincar\'e constants (precise definitions and Kusuoka-Zhou's estimates will be recalled in Section \ref{sec3}). Recall that for standard Sierpinski carpets \cite{BB1,BB5}, this was solved by using the ``corner move'' and ``Knight move'' argument which essentially depending on the exact way that cells intersect. As for $\mathcal {USC}$, to overcome the difficulty caused by the worse geometry, we took two (purely analytic) steps in \cite{CQ1}: first, we obtained a face-to-face resistance estimate via a chain argument, using the fact that the functions in the domain are H\"older continuous; then, we used an extension argument to create nice bump functions with small energy. In this paper, the first step no longer works due to transience, while the second is kept. \vspace{0.2cm}

The major part of the paper, including part of Section \ref{sec2} and Sections \ref{sec4}--\ref{sec7}, is about the face-to-face resistance estimate. The proof can be split into three steps.

In the first step (Section \ref{sec4}), we consider a more complicated chain argument as in \cite{CQ1} to give a face-to-face resistance estimate in the sense of averaged values: we prove that there are functions with large differences in the average values on neighboring faces and relatively small energies. This step is purely analytic, and can be extended to even higher dimensional case without difficulty.

In the second step (Section \ref{sec6}), we show a consequence of the Poincar\'e inequality and the reflection symmetry: if the random walk on a approximating cell graph starting from one point of a face hits the opposite face in a short time, it has the potential to oscillate between the two faces quickly for a few times. Using this observation, if we pile up copies of a cell graph to get a ``vertical wall'' (consider all the $m$-cells attached to one face of the fractal), we can show that under the same assumption,  the diffusion hits the bottom face of the wall quickly from corresponding initial points, via a coupling argument (Section \ref{sec5}). 

In the last step (Section \ref{sec7}), we combine the two steps together to prove the face-to-face resistance estimate. 

Section \ref{sec8} is same as that in \cite{CQ1,CQW}, where we construct bump functions via a delicate extension argument.

Finally, we finish the story in Sections \ref{sec9}-- \ref{sec10}. In Section \ref{sec9}, we recall some facts about equivalent characterizations of the heat kernel estimates, and in particular, using \cite{M2} by Murugan and using \cite{AB,BB6,GHL2}, we see the diffusions on cable systems of the cell graphs have nice heat kernel estimates. Then, noticing that the heat kernels are equicontinuous, we can take a subsequential limit directly to get a Markov kernel on a $\USC^{(3)}$. Thus, we get a limit diffusion process with a sub-Gaussian heat kernel estimate. In Section \ref{sec10}, we use the strategy in \cite{CQ2} by the authors to construct a self-similar Dirichlet form, that is comparable with the Dirichlet form associated with the limit process in Section \ref{sec9}. This finishes the proof by the stability of the equivalent characterizations \cite{AB,BCM,BM,GHL2}.  \vspace{0.2cm}

Before ending this section, we remark that since we would like the fractal to contain each face of the initial cube so that we can slide inside cells without too much worry about the bad geometry, the $\USC^{(3)}$ class will not cover all generalized Sierpinski carpets  \cite{BB5,BBKT} in $\mathbb{R}^3$. However, in the case of $d_H<3$, the method of this paper can be easily modified to work on generalized Sierpinski carpets. \vspace{0.2cm}
 
Throughout the paper, we will simply write $x$ (or $y, z$) for a point in $\mathbb{R}^3$, and from time to time we write $x=(x_1,x_2,x_3)$ to specify the three coordinates of $x$ (we will also use $o$, short for \textit{orientation}, for an undetermined index $1,2,3$).  As we already did in the definition of $\USC^{(3)}$, for each $x=(x_1,x_2,x_3)\in \mathbb{R}^3$ and $t\in \mathbb{R}$, we write $tx:=(tx_1,tx_2,tx_3)\in \mathbb{R}^3$. 

We always use $d$ to denote the Euclidean metric on $\mathbb{R}^3$ unless there is further notice. As usual, for $x\in\mathbb{R}^3$ and $r>0$, we will denote $B(x,r):=\{y\in \mathbb{R}^3: d(x,y)<r\}$ the open ball centered at $x$ with radius $r$; and for $A,B\subset \mathbb{R}^3$, we will write
\[d(A,B)=\inf_{x\in A,y\in B}d(x,y),\]
which is positive if $A,B$ are disjoint compact sets. Also write $d(x,A)=d(\{x\},A)$ for short.

From time to time, we write $a\lesssim b$ for two variables (functions, forms) if there is a constant $C>0$ such that $a\leq C\cdot b$, and write $a\asymp b$ if both $a\lesssim b$ and $b\lesssim a$ hold. For two reals $a,b$, we write $\lfloor a\rfloor:=\max\{n\in\mathbb{Z}: n\leq a\}$ and $\lceil a\rceil:=\min\{n\in\mathbb Z: n\geq a\}$, and always abbreviate that $a\wedge b=\min\{a,b\}$ and $a\vee b=\max\{a,b\}$.

\section{Cell graphs}\label{sec2}
In this section, we present some geometric properties  of $\mathcal{USC}^{(3)}$, and introduce the associated cell graphs. To help readers digest the conditions in the definition of $\USC^{(3)}$, we begin this section with some remarks.

\begin{remark}\label{remarka} A point $x$ is a \textit{local cut-point} of a metric space $(X,d)$ if there is a connected neighborhood $U$ of $x$ such that $U\setminus \{x\}$ is not connected. It is necessary to exclude the existence of local cut-points in a $\USC^{(3)}$, since we expect that the diffusion process to be constructed  is transient and each single point set is polar.
	
We also remark here that the strong connectivity condition is easy to verify. A local cut-point exists if and only if the ``\textit{diagonal}'' case happens: there are $1\leq i,i'\leq N$ such that $\#(\Psi_i \square\cap \Psi_{i'}\square)=1$ and in addition there doesn't exist $i''\in \{1,2,\cdots,N\}\setminus\{i,i'\}$ such that $\Psi_i\square\cap \Psi_{i'}\square\subset \Psi_{i''}\square$.
\end{remark}

\begin{remark} To reduce the number of brackets, in the following context, we always write $\Psi_iA$ instead of $\Psi_i(A)$ for $A\subset\square$. We do similarly for the similarities $\Psi_w$ that will be introduced in Definition \ref{def27}.
\end{remark}

\begin{remark} Face included condition is a technical assumption for creating cut-off functions. The method that will be developed in this paper should extend to some other fractals under a weaker boundary condition.
\end{remark}

\begin{remark}\label{remarkd} The condition $k\geq 3$ is to avoid trivial set by the symmetry condition. The condition $N\geq 8+12(k-2)+6(k-2)^2$ is a requirement of the face included condition, and the condition $N\leq k^3-1$ ensures that we are dealing with a non-trivial self-similar set.
\end{remark}

Now we proceed to study the basic properties of $\USC^{(3)}$. Especially, we will focus on a consequence of the strong connectivity condition (see Lemma \ref{lemma28}). Also, we will define a sequence of graphs on the cells of a $\USC^{(3)}$ (denoted as $K$), and show properties (A1)-(A4) as in \cite{KZ}. We first list some useful notations. \vspace{0.2cm}

\noindent(\textbf{\emph{$H_{o,s}$-planes}}). For $o\in \{1,2,3\}$ and $s\in \mathbb{R}$, we define $H_{o,s}$ the plane in $\mathbb{R}^3$ as
\[H_{o,s}=\{x=(x_1,x_2,x_3)\in\mathbb{R}^3: x_o=s\}.\]

\noindent(\textbf{\emph{Faces}}). For $o\in \{1,2,3\}$ and $s\in \{0,1\}$, we define
\[F_{o,s}=\square\cap H_{o,s},\]
and call it a \textit{face} of $\square$. For $o\in\{1,2,3\},s\in \{0,1\}$, we say $F_{o,1-s}$ is the face opposite to $F_{o,s}$, and for  $o\neq o'\in \{1,2,3\}$ and $s,s'\in \{0,1\}$, we say $F_{o',s'}$ is a neighboring face of $F_{o,s}$. Clearly, two faces are opposite if they are disjoint, otherwise they are neighboring. We denote the boundary of $\square$ by $\partial \square$, i.e. $\partial \square=\bigcup_{o=1}^3\bigcup_{s=0}^1 F_{o,s}$. \vspace{0.2cm}

\noindent(\textbf{\emph{The group of symmetry}}). We denote $\msG$ the \textit{group of self-isometries} on $\square$. This symmetry group is larger than that for the $2$-dimensional case. We only name those symmetries in $\msG$ that will be frequently used. \vspace{0.2cm}

(\textbf{\emph{The reflection $\Gamma_{[o]}$}}). For $o\in \{1,2,3\}$, we define $\Gamma_{[o]}:\square\to\square$ as $\Gamma_{[o]}(x)=x'$ so that $x_o+x'_o=1$ and $x_{o'}=x'_{o'},\forall o'\neq o$.

(\textbf{\emph{The reflection $\Gamma_{[o,o']}$}}). For $\{o,o',o''\}=\{1,2,3\}$, we define $\Gamma_{[o,o']}:\square\to\square$ as $\Gamma_{[o,o']}(x)=x'$ so that $x_{o}=x'_{o'}$, $x_{o'}=x'_{o}$ and $x_{o''}=x'_{o''}$.  \vspace{0.2cm}

Clearly, the collection of all reflections $\Gamma_{[o]}$ and $\Gamma_{[o,o']}$  generates $\msG$.

\begin{definition}\label{def27}
	Let $W_0=\{\emptyset\},W_1=\{1,2,\cdots,N\}$ be the \textit{alphabet} associated with $K$,  and for $n\geq 1$, $W_n=W_1^n:=\{w=w_1\cdots w_n:w_i\in W_1,\forall 1\leq i\leq n\}$ be the collection of \textit{words} of length $n$. Also, write $W_*=\bigcup_{n=0}^\infty W_n$ for the collection of words of finite length.
	
	(a). For each $w=w_1\cdots w_n\in W_n$, we denote $|w|=n$ the \textit{length} of $w$, and write
	\[\Psi_w=\Psi_{w_1}\circ \cdots \circ \Psi_{w_n}.\]
	So each $w\in W_n$ represents an \textit{$n$-cell} $\Psi_wK$ in $K$.

	(b). For $w=w_1\cdots w_n\in W_n,w'=w'_1\cdots w'_m\in W_m$, we write \[w\cdot w'=w_1\cdots w_n w'_1\cdots w'_m\in W_{n+m}.\]  For $A\subset W_n,B\subset W_m$, we denote $A\cdot B=\{w\cdot w':w\in A,w'\in B\}$. In particular, we abbreviate  $\{w\}\cdot B$ to $w\cdot B$.
\end{definition}

\begin{definition}\label{def29}
	For $o\in \{1,2,3\}$, $s\in\{0,1\}$, and $n\geq 0$, we define the \textit{faces} of $W_n$ as
	\begin{align*}
		\tilde{F}_{n,o, s}&=\{w\in W_n:\Psi_wK\cap F_{o,s}\neq \emptyset\},
	\end{align*}
	and write
	\[\partial W_n=\bigcup_{o=1}^3\bigcup_{s=0}^1\tilde{F}_{n,o,s},\]
	which represents the \textit{border cells(words)} of $W_n$. From time to time, we also use the notation
	\[(\partial W_1)^n:=\{w=w_1\cdots w_n\in W_n:w_i\in \partial W_1,\forall 1\leq i\leq n\}.\]
\end{definition}

The following Lemma \ref{lemma28} shows that  cells are connected in a strong way.
\begin{lemma}\label{lemma28}
	There is a positive constant $C$ depending only on $K$ such that for any $n\geq 1$ and any distinct $w,w'\in W_n$, the following holds.
	
	(a). There is a sequence $w=w^{(0)},\cdots,w^{(L)}=w'$, such that for any $0<l\leq L$,  $\Psi_{w^{(l-1)}}K\cap \Psi_{w^{(l)}}K\neq\emptyset$, and in addition,\\
	\noindent (a.1). if $\Psi_{w^{(l-1)}}K\cap \Psi_{w^{(l)}}K$ is a line segment, the length of $\Psi_{w^{(l-1)}}K\cap \Psi_{w^{(l)}}K$ is larger than $Ck^{-n}$;\\
	\noindent (a.2). if $\Psi_{w^{(l-1)}}K\cap \Psi_{w^{(l)}}K$ is a rectangle, the area of $\Psi_{w^{(l-1)}}K\cap \Psi_{w^{(l)}}K$ is larger than $Ck^{-2n}$.
	
	(b). In particular, if $\Psi_wK\cap \Psi_{w'}K\neq\emptyset$, we can require $L\leq 3$ in (a).
\end{lemma}
\begin{proof}
	(a) is an easy consequence of (b). In fact, by the strong connectivity condition, for any $n\geq 1$ and $w,w'\in W_n$, there is a sequence $w=w^{(0)},\cdots,w^{(L)}=w'$, such that $\Psi_{w^{(l-1)}}K\cap \Psi_{w^{(l)}}K\neq \emptyset$ for any $0<l\leq L$, then we simply apply (b) to get a sequence satisfying (a.1), (a.2) of length less than $3L+1$. Hence, it suffices to prove (b).
	
	In the following, we use $\mathcal{H}^t$ to denote the $t$-dimensional Hausdorff measure for $t=1,2$ (normalized so that the unit interval has $\mathcal{H}^1$-measure $1$, the unit square has $\mathcal{H}^2$-measure $1$). We let
	\[\begin{aligned}
		C'=&k\cdot \min\big\{\mathcal H^1\big(P_o(\Psi_iK\cap \Psi_jK)\big):  \mathcal H^1\big(P_o(\Psi_iK\cap \Psi_jK)\big)>0, \ 1\leq i\neq j\leq N,\  o=1,2,3\big\},
	\end{aligned}\]
	where $P_o$ is the orthogonal projection defined as $P_o(x)=x_o$, $o=1,2,3$. We will prove the lemma for
	\[C:=\frac{C'^2}{81}=\min\big\{\frac{C'}{9},\frac{1}{9},\frac{C'^2}{81},\frac{C'}{81},\frac{1}{81}\big\}.\]
	
	Now, let $w,w'$ be the words in (b). We can find $1\leq n'\leq n$ such that
	\[w=\iota\cdot i\cdot \tau,\quad w'=\iota\cdot i'\cdot \tau', \]
	for some $\iota\in W_{n'-1}$, $i\neq i'\in W_1$ and $\tau,\tau'\in \partial W_{n-n'}$. Let $B=\Psi_{\iota\cdot i}K\cap \Psi_{\iota\cdot i'}K$.
	
	First, if $\#B=1$, by the strong connectivity condition, there exists $i''\in W_1\setminus\{i,i'\}$ such that $B\subset \Psi_{\iota\cdot i''}K$, since otherwise $\Psi_\iota^{-1}(B)$ is a local cut-point of $\bigcup_{j=1}^N \Psi_j\square$. Pick $\tau''\in \partial W_{n-n'}$ such that $B\subset \Psi_{\iota\cdot i''\cdot \tau''}K$. Then one can see either that $\mathcal{H}^1(\Psi_wK\cap \Psi_{\iota\cdot i''\cdot \tau''}K)=k^{-n}, \mathcal{H}^2(\Psi_{w'}K\cap \Psi_{\iota\cdot i''\cdot \tau''}K)=k^{-2n}$ or $\mathcal{H}^2(\Psi_wK\cap \Psi_{\iota\cdot i''\cdot \tau''}K)=k^{-2n}, \mathcal{H}^1(\Psi_{w'}K\cap \Psi_{\iota\cdot i''\cdot \tau''}K)=k^{-n}$.
	
	It remains to consider the case that $\#B=\infty$, i.e. $B$ is a line segment or a rectangle.
	Let
	\[
	\begin{cases}
	N_\tau=\big\{\tau''\in \partial W_{n-n'}:\#(\Psi_{\tau''}K\cap \Psi_{\tau}K)=\infty\big\},\\
	N_{\tau'}=\big\{\tau''\in \partial W_{n-n'}:\#(\Psi_{\tau''}K\cap \Psi_{\tau'}K)=\infty\big\},
	\end{cases}
	\]
	and $A_w=\bigcup_{\tau''\in N_\tau}\Psi_{\iota\cdot i\cdot \tau''}K$, $A_{w'}=\bigcup_{\tau''\in N_{\tau'}}\Psi_{\iota\cdot i'\cdot \tau''}K$. We can see that
	\[
	\begin{cases}
		A_w\cap B=\big\{y\in \mathbb{R}^3:\sup_{o=1}^3|x_o-y_o|\leq k^{-n}\text{ for some }x\in \Psi_w\square\big\}\cap B,\\
		A_{w'}\cap B=\big\{y\in \mathbb{R}^3:\sup_{o=1}^3|x_o-y_o|\leq k^{-n}\text{ for some }x\in \Psi_{w'}\square\big\}\cap B.
	\end{cases}	
	\]
	Hence, by fixing a point $x\in \Psi_wK\cap \Psi_{w'}K$, we have
	\[A_w\cap A_{w'}\supset B\cap \big\{y\in \mathbb{R}^3:\sup_{o=1}^3|x_o-y_o|\leq k^{-n}\big\}.\]
    Then, if $B$ is a line segment, one can see that $\mathcal{H}^1(A_w\cap A_{w'})\geq \min\{C'k^{-n'},k^{-n}\}$, so we can find $\tau^{(1)}\in N_{\tau},\tau^{(2)}\in N_{\tau'}$ such that $\mathcal{H}^1\big(\Psi_{\iota\cdot i\cdot\tau^{(1)}}K\cap \Psi_{\iota\cdot i'\cdot\tau^{(2)}}K\big)\geq \frac {1}{3\cdot 3}\min\{C'k^{-n'}, k^{-n}\}$; if $B$ is a rectangle, one can see that $\mathcal{H}^2(A_w\cap A_{w'})\geq \min\{C'^2k^{-2n'},C'k^{-n'-n},k^{-2n}\}$, so we can find $\tau^{(1)}\in N_{\tau},\tau^{(2)}\in N_{\tau'}$ such that $\mathcal{H}^2\big(\Psi_{\iota\cdot i\cdot\tau^{(1)}}K\cap \Psi_{\iota\cdot i'\cdot\tau^{(2)}}K\big)\geq \frac{1}{9\cdot 9}\min\{C'^2k^{-2n'},C'k^{-n'-n},k^{-2n}\}$. Hence, in both cases, we can find a desired sequence.
\end{proof}

Now for $n\geq 0$, we are ready to define the cell graph associated with $W_n$.

\begin{definition}\label{def210}
Given two distinct words $w,w'\in W_n$, we say $w\sn w'$ if one of the following cases happens.\vspace{0.2cm}

\emph{Case 1. $\{w,w'\}\subset (\partial W_1)^n$, and there are $o\in \{1,2,3\}$ and $s\in \{0,1\}$ such that $\Psi_wF_{o,s}=\Psi_{w'}F_{o,1-s}$.}\vspace{0.2cm}

\emph{Case 2. $\{w,w'\}\not\subset (\partial W_1)^n$, and $\Psi_w\square\cap \Psi_{w'} \square\neq \emptyset$ such that the following holds:	\\\noindent (a.1).  if $\Psi_{w}K\cap \Psi_{w'}K$ is a line segment, the length of $\Psi_{w}K\cap \Psi_{w'}K$ is larger than $Ck^{-n}$;\\
\noindent (a.2). if $\Psi_{w}K\cap \Psi_{w'}K$ is a rectangle, the area of $\Psi_{w}K\cap \Psi_{w'}K$ is larger than $Ck^{-2n}$,\\
where $C$ is the same constant in Lemma \ref{lemma28}.}
\vspace{0.2cm}

\noindent We use the notation $E_n:=\{(w,w')\in W_n^2: w\sn w'\}$ to denote the set of \textit{edges} induced by the relation $\sn$, and call $(W_n,E_n)$ the \textit{level-$n$ cell graph} associated with $K$. For a set $A\subset W_n$, say $A$ is  \textit{connected}  if the graph $(A, E_n|_{A\times A})$ is connected. For connected sets $A,B\subset W_n$, we write $A\sn B$ provided that there exist $w\in A$, $w'\in B$ so that $w\sn w'$, and write $w\sn B$ if $A=\{w\}$. In addition, for a finite collection of distinct words $w^{(0)}, \cdots w^{(L)}$ in $W_n$ with $L\geq 0$ satisfying that $w^{(i-1)}\sn w^{(i)}$, $\forall 0<i\leq L$, we call it a \textit{path} from $w^{(0)}$ to $w^{(L)}$, and say the \textit{length} of this path is $L+1$.
\end{definition}

\begin{remark}\label{remark210} By the face included condition, cells represented by $(\partial W_1)^n$ are located on $1/k^n$-grids in $\mathbb{R}^3$. According to the graph structure defined above, neighboring words in $(\partial W_1)^n$ only represent those face-to-face cells, which automatically enjoy nice local reflection symmetry. This enable us to easily deal with the reflection between borders of two \textit{blocks} $w\cdot W_m,w'\cdot W_m$ for $w\sn w'\in (\partial W_1)^n$ and $m\geq 0$.
\end{remark}

\begin{remark}\label{remark212}
For $w, w'\in (\partial W_1)^n$ satisfying $\Psi_wK\cap \Psi_{w'}K\neq\emptyset$, there always exists a path from $w$ to $w'$ with length at most $4$.
\end{remark}

Finally, before ending  this section, we point out that a $\USC^{(3)}$ $K$ always has the following geometric properties (\textbf{A1})-(\textbf{A4}) (we use the same notations in \cite{KZ}, also  \cite{CQ1}).

\vspace{0.2cm}

\noindent(\textbf{A1}). \emph{The open set condition: there exists a non-empty open set $U\subset \mathbb{R}^3$ such that $\bigcup_{i=1}^N \Psi_iU\subset U$ and $\Psi_iU\cap \Psi_jU=\emptyset$, $\forall i\neq j\in \{1,\cdots,N\}$. }

\noindent(\textbf{A2}). \emph{$(W_n,E_n)$ is connected for each $n\geq 1$.}

\noindent(\textbf{A3}). \emph{There is a constant $0<c_*<1$ and $L_*\in \mathbb{N}$ such that for any $n\geq 1$, }

\emph{(a).  if $x,y\in K$ and $d(x,y)<c_*k^{-n}$,  there exists a path $w^{(0)}, \cdots, w^{(L)}$ in  $W_n$  such that $x\in \Psi_{w^{(0)}}K, y\in\Psi_{w^{(L)}}K$ and $L\leq L_*$};

\emph{(b).  if $w,w'\in W_n$ and there is no path from $w$ to $w'$ with length at most $L_*+1$, then $d(x,y)\geq c_*k^{-n}$ for any $x\in \Psi_wK$ and $y\in\Psi_{w'}K$.}

\noindent(\textbf{A4}). \emph{$\partial W_n\neq W_n$ for $n\geq 2$.}\vspace{0.2cm}


\noindent(\textbf{\emph{Hausdorff measure}}). By (\textbf{A1}), the Hausdorff dimension of $K$ is $d_H:=\frac{\log N}{\log k}$, the unique solution to the equation $\sum_{i=1}^N(\frac{1}{k})^\alpha=1$. Throughout the paper, we will always choose $\mu$ to be the normalized $d_H$-dimensional Hausdorff measure restricted on $K$. Equivalently, $\mu$ is the unique self-similar probability measure on $K$ such that $\mu=\frac{1}{N}\sum_{i=1}^N\mu\circ \Psi_i^{-1}$.

\begin{proof}[Proof of (\textbf{A1})-(\textbf{A4})] \hspace{0.2cm}

(\textbf{A1}) and (\textbf{A4}) are obvious. (\textbf{A2}) follows from Lemma \ref{lemma28} and Remark \ref{remark212}. It remains to verify (\textbf{A3}). \vspace{0.2cm}

Let
\[c_*=\frac 12\wedge \Big\{ k\cdot\min\big\{d(\Psi_iK,\Psi_jK):\Psi_iK\cap \Psi_jK=\emptyset, 1\leq i,j\leq N\big\}\Big\}.\]

(\textbf{A3})-(a). Let $x,y\in K$ with  $d(x,y)<c_*k^{-n}$ for some $n\geq 1$. Due to Lemma \ref{lemma28} (b) and Remark \ref{remark212}, to prove \textbf{(A3)}-(a), it suffices to prove that there exist $w,w',w''\in W_n$ satisfying $x\in\Psi_wK$, $y\in\Psi_{w'}K$ and $\Psi_{w}K\cap\Psi_{w''}K\neq \emptyset$, $\Psi_{w'}K\cap\Psi_{w''}K\neq \emptyset$.

Let \[m_0=1+\max\big\{m\geq 0: \text{ there is } \tau\in W_m \text{ such that } x,y\in \Psi_\tau K\big\}.\]

\textit{Case 1. $m_0>n$. }\vspace{0.2cm}
	
In this case, there is $w\in W_n$ such that $x,y\in \Psi_{w}K$.\vspace{0.2cm}

\textit{Case 2. $m_0\leq n$. }\vspace{0.2cm}
	
In this case, we show that there are $v\neq v'\in W_{m_0}$, $o\in\{1,2,3\}$ and $s\in\{0,1\}$ such that
$\Psi_{v}K\cap \Psi_{v'}K\neq \emptyset$,
and
$x\in \bigcup_{\tilde{v}\in \tilde F_{n-m_0,o,s}}\Psi_{v\cdot \tilde{v}}K$, $y\in \bigcup_{\tilde{v}\in \tilde F_{n-m_0,o,1-s}}\Psi_{v'\cdot \tilde{v}}K$.
This makes the geometry clear, and the desired result follows immediately. We prove the observation below.

First, by the definition of $m_0$, there is $\tau\in W_{m_0-1}$ such that $x\in \Psi_{\tau} K$, $y\in \Psi_{\tau}K$.  Next, noticing that $d(\Psi_\tau^{-1}x,\Psi_\tau^{-1}y)<c_* k^{-n+(m_0-1)}\leq c_*k^{-1}$, by the definition of $c_*$, we have $1\leq j\neq j'\leq N$ so that $\Psi_\tau^{-1}x\in \Psi_jK$, $\Psi_{\tau}^{-1}y\in \Psi_{j'}K$ and $\Psi_jK\cap \Psi_{j'}K\neq\emptyset$. We let $v=\tau\cdot j$ and $v'=\tau\cdot j'$. Finally, choose $o\in\{1,2,3\}$ and $s\in\{0,1\}$ such that
\[\Psi_vK\cap \Psi_{v'}K\subset \bigcup_{\tilde{v}\in \tilde F_{n-m_0,o,s}}\Psi_{v\cdot \tilde{v}}K,\quad \Psi_vK\cap \Psi_{v'}K\subset \bigcup_{\tilde{v}\in \tilde F_{n-m_0,o,1-s}}\Psi_{v'\cdot \tilde{v}}K.\]
Clearly, $x\in \bigcup_{\tilde{v}\in \tilde F_{n-m_0,o,s}}\Psi_{v\cdot \tilde{v}}K$ since otherwise $d(x,y)\geq d(x,\Psi_{v'}K)\geq k^{-n}$, which contradicts the fact $d(x,y)<c_*k^{-n}$. Similarly,   $y\in \bigcup_{\tilde{v}\in \tilde F_{n-m_0,o,1-s}}\Psi_{v'\cdot \tilde{v}}K$.

\vspace{0.2cm}

(\textbf{A3})-(b). We prove  by contradiction. Let $w, w'\in W_n$, and assume there is no path from $w$ to $w'$ with length at most $L_*+1$. If there are $x\in \Psi_wK,y\in \Psi_{w'}K$ so that $d(x,y)< c_*k^{-n}$, then we can find $x'\in \Psi_w(K\setminus \partial\square)$ and $y'\in \Psi_{w'}(K\setminus \partial\square)$ so that $d(x',y')<c_*k^{-n}$. In particular, $\Psi_wK$ and $\Psi_{w'}K$ are the only $n$-cells containing $x'$ and $y'$ respectively. By (\textbf{A3})-(a) with respect to $x'$ and $y'$, we know there is a path from $w$ to $w'$ with length less than $L_*+1$.  A contradiction.	
\end{proof}

\section{Cell graph energies and Poincar\'e constants}\label{sec3}
In this section, we introduce the cell graph energies and the Poincar\'e constants. The story in this section is similar to that in \cite{CQ1} for $\USC$, which is inspired from the celebrated work of Kusuoka-Zhou \cite{KZ}. Throughout this paper, for a set $A$, we always use $l(A)$ to denote the collection of real functions on $A$.\vspace{0.2cm}

\noindent(\textbf{\emph{Cell graph energies}}). For a connected set $A\subset W_n$, we define a symmetric bilinear form $\mcD_{n,A}$ on  $l(A)$ as,
\[\mcD_{n,A}(f,g)=\sum_{\substack{w,w'\in A\\w\sn w'}}\big(f(w)-f(w')\big)\big(g(w)-g(w')\big), \quad \forall f,g\in l(A),\]
and call $\mcD_{n,A}$ a \textit{graph energy form} on $A$.
In particular, we write $\mcD_{n,A}(f):=\mcD_{n,A}(f,f)$, and  $\mcD_{n}:=\mcD_{n,W_n}$ for short.\vspace{0.2cm}

\noindent(\textbf{\emph{Averages of functions}}). Let $\nu$ be a Borel measure on a topological space $A$, and let $f$ be a measurable function on $A$. We use the notation
\[\fint_A fd\nu=\frac{1}{\nu(A)}\int_A fd\nu\]
to denote the  \textit{average} of $f$ on $A$ with respect to $\nu$. \vspace{0.2cm}

\noindent(\textbf{\emph{Counting measures}}). We always let $\mu_n$ be the counting measure on $W_n$, i.e. $\mu_n=\sum_{w\in W_n}\delta_{w}$, where $\delta_{w}$ is the delta mass at $w$ ($\delta(A)=0$ if $w\notin A$; $\delta(A)=1$ if $w\in A$).  \vspace{0.2cm}

\begin{definition}[Poincar\'e constants \cite{KZ}]\label{def31}\

(a). For $n\geq 1$, define
	\[\lambda_n=\sup\big\{\sum_{w\in W_n}\big(f(w)-\fint_{W_n}fd\mu_n\big)^2:f\in l(W_n), \mcD_n(f)=1\big\}.\]
	
	
	(b). For $A,B\subset W_n$ with $A\cap B=\emptyset$,  define \[R_{n}(A,B)=\max\big\{\big(\mcD_n(f)\big)^{-1}:f\in l(W_n),f|_A=0,f|_B=1\big\},\] the \textit{effective resistance} between $A$ and $B$.
	
	Let $L_*$ be the constant defined in (\textbf{A3}). For $w\in W_*$, we write
	\[
	\mathcal{N}_w=\bigcup\big\{w'\in W_{|w|}:\text{there is a path }w^{(0)}=w,\cdots,w^{(L)}=w'\text{ with  }L\leq L_*\big\},
	\]
	call it the \textit{$L_*$-adapted neighborhood} of $w$ (see  \cite[Chapter 2]{ki5} for the concept ``adaptedness''), and we abbreviate $W_{|w|}\setminus \mathcal{N}_w$ to $\mathcal {N}_w^c$.
	
	For $m\geq 1$, define
	\[R_m=\inf\big\{R_{{|w|+m}}(w\cdot W_m,  \mathcal{N}_w^c\cdot W_m):w\in W_*\setminus \{\emptyset\}\big\}.\]
	
	(c). For $m\geq 1$ and $w\sn w'$,  define
	\[
	\begin{aligned}
	\sigma_{m}(w,w')=\sup\big\{N^m\big(\fint_{w\cdot W_m}fd\mu_{n+m}-\fint_{w'\cdot W_m}fd\mu_{n+m}\big)^2:&f\in l(W_{n+m}), \\
	&\mcD_{n+m, \{w,w'\}\cdot W_m}(f)=1\big\}.
    \end{aligned}
    \]
	
	For $m\geq 1$, define
	\[\sigma_m=\sup\big\{\sigma_{m}(w,w'):n\geq 1, w,w'\in W_n,w\sn w'\big\}.\]
\end{definition}\vspace{0.2cm}

\begin{remark} In \cite{KZ}, another class of Poincar\'e type constants $\{\lambda^D_n\}_{n\geq 1}$ was also introduced. Similar to \cite{CQ1}, we will not use them directly. However, the developments in Sections \ref{sec5},\ref{sec6} are partially inspired by the mean escaping time interpretation of these constants.
\end{remark}

The following Proposition \ref{prop32} follows from exactly a same proof of \cite[Theorem 2.1]{KZ} (also see \cite[Proposition 3.2]{CQ1} for a reproduced version for $\mathcal{USC}$).
\begin{proposition}[{\cite[Theorem 2.1]{KZ}}]\label{prop32}
	For a fixed $\USC^{(3)}$ $K$, there is a constant $C>0$ such that
	\begin{equation}\label{eqn31}
		C^{-1}\cdot \lambda_nN^{m}R_m\leq \lambda_{n+m}\leq C\cdot \lambda_n\sigma_m, \quad\forall n\geq 1,m\geq 1,
	\end{equation}
	and,
	\begin{equation}\label{eqn32}
		R_n\geq C\cdot(k^2N^{-1})^n,\quad \forall n\geq 1.
	\end{equation}
	In addition, all the constants $\lambda_n,R_n$ and $\sigma_n$, $n\geq 1$ are positive and finite.
\end{proposition}

\begin{remark}\label{remark33}
	By letting $m=1$ in \eqref{eqn31}, one immediately see that $\lambda_n\asymp\lambda_{n+1}$. Then by letting $n=1$ in \eqref{eqn31}, one can easily see that
	\begin{equation}\label{eqn33}
		k^{2n}\lesssim N^nR_n\lesssim \lambda_n\lesssim \sigma_n.
	\end{equation}
	In addition, we will soon see (Lemma \ref{lemma33}) that $\sigma_n\lesssim\lambda_n$. This together with (\ref{eqn33}) implies that $ \sigma_n\asymp \lambda_n\asymp\lambda_{n+1}\asymp\sigma_{n+1}$.
\end{remark}

\section{Averages of functions and Poincar\'e constants $\mathscr{R}_n$}\label{sec4}
Just like in \cite{CQ1}, the key ingredient of this paper is to prove a capacity estimate that $R_n\gtrsim N^{-n}\lambda_n$, $\forall n\geq 1$, which will be achieved through five sections, Sections \ref{sec4}-\ref{sec8}.

In this section, we focus on estimating the difference of averages of a function with respect to different measures. We will end this section with a lower bound estimate, $\mathscr{R}_n\gtrsim N^{-n}\lambda_n$, of a new class of Poincar\'e type constants $\mathscr{R}_n$, $n\geq 1$, which deals with the difference of  averages  on neighboring faces of a function. We will improve this estimate to arrive at a face-to-face resistance estimate with the aid of a random walk argument in Sections \ref{sec5}-\ref{sec7}.\vspace{0.2cm}

\noindent(\textbf{\emph{The constants $\mathscr{R}_n$}}). For $n\geq 1$, we define
\[
\mathscr{R}_n=\sup\big\{\frac{\big|\fint_{\tilde{F}_{n,1,0}}fd\mu_n -\fint_{\tilde{F}_{n,2,0}}fd\mu_n\big|^2}{\mcD_n(f)}\big\},
\]
where the supreme is taken over all non-constant functions in $l(W_n)$.
\vspace{0.2cm}

In view of the following lemma, roughly speaking, $\mathscr{R}_n$ provides the largest possible difference among the averages of $f$ on different faces, controlled by its graph energy.

\begin{lemma}\label{lemma40}
	For $n\geq 1$, we always have
	\[
	\frac{1}{4}\mathscr{R}_n\leq \big\{\frac{\big|\fint_{\tilde{F}_{n,1,0}}fd\mu_n -\fint_{\tilde{F}_{n,1,1}}fd\mu_n\big|^2}{\mcD_n(f)}\big\}\leq 4\mathscr{R}_n.
	\]
\end{lemma}

\begin{proof}
	Denote $\tilde {\mathscr R}_n=\sup\big\{ {\big |\fint_{\tilde{F}_{n,1,0}}fd\mu_n -\fint_{\tilde{F}_{n,1,1}}fd\mu_n\big |^2}/{\mcD_n(f)}\big\}$ for short. Pick a function $g\in l(W_n)$ attaining the supremum in the definition of $\mathscr R_n$.
	Let $\Gamma\in \mathscr{G}$ be the reflection symmetry that maps $F_{2,0}$ to $F_{1,1}$, and write $\tilde\Gamma: W_n\to W_n$ the induced map satisfying $\Psi_{\tilde {\Gamma}(w)}K=\Gamma(\Psi_w K)$, $\forall w\in W_n$. Define
	\[g'(w)=
	\begin{cases}
		g(w),&\text{ if }d(\Psi_wK,F_{2,0})\leq d(\Psi_wK,F_{1,1}),\\
		g\circ \tilde{\Gamma}(w),&\text{ if }	d(\Psi_wK,F_{2,0})>d(\Psi_wK,F_{1,1}).
	\end{cases}\]
	Then,
	\[
	\begin{cases}
		\big|\fint_{\tilde{F}_{n,1,0}}g'd\mu_{n}-\fint_{\tilde{F}_{n,1,1}}g'd\mu_{n}\big|=\big|\fint_{\tilde{F}_{n,1,0}}gd\mu_{n}-\fint_{\tilde{F}_{n,2,0}}gd\mu_{n}\big|,\\
		\mcD_{n}(g')\leq 4\mcD_{n}(g),
	\end{cases}
	\]
	which gives $\tilde{\mathscr R}_n\geq \frac 1 4 \mathscr R_n$.

	On the other hand, for any function $f\in l(W_n)$, since
	$$\big|\fint_{\tilde{F}_{n,1,0}}fd\mu_n -\fint_{\tilde{F}_{n,1,1}}fd\mu_n\big|\leq \big|\fint_{\tilde{F}_{n,1,0}}fd\mu_n -\fint_{\tilde{F}_{n,2,0}}fd\mu_n\big|+\big|\fint_{\tilde{F}_{n,1,1}}fd\mu_n -\fint_{\tilde{F}_{n,2,0}}fd\mu_n\big|,$$
	we immediately have $\tilde{\mathscr{R}}_n\leq 4 {\mathscr R}_n$.
\end{proof}

\subsection{Preliminary estimates}
The following Proposition \ref{prop41} can be viewed as a special case of \cite[Lemma 3.9]{KZ}. Concrete examples of measures are provided in Example \ref{example42}, with a new application inspired by \cite{M2}.

\begin{proposition}\label{prop41}
	Let $\nu$ be a measure on $W_n$ such that
	\begin{equation}\label{eqn41}
		\frac{\max_{w\in W_m}\nu(w\cdot W_{n-m})}{\nu(W_n)}\leq C_1\cdot k^{-m},\qquad \forall 1\leq m\leq n, w\in W_m,
	\end{equation}
	for some $C_1>0$. Then, there exists $C_2>0$ independent of $C_1,n$ and $\nu$ such that
	\[\big|\fint_{W_n} fd\nu-\fint_{W_n} fd\mu_n\big|\leq C_2\sqrt{C_1\cdot N^{-n}\lambda_n\cdot\mcD_n(f)},\quad\forall f\in l(W_n).\]
	
\end{proposition}
\begin{proof}
	Without loss of generality, we assume that $\nu$ is a probability measure. Then \eqref{eqn41} simply means $\max_{w\in W_m}\nu(w\cdot W_{n-m})\leq C_1\cdot k^{-m}$ for each $1\leq m\leq n$.
	
	For each $0\leq m\leq n$, we define $f_m\in l(W_n)$ as
	\[f_m(w\cdot w')=\fint_{w\cdot W_{n-m}}fd\mu_n,\qquad\forall w\in W_m,w'\in W_{n-m}.\]
	In other words, $f_m$ can be viewed as the level-$m$ piecewise average of $f$ with respect to $\mu_n$, and in particular, $f_0$ is taking the average value $\fint_{W_n} fd\mu_n$ everywhere in $W_n$.
	
	By using the definition of $\lambda_{n-m+1}$ locally on each block $w\cdot W_{n-m+1}$, and by using the Cauchy-Schwarz inequality, for each $1\leq m\leq n$, one can see
	\[
	\begin{aligned}
		&\sum_{w\in W_{m-1}}\sum_{i\in W_1}\big|\fint_{w\cdot W_{n-m+1}}fd\mu_n-\fint_{w\cdot i\cdot W_{n-m}}fd\mu_n\big|^2\\
		\leq&\sum_{w\in W_{m-1}}\lambda_{n-m+1}N^{-n+m}\mcD_{n,w\cdot W_{n-m+1}}(f)\\
		\leq &\lambda_{n-m+1}N^{-n+m}\cdot\mcD_n(f).
	\end{aligned}
	\]
	Then by the definition of $f_m$, we can see that
	\[
	\begin{aligned}
		\|f_m-f_{m-1}\|^2_{l^2(W_n,\nu)}&\leq  \sum_{w\in W_{m-1}}\sum_{i\in W_1}\max_{v\in W_m}\nu(v\cdot W_{n-m})\cdot\big|\fint_{w\cdot W_{n-m+1}}fd\mu_n-\fint_{w\cdot i\cdot W_{n-m}}fd\mu_n\big|^2\\
		&\leq C_1\cdot N^{-n+m}\lambda_{n-m+1}k^{-m}\cdot\mcD_n(f).
	\end{aligned}
	\]
	Hence, noticing that by (\ref{eqn31}) and (\ref{eqn32}), $N^{-n+m}\lambda_{n-m+1}\leq C_3 (k^2N^{-1})^{-m}N^{-n}\lambda_n$ for some $C_3>0$ independent of $C_1, n$ and $\nu$, it follows that
	\[\|f_m-f_{m-1}\|^2_{l^2(W_n,\nu)}\leq C_3C_1\cdot(k^{-3}N)^m\cdot N^{-n}\lambda_n\cdot\mcD_n(f).\]
	Using the Minkowski inequality, and noticing that $f=f_n$, we finally see
	\[\|f-f_0\|_{l^1(W_n,\nu)}\leq\|f-f_0\|_{l^2(W_n,\nu)}\leq \sum_{m=1}^n\|f_m-f_{m-1}\|_{l^2(W_n,\nu)}\leq C_2\sqrt{C_1\cdot N^{-n}\lambda_n\cdot \mcD_n(f)},\]
	for some $C_2>0$ independent of $C_1, n$ and $\nu$, which is the desired estimate.
\end{proof}

Here we list some measures on $W_n$ that we will use. \vspace{0.2cm}

\noindent\textbf{Notations.} Let $1\leq m<n$ and $w\sm w'\in W_m$.  We write
\[
\mathcal{J}_n(w,w')=\{\eta\in W_{n-m}:w\cdot \eta\sn w'\cdot W_{n-m}\}\subset W_{n-m},
\]
and let
\[\mathcal{I}_n(w,w')=w\cdot \mathcal{J}_n(w,w')\subset W_n.\]
Here, we need to highlight that the definitions of  $\mathcal{J}_n(w,w')$ and $\mathcal{I}_n(w,w')$ depend on the order of $w,w'$.

\begin{example}\label{example42}\
	
	(a). Let $w,w'\in W_m$ such that $w\sm w'$, and let $\nu$ be the counting measure on $\mathcal{I}_n(w,w')$. More precisely, $\nu$ is the restriction of $\mu_n$ on $\mathcal{I}_n(w,w')$, i.e.
	\[\fint_{W_n}fd\nu=\fint_{\mathcal{I}_n(w,w')}fd\mu_n,\quad \forall f\in l(W_n).\]
	It is easy to see that $\nu$ satisfies (\ref{eqn41}) with some constant $C_1>0$ depending only on $K$.
	
	(b). Let $I_n=\{w\in W_n: \Psi_wK\cap H_{1,1/2}\neq \emptyset\}$. Let $B$ be a connected subset of $W_n$ such that
	\[B\cap \tilde{F}_{n,1,1}\neq \emptyset,\quad B\cap I_n\neq \emptyset.\]
	For each $i\in\{1,\cdots,\lfloor k^n/2\rfloor\}$, we choose $w^{(i)}\in B$ such that $\Psi_{w^{(i)}}H_{1,0}= H_{1,l}$ for some $\frac{k^n-i}{k^n}\leq l<\frac{k^n-i+1}{k^n}$. Let $\nu$ be a measure supported on $B$ defined as
		\[\nu=\sum_{i=1}^{\lfloor k^n/2\rfloor}\delta_{w^{(i)}}.\]
		Then one can check $\nu$ satisfies (\ref{eqn41}) with $C_1\leq 4$.
\end{example}

Example \ref{example42} (b) will be used in Section \ref{sec7}, and this is the only place that the setting $K\subset\mathbb{R}^3$ plays an essential role among Sections \ref{sec2}-\ref{sec7}.

	Finally, we end this subsection with an easy observation.
	
	\begin{corollary}\label{coro44}
		Let $n>m\geq 0$,  $A\subset W_m$. Let $\nu$ be a measure on $A\cdot W_{n-m}$ satisfying that
		\[
		\nu(w\cdot W_{n-m})=\frac{1}{\#A}\nu(A\cdot W_{n-m}), \quad \forall w\in A,
		\]
		and
		\[
		\frac{\max_{w'\in W_i}\nu(w\cdot w'\cdot W_{n-m-i})}{\nu(w\cdot W_{n-m})}\leq C_1\cdot k^{-i}, \quad\forall 1\leq i\leq n-m, \forall w\in A,
		\]
		for some constant $C_1>0$. Then, for each $f\in l(A\cdot W_{n-m})$, we have
		\[
		\big|\fint_{A\cdot W_{n-m}} fd\nu-\fint_{A\cdot W_{n-m}} fd\mu_n\big|\leq C_2(\frac{N^mk^{-2m}}{\#A})^{1/2}\sqrt{C_1N^{-n}\lambda_n\mcD_{n,A\cdot W_{n-m}}(f)}
		\]
		for some constant $C_2>0$ depending only on $K$.
	\end{corollary}
	\begin{proof}
		By applying Proposition \ref{prop41} locally, and using Proposition \ref{prop32}, we see
		\[
		\begin{aligned}
			\sum_{w\in A}\big|\fint_{w\cdot W_{n-m}} fd\nu-\fint_{w\cdot W_{n-m}} fd\mu_n\big|^2&\leq C_1C_3^2\sum_{w\in A} N^{-n+m}\lambda_{n-m}\mcD_{n,w\cdot W_{n-m}}(f)\\
			&\leq C_1C_3^2C_4\cdot N^{m}k^{-2m}\cdot N^{-n}\lambda_n\cdot\mcD_{n,A\cdot W_{n-m}}(f),
		\end{aligned}
		\]
		for some constants $C_3,C_4>0$ depending only on $K$. The corollary follows from the Cauchy-Schwarz inequality immediately.
\end{proof}

\subsection{An estimate of $\mathscr{R}_n$}
In this subsection, we aim to show that $\mathscr{R}_n$ is bounded below by $N^{-n}\lambda_n$ up to a constant multiplier.   \vspace{0.2cm}

\begin{lemma}\label{lemma44}
	Let $1\leq m<n$ and $w\sm w'\in W_m$. Then for any $f\in l(\{w,w'\}\cdot W_{n-m})$,
	\[
	\big|\fint_{\mathcal I_n(w,w')} fd\mu_n-\fint_{\mathcal I_n(w',w)} fd\mu_n\big|\leq C\cdot\sqrt{k^{-(n-m)}\mcD_{n,\{w,w'\}\cdot W_{n-m}}(f)},
\]
	\[
	\big|\fint_{\mathcal I_n(w,w')} fd\mu_n-\fint_{\mathcal I_n(w',w)} fd\mu_n\big|\leq C\cdot\sqrt{N^nk^{{-3n+m}}}\cdot\sqrt{N^{-n}\lambda_n\mcD_{n,\{w,w'\}\cdot W_{n-m}}(f)},
	\]
	for some $C>0$ depending only on $K$.
\end{lemma}
\begin{proof}
	Since by \eqref{eqn33}, $\lambda_{n} \gtrsim k^{2n}$, the second inequality is an immediate consequence of the first one. So we only need to prove there is $C_1 > 0$ so that
	\[
	\big(\sum_{v\in \mathcal{I}_n(w,w')}f(v) - \sum_{v\in \mathcal{I}_n(w',w)}f(v)\big)^2 \leq C_1\big(\#\mathcal{I}_n(w,w')\big)^2k^{-n + m}\mathcal{D}_{n,\{w,w'\}\cdot W_{n - m}}(f).
	\]
	
	Clearly there is a bijection $\varphi: \mathcal{I}_n(w,w')\to \mathcal{I}_n(w',w)$ such that $v\sn \varphi(v)$ for any $v\in \mathcal{I}_n(w,w')$. Thus we deduce
	\[\begin{aligned}\mathcal{D}_{n,\{w,w'\}\cdot W_{n - m}}(f)
		\geq & \sum_{v\in \mathcal I_n(w,w')}\big(f(v) - f(\varphi(v))\big)^2 \\
		\geq & (\#\mathcal I_{n}(w,w'))^{-1}\big(\sum_{v\in \mathcal{I}_n(w,w')}f(v) - \sum_{v\in \mathcal{I}_n(w',w)}f(v)\big)^2.
	\end{aligned}\]
	Combining the above estimate with the fact that $\# \mathcal{I}_{n}(w,w') \gtrsim k^{n - m}$, the desired estimate follows.
\end{proof}

In particular, Lemma \ref{lemma44} tells us that when $n$ is large enough, and $m$ is bounded, the left sides of the inequalities  can be ignored compared with $\sqrt{N^{-n}\lambda_n\mcD_{n,\{w,w'\}\cdot W_{n-m}}(f)}$.\vspace{0.2cm}

Next, we apply Example \ref{example42} (a) and Lemma \ref{lemma44} to show the following Lemma \ref{lemma33}, which was mentioned at the end of Section \ref{sec3}.

\begin{lemma}\label{lemma33}
	Let $K$ be a $\USC^{(3)}$. Then there exists $C>0$ such that for any $n\geq 1$,
	\begin{align*}
		\sigma_n\leq C\cdot \lambda_n.
	\end{align*}
\end{lemma}
\begin{proof}
Let $w\stackrel{n'}{\sim} w'$ and $f\in l(\{w,w'\}\cdot W_n)$. We see that
\[\begin{aligned}
&\big|\fint_{w\cdot W_n}fd\mu_{n'+n}-\fint_{w'\cdot W_n}fd\mu_{n'+n}\big|\\
\leq &\big|\fint_{\mathcal I_{n'+n}(w,w')}fd\mu_{n'+n}-\fint_{\mathcal I_{n'+n}(w',w)}fd\mu_{n'+n}\big|\\ &+\big|\fint_{w\cdot W_n}fd\mu_{n'+n}-\fint_{\mathcal I_{n'+n}(w,w')}fd\mu_{n'+n}\big|+\big|\fint_{w'\cdot W_n}fd\mu_{n'+n}-\fint_{\mathcal I_{n'+n}(w',w)}fd\mu_{n'+n}\big|\\
\leq & C_1\cdot \big(\sqrt{N^{-n}\lambda_n}+\sqrt{ k^{-n}}\big)\cdot \sqrt{\mcD_{n'+n,\{w,w'\}\cdot W_n}(f)}\\
\leq & C_2\cdot \sqrt{N^{-n}\lambda_n\mcD_{n'+n,\{w,w'\}\cdot W_n}(f)},
\end{aligned}
\]
for some constants $C_1, C_2>0$ depending only on $K$, where in the second inequality we use Proposition \ref{prop41}, Example \ref{example42} (a) and Lemma \ref{lemma44}, and in the third inequality we use the fact that $ k^{-n}\leq N^{-n} k^{2n}\lesssim N^{-n}\lambda_n$ by (\ref{eqn33}).
The lemma follows immediately.
\end{proof}

\begin{proposition}\label{prop45}
	$\mathscr{R}_n\geq C\cdot N^{-n}\lambda_n$ for any $n\geq 1$, where $C>0$ is a constant depending only on the $\USC^{(3)}$ $K$.
\end{proposition}
\begin{proof}
	We assume $n$ is large enough, and will show that $\mathscr{R}_{n-M_1-M_2}\gtrsim N^{-n}\sigma_n$ for some $M_1,M_2\in\mathbb{N}$ independent of $n$. The proposition will then follow from (\ref{eqn33}) immediately. 
	
	We provide the proof through $4$ steps.\vspace{0.2cm}
	
	\noindent\textit{Step 1}. We pick $n'\in \mathbb{N}$ and $\eta,\eta'\in W_{n'}$ so that $\eta\stackrel{n'}{\sim}\eta'$ and $\sigma_n(\eta,\eta')=\sigma_n$. By the definition of $\sigma_n(\eta,\eta')$, we could choose  $f'\in l(\{\eta,\eta'\}\cdot W_n)$ such that
	\[
	\fint_{\eta\cdot W_n}f'd\mu_{n'+n}=1,\  \fint_{\eta'\cdot W_n}f'd\mu_{n'+n}=-1\text{  and }\mcD_{n+n',\{\eta,\eta'\}\cdot W_n}(f')=4N^n\sigma^{-1}_n.
	\]
	Furthermore,  let $\Gamma_{\eta,\eta'}$ be the self-isometry that acts as central symmetry on $\Psi_\eta K\cup\Psi_{\eta'} K$, i.e. $\Gamma_{\eta,\eta'}(\Psi_\eta K)=\Psi_{\eta'}K$ and $\Gamma_{\eta,\eta'}(\Psi_{\eta'} K)=\Psi_{\eta}K$, and let $\tilde\Gamma_{\eta,\eta'}$ be the induced symmetry on $\{\eta,\eta'\}\cdot W_n$ so that $\Psi_{\tilde\Gamma_{\eta,\eta'}(v)}K=\Gamma_{\eta,\eta'}(\Psi_vK)$. Then we can assume that $f'$ is anti-symmetric with respect to $\tilde\Gamma_{\eta,\eta'}$ (i.e. $f(\tilde \Gamma_{\eta,\eta'}(v))=-f(v)$) noticing that $\tilde f':=\frac{f'-f'\circ\tilde\Gamma_{\eta,\eta'}}{2}$ will decrease the energy of $f'$ but still keep $\fint_{\eta\cdot W_n}\tilde f'd\mu_{n'+n}=1,\  \fint_{\eta'\cdot W_n}\tilde f'd\mu_{n'+n}=-1$. We claim that there is $C_1\geq 1/2$ depending only on $K$ such that $f''=(f'\wedge C_1)\vee (-C_1)$ satisfying
	\[
	\fint_{\eta\cdot W_n}f''d\mu_{n'+n}\geq \frac12,\quad  \fint_{\eta'\cdot W_n}f''d\mu_{n'+n}\leq -\frac12.
	\]
	We then let $f\in l(W_n)$ be defined as
	\[f(w)=f''(\eta\cdot w),\qquad\forall w\in W_n.\]

	By the anti-symmetry of $f'$, the only thing that we need to prove in this step is the existence of $C_1$ such that $\fint_{W_n}(g\wedge C_1)\vee (-C_1)d\mu_{n}\geq 1/2$, where $g\in l(W_n)$ is defined as $g(w)=f'(\eta\cdot w),\forall w\in W_n$. In fact, one can see
	\[
	\begin{cases}
		\fint_{W_n}gd\mu_n=1,\\
		\fint_{W_n}(g-1)^2d\mu_n\leq N^{-n}\lambda_n\mcD_{n}(g)\leq C_2,
	\end{cases}
	\]
	for some constant  $C_2>0$ depending only on $K$, where we use the fact that $\mcD_n(g)\leq 2N^n\sigma^{-1}_n$ and $\lambda_n\lesssim\sigma_n$ by \eqref{eqn33} in the last inequality. Whence, there is $C_1$ depending only on $C_2$ such that $\fint_{W_n}|g|\cdot 1_{\{|g|>C_1\}}d\mu_n\leq\frac{1}{2}$, so $f=(g\wedge C_1)\vee (-C_1)$ satisfies the desired property. \vspace{0.2cm}

	For  convenience, we list the properties of $f$ that we will use:
	
	1). $\mcD_n(f)\leq 2N^{n}\sigma_n^{-1}$;
	
	2). $\fint_{W_n}fd\mu_n\geq 1/2$;
	
        3). $\fint_{\mathcal{J}_{n'+n}(\eta,\eta')} fd\mu_n\leq C_3\cdot (Nk^{-3})^{n/2}$ for some $C_3>0$ depending only on $K$.\vspace{0.2cm}
	
	Here (1), (2) are obvious, and  (3) is a consequence of the anti-symmetry of $f''$, the energy estimate of $f''$, and Lemma \ref{lemma44}. Indeed, one can check that
	\[
	\fint_{\mathcal{J}_{n'+n}(\eta,\eta')} fd\mu_n\leq \sqrt{C_4k^{-n}\mcD_{n'+n,\{\eta,\eta'\}\cdot W_n}(f'')}\leq \sqrt{C_4k^{-n}\cdot 4N^n\sigma_n^{-1}},
	\]
	for some constants $C_4>0$ depending only on $K$, and then (3) follows from \eqref{eqn33}.
 As pointed out after Lemma \ref{lemma44}, the right side of (3) is rather small as long as $n$ is large enough. \vspace{0.2cm}

	\noindent\textit{Step 2.  There exists $M_1\in\mathbb{N}$  (large but independent of large $n$), so that the following hold:}
	
	\textit{(2.a). There are $\alpha,\alpha'\in W_{M_1}$ such that $\alpha\stackrel{M_1}{\sim}\alpha'$ and} \[\fint_{\mathcal{I}_n(\alpha,\alpha')}fd\mu_n>\frac{1}{4}.\]
	
	\textit{(2.b). There is $\beta\in W_{M_1}$, $o\in \{1,2,3\}$ and $s\in \{0,1\}$ such that }
	\[
	\beta\cdot \tilde F_{n-M_1,o,s}\cap \mathcal{J}_{n'+n}(\eta,\eta')\neq \emptyset,\qquad \fint_{\beta\cdot \tilde{F}_{n-M_1,o,s}} fd\mu_n<\frac 18.
	\]
	
	(2.a) can be shown by contradiction. We assume that for any large $M_{1,0}$, such a pair doesn't exist, then for each $\tilde{\alpha}\in W_{M_{1,0}}$, there is $\tilde{\alpha}'\stackrel{M_{1,0}}{\sim} \tilde{\alpha}$ such that
	\[
		\fint_{\mathcal{I}_n(\tilde{\alpha},\tilde{\alpha}')}fd\mu_n\leq \frac{1}{4}.
	\]
	Define a measure $\nu$ on $W_n$ by $\nu(B)=\sum_{\tilde{\alpha}\in W_{M_{1,0}}}\frac{\mu_n(B\cap \mathcal{I}_n(\tilde{\alpha},\tilde{\alpha}'))}{\mu_n(\mathcal{I}_n(\tilde{\alpha},\tilde{\alpha}'))}$, then clearly $\fint_{W_n}fd\nu\leq \frac14$. In addition, by Corollary \ref{coro44} (see also Example \ref{example42}), we have that for some constant $C_5>0$,
		\[
		\big|\fint_{W_n}fd\mu_n-\fint_{W_n}fd\nu\big|\leq \sqrt{C_5k^{-2M_{1,0}}}\cdot \sqrt{N^{-n}\lambda_n\mcD_{n}(f)}.
		\]
		Noticing that due to Step 1-(1) and (\ref{eqn33}), $\mcD_{n}(f)\leq C_6N^n\lambda_n^{-1}$ for some $C_6>0$, by $\fint_{W_n}fd\nu\leq \frac14$, when $M_{1,0}$ is large enough, we must have $\fint_{W_n}f\mu_n<\frac12$, a contradiction to Step 1-(2).\vspace{0.2cm}
	
	(2.b) also follows from the properties of $f$. We need to consider two possible cases: $K_\eta\cap K_{\eta'}$ is a line segment or a rectangle.\vspace{0.2cm}

	If $K_\eta\cap K_{\eta'}$ is a rectangle, for large $M_{1,1}$, we let
	\[
	A_{M_{1,1}}=\bigcup\big\{\beta\cdot \tilde F_{n-M_{1,1},o,s}: \beta\in W_{M_{1,1}}, o\in \{1,2,3\},s\in \{0,1\}\text{ and }\beta\cdot \tilde F_{n-M_{1,1},o,s}\subset \mathcal{J}_{n+n'}(\eta,\eta')\big\}.
	\]
	Then, for $M_{1,1}$ (and $n$) large enough, $\fint_{A_{M_{1,1}}}fd\mu_n<\frac{1}{8}$, where we use the fact that $f$ is bounded below by $-C_1$, $\fint_{\mathcal{J}_{n+n'}(\eta,\eta')}fd\mu_n$ can be arbitrarily small by Step 1-(3), and the fact that $\frac{\mu_n(\mathcal{J}_{n'+n}(\eta,\eta')\setminus A_{M_{1,1}})}{\mu_n(\mathcal{J}_{n'+n}(\eta,\eta'))}\to 0$ as $M_{1,1}\to \infty$ with convergence rate depending only on $K$. Then (2.b) follows immediately.\vspace{0.2cm}
	
	 If $K_\eta\cap K_{\eta'}$ is a line segment,  for large $M_{1,1}$, we denote $\{\tilde E_{n-M_{1,1},i}\}_{i=1}^{12}$ for all  the possible intersections $\tilde F_{n-M_{1,1},o,s}\cap \tilde F_{n-M_{1,1},o',s'}$,
	$o\neq o'\in\{1,2,3\}$, $s,s'\in\{0,1\}$,  and let
	\[
	A_{M_{1,1}}=\bigcup\big\{\beta\cdot \tilde E_{n-M_{1,1},i}: \beta\in W_{M_{1,1}}, i\in \{1,\cdots,12\}\text{ and }\beta\cdot \tilde E_{n-M_{1,1},i}\subset \mathcal{J}_{n+n'}(\eta,\eta')\big\}.
	\]
	Then same as the previous case, for large enough $M_{1,1}$, we have $\fint_{A_{M_{1,1}}}fd\mu_n < \frac{1}{16}$. Then for large $M_{1,2} > M_{1,1}$, we denote
	\[A'_{M_{1,2}} = \bigcup\big\{\beta\cdot \tilde{F}_{n - M_{1,2},o,s}:\beta \in W_{M_{1,2}},o\in \{1,2,3\},s\in \{0,1\},\beta\cdot\tilde{F}_{n - M_{1,2},o,s}\cap A_{M_{1,1}}\neq \emptyset\big\}.
	\]
	Define measures $\nu_1,\nu_2$ on $W_n$ by $\nu_1(B) = \mu_n(B\cap A_{M_{1,1}})$ and $\nu_2(B) = \mu_n(B\cap A'_{M_{1,2}})$ for $B\subset W_n$. Then by Corollary \ref{coro44} and using (\ref{eqn33}), we have
	\[\big|\fint_{W_n}fd\nu_1 - \fint_{W_n}fd\nu_2\big| \leq  C_7(N^{M_{1,2}}k^{-3M_{1,2}})^{1/2}
	\]	
	for some constant $C_7>0$. Noticing that $\fint_{W_n}fd\nu_1 = \fint_{A_{M_{1,1}}}fd\mu_n < 1/16$ and $\fint_{W_n} fd\nu_2 = \fint_{A'_{M_{1,2}}}fd\mu_n$, {(2.b)} follows immediately by taking $M_{1,2}$ large enough.

	Finally, we simply choose $M_1$ large enough so that both (2.a) and (2.b) hold.\vspace{0.2cm}

	\noindent\textit{Step 3. There exist $\gamma\in W_{M_1}$ and subsets $A,B\subset \gamma\cdot W_{n-M_1}$ that both take the form $\gamma\cdot \tilde{F}_{n-M_1,o,s}$ with $o\in\{1,2,3\}$, $s\in\{0,1\}$, or $\mathcal{I}_n(\gamma,\gamma')$ with $\gamma'\stackrel{M_1}{\sim}\gamma$, satisfying
	\[
	\fint_Afd\mu_n-\fint_Bfd\mu_n>C_8,
	\]
	for some $C_8>0$ depending only on $K$. }\vspace{0.2cm}
	
	By condition (\textbf{A2}), we can find $L\leq N^{M_{1,1}}$ and a path $\alpha=\gamma^{(0)},\gamma^{(1)},\cdots,\gamma^{(L)}=\beta$, where $\alpha,\beta\in W_{M_1}$ are defined in Step 2. Then, we let
	\[
	\begin{cases}
		A_l=\mathcal{I}_{n}(\gamma^{(l)},\gamma^{(l-1)}),&\text{ if }0<l\leq L,\\
		B_l=\mathcal{I}_{n}(\gamma^{(l)},\gamma^{(l+1)}),&\text{ if }0\leq l<L,
	\end{cases}
	\]
	and $A_0=\mathcal{I}_n(\alpha,\alpha')$ as in (2.a), and $B_L=\beta\cdot \tilde F_{n-M_1,o,s}$ as in (2.b). By Step 2, we know that $\fint_{A_0}fd\mu_n-\fint_{B_L}fd\mu_n>\frac{1}{8}$; by Lemma \ref{lemma44}, we know that each $|\fint_{A_l}fd\mu_n-\fint_{B_{l-1}}fd\mu_n|$ can be ignored as long as $n$ is large. Hence, we can find $0\leq l\leq L$ such that $\gamma^{(l)}, A_l,B_l$ satisfy the requirement of the claim.\vspace{0.2cm}

	\noindent\textit{Step 4. There exist large enough $M_2\in\mathbb{N}$, small enough $C_{9}>0$ depending only on $K$, $\xi\in \gamma\cdot W_{M_2}$, $o_1,o_2\in \{1,2,3\}$ and $s_1,s_2\in \{0,1\}$ such that
	\[
	\big|\fint_{\xi\cdot \tilde{F}_{n-M_1-M_2,o_1,s_1}}fd\mu_n-\fint_{\xi\cdot \tilde{F}_{n-M_1-M_2,o_2,s_2}}fd\mu_n\big|>C_{9}.
	\]}
	
	By a same argument as (2.b) and using Corollary \ref{coro44} again, we can find $M_2$ large enough and $w,w'\in \gamma\cdot W_{M_2}$ such that there are $o,o'\in\{1,2,3\},s,s'\in\{0,1\}$ satisfying
	\[
	w\cdot \tilde{F}_{n-M_1-M_2,o,s}\cap A\neq \emptyset,\quad w'\cdot \tilde{F}_{n-M_1-M_2,o',s'}\cap B\neq \emptyset,
	\]
	where $A,B$ are the same in Step 3, and
	\[
	\big|\fint_{w\cdot \tilde{F}_{n-M_1-M_2,o,s}}fd\mu_n-\fint_{w'\cdot \tilde{F}_{n-M_1-M_2,o',s'}}fd\mu_n\big|>C_8/2.
	\]
	Then, Step 4 follows from a same argument as Step 3, by arranging a path connecting $w,w'$ along the faces of $\gamma\cdot W_{M_2}$.\vspace{0.2cm}
	
	Finally, we let $h\in l(W_{n-M_1-M_2})$ be defined as $h(w)=f(\xi\cdot w),\forall w\in W_{n-M_1-M_2}$. Then by Step 4 and Lemma \ref{lemma40}, we immediately have $$\mathscr R_{n-M_1-M_2}\geq \frac 14 C_{9}^2/\mcD_{n-M_1-M_2}(h)\geq  \frac 18 C_{9}^2 N^{-n}\sigma_n,$$
	and thus the proposition follows.
\end{proof}

\section{A pair of random walks}\label{sec5}
In this section, we introduce a pair of random walks that are properly coupled. A hitting time estimate will be given in Sections \ref{sec6} and \ref{sec7}. \vspace{0.2cm}

First, we introduce some notations about the underlying spaces. For $m\geq 0, n\geq 1$, we let
\[\tilde{L}_{m,n}:=\tilde{F}_{m,2,0}\cdot W_n.\]
Intuitively,  $\tilde{L}_{m,n}$ is like a thick wall, consisting of copies of $W_n$, perpendicular to $\tilde{F}_{m+n}:=\tilde{F}_{m+n,1,0}$. In particular, we can fold $\tilde{L}_{m,n}$ into $W_n$ in a natural way.

\begin{definition}[\textbf{a folding map}]\label{def51}
	Define $\Theta:\mathbb{R}\to [0,1]$ by $\Theta(t)=\min\big\{|t-2i|:i\in \mathbb{Z}\big\}$. Let $\Theta^{(3)}:\mathbb{R}^3\to [0,1]^3$ be defined as $\Theta^{(3)}(x_1,x_2,x_3)=\big(\Theta(x_1),\Theta(x_2),\Theta(x_3)\big), \forall x=(x_1,x_2,x_3)$. We then define a folding map $\tilde{\Theta}_{m,n}:\tilde{L}_{m,n}\to W_n$ as follows. For each $w'\cdot w\in \tilde{L}_{m,n}$ with $w'\in \tilde{F}_{m,2,0}$, $w\in W_n$, we let $\tilde{\Theta}_{m,n}(w'\cdot w)$ be the unique element in $W_n$ such that
	\[\Theta^{(3)}\big(k^m\Psi_{w'\cdot w}K\big)=\Psi_{\tilde{\Theta}_{m,n}(w'\cdot w)}K,\]
	where $k^m\Psi_{w'\cdot w}K=\{k^mx: x\in \Psi_{w'\cdot w}K\}$.
	
	In addition, for  a measure $\nu$ on $\tilde{L}_{m,n}$, we let $\tilde{\Theta}_{m,n}^*\nu$ be the image measure on $W_n$ of $\nu$ under $\tilde \Theta_{m,n}$, i.e. $\tilde{\Theta}_{m,n}^*\nu(B)=\nu\big(\tilde{\Theta}_{m,n}^{-1}(B)\big)$, $\forall B\subset W_n$.
\end{definition}

We assign suitable measures on $W_n$ and $\tilde{L}_{m,n}$ as follows.

\noindent(\textbf{\emph{A measure $\pi_n$ on $W_n$}}). We let
\[\deg(w)=\#\{v\in W_n:v\neq w,v\sn w\},\qquad\forall n\geq 1,w\in W_n.\]
For $n\geq 1$, define a measure $\pi_n$ on $W_n$ by
\[\pi_n(w):=\pi_n(\{w\})=
\begin{cases}
	\deg(w)+1, &\text{ if }w\in \partial W_n,\\
	\deg(w),   &\text{ if }w\in W_n\setminus \partial W_n.
\end{cases}
\]

\noindent(\textbf{\emph{A measure $\pi_{m,n}$ on $\tilde{L}_{m,n}$}}). For $m\geq 0,n\geq 1$, define a measure $\pi_{m,n}$ on $\tilde{L}_{m,n}$ by
\[
\pi_{m,n}(w'\cdot w):=\pi_{m,n}(\{w'\cdot w\})=\pi_n\big(\tilde\Theta_{m,n}(w'\cdot w)\big),\qquad\forall w'\in\tilde{F}_{m,2,0},\ w\in W_n.
\]\vspace{0.2cm}

Now, we introduce a random walk on $\tilde{L}_{m,n}$.

\noindent(\textbf{\emph{A random walk on $\tilde L_{m,n}$}}). For each pair $w'\cdot w, v'\cdot v$ with $w',v'\in \tilde{F}_{m,2,0}$ and $w,v\in W_n$, we define
\begin{equation}\label{eqn51}
\begin{aligned}
	&p^{(m,n)}_1(w'\cdot w,v'\cdot v)\\
	=&\begin{cases}
		(\deg(w'\cdot w)-\deg(w)+1)^{-1}\cdot\pi_n(w)^{-1},&\text{ if }w'\cdot w= v'\cdot v\text{ and } w=v\in\partial W_n,\\
		(\deg(w'\cdot w)-\deg(w)+1)^{-1}\cdot\pi_n(w)^{-1},&\text{ if }w'\cdot w\stackrel{m+n}{\sim}v'\cdot v\text{ and }w'\neq v',\\
		\pi_n(w)^{-1},&\text{ if }w'\cdot w\stackrel{m+n}{\sim}v'\cdot v\text{ and }w'=v',\\
		0,&\text{ otherwise}.
	\end{cases}
\end{aligned}
\end{equation}
Then $p^{(m,n)}_1(w'\cdot w,v'\cdot v)$ is a Markov kernel on $\tilde L_{m,n}$. We also take $p^{(m,n)}_0(w'\cdot w,v'\cdot v)=\delta_{w'\cdot w,v'\cdot v}$, where $\delta_{a,b}$ is the Kronecker delta, i.e. $\delta_{a,b}=1$ if $a=b$; $\delta_{a,b}=0$ if $a\neq b$.

\begin{remark} For the first two cases in (\ref{eqn51}), the term $\deg(w'\cdot w)-\deg(w)$ is the number of neighbors of $w'\cdot w$ in $\tilde L_{m,n}\setminus w'\cdot W_n$, which only takes values in $\{0,1,2\}$ by recalling Remark \ref{remark210}.
Sometimes, we also regard $p^{(m,n)}_1$ as an operator $p^{(m,n)}_1: l(\tilde{L}_{m,n})\to l(\tilde{L}_{m,n})$ by
\[
p_1^{(m,n)}f(w'\cdot w):=\sum_{v'\cdot v\in \tilde{L}_{m,n}}p_1^{(m,n)}(w'\cdot w,v'\cdot v)f(v'\cdot v),\quad \forall f\in l(\tilde L_{m,n}),  \forall w'\cdot w\in \tilde{L}_{m,n}.
\]
\end{remark}

Let $\big(\Omega,\mathbb{P}^{(m,n)}_{w'\cdot w}, X^{(m,n)}\big)$ be the associated \textit{random walk on $\tilde L_{m,n}$}, where $X^{(m,n)}$ is short for $\{X^{(m,n)}_i\}_{i\geq 0}$. Here $\Omega$ is the sample space, $\mathbb{P}^{(m,n)}_{w'\cdot w}$ is a probability measure for each $w'\cdot w\in \tilde{L}_{m,n}$ (the law of the random walk starting at $w'\cdot w$) and for $i\geq 0$, $X^{(m,n)}_i: \Omega\to \tilde L_{m,n}$ are random variables. Then for any $w'\cdot w,v'\cdot v\in \tilde{L}_{m,n}$ and $i\geq 0$, we have
\[
\begin{aligned}
	\mbPmn_{w'\cdot w}\big(X^{(m,n)}_i=v'\cdot v\big)&=p_i^{(m,n)}(w'\cdot w,v'\cdot v),
\end{aligned}
\]
where for $i\geq 2$, $p_i^{(m,n)}(w'\cdot w,v'\cdot v)$ is defined as
\[
p_i^{(m,n)}(w'\cdot w,v'\cdot v):=\sum_{\eta^{(1)},\eta^{(2)},\cdots,\eta^{(i-1)}\in \tilde{L}_{m,n}}p^{(m,n)}_1(w'\cdot w,\eta^{(1)})\cdots p^{(m,n)}_1(\eta^{(i-1)},v'\cdot v).
\]

\begin{lemma}\label{lemma52}
$X^{(m,n)}$ is reversible with respect to $\pi_{m,n}$, i.e.
\[
\pi_{m,n}(w'\cdot w)p^{(m,n)}_1(w'\cdot w,v'\cdot v)=\pi_{m,n}(v'\cdot v)p^{(m,n)}_1(v'\cdot v,w'\cdot w),\quad \forall w'\cdot w,v'\cdot v\in \tilde{L}_{m,n}.
\]
In addition,
we have the energy estimate
\[
\frac{1}{3}\mcD_{m+n,\tilde{L}_{m,n}}(f)\leq \langle f-p_1^{(m,n)}f,f\rangle_{l^2(\tilde L_{m,n},\pi_{m,n})}\leq \mcD_{m+n,\tilde{L}_{m,n}}(f),\quad \forall f\in l(\tilde L_{m,n}),
\]
where $\langle f,g\rangle_{l^2(\tilde L_{m,n},\pi_{m,n})}:=\sum_{w'\cdot w\in \tilde{L}_{m,n}}f(w'\cdot w)g(w'\cdot w)\pi_{m,n}(w'\cdot w)$.
\end{lemma}
\begin{proof}
The first claim is straightforward to verify. The second claim is due to the fact that for any $w'\cdot w,v'\cdot v\in \tilde{L}_{m,n}$,
$\pi_{m,n}(w'\cdot w)p^{(m,n)}_1(w'\cdot w,v'\cdot v)\in \{1,\frac 12,\frac 13\}$ provided $p^{(m,n)}_1(w'\cdot w,v'\cdot v)\neq 0$.
\end{proof}

Next, we consider a random walk on $W_n$.

\noindent(\textbf{\emph{Transition kernels on $W_n$}}). For $n\geq 1$, $w, v\in W_n$, define
\[p^{(n)}_1(w,v)=
\begin{cases}
	\pi_n(w)^{-1},&\text{ if }w\sn v,\\
	\pi_n(w)^{-1},&\text{ if }w=v\in \partial W_n,\\
	0,&\text{ otherwise}.
\end{cases}\]
Then $p^{(n)}_1(w,v)$ is a Markov kernel on $W_n$. Also, we take $p^{(n)}_0(w,v)=\delta_{w,v}$ and for $i\geq 2$,
\[
p_i^{(n)}(w,v):=\sum_{w^{(1)},w^{(2)},\cdots,w^{(i-1)}\in W_n}p^{(n)}_1(w,w^{(1)})p^{(n)}_1(w^{(1)},w^{(2)})\cdots p^{(n)}_1(w^{(i-1)},v).
\]

\begin{lemma}\label{lemma53}
Let $X^{(n)}_i=\tilde{\Theta}_{m,n}(X^{(m,n)}_i)$, $i\geq 0$. Then, $X^{(n)}:=\{X^{(n)}_i\}_{i\geq 0}$ is a random walk on $W_n$ with transition kernels $p_i^{(n)},i\geq 1$, whose initial distribution is given by $\tilde{\Theta}_{m,n}^*\nu$ under the law $\mathbb{P}^{(m,n)}_{\nu}:=\sum_{w'\cdot w\in \tilde{L}_{m,n}}\nu(w'\cdot w)\cdot\mathbb{P}_{w'\cdot w}^{(m,n)}$  for any probability measure $\nu$ on $\tilde{L}_{m,n}$.
\end{lemma}
\begin{proof}
For each $w=\tilde{\Theta}_{m,n}(v'\cdot v)$ with $v'\in \tilde F_{m,2,0}, v\in W_n$, it is easy to check that
\[
\mathbb{P}^{(m,n)}_{v'\cdot v}(X^{(n)}_0=w,X^{(n)}_1=w^{(1)},\cdots, X^{(n)}_i=w^{(i)})=p^{(n)}_1(w,w^{(1)})p^{(n)}_1(w^{(1)},w^{(2)})\cdots p^{(n)}_1(w^{(i-1)},w^{(i)}),
\]
for any $i\geq 0$ and any path $w,w^{(1)},\cdots,w^{(i)}$. The lemma follows immediately.
\end{proof}

Throughout the remaining part of this section, and throughout Sections \ref{sec6}-\ref{sec7}, we always consider $X^{(n)}$ as the random walk defined in Lemma \ref{lemma53}, which is properly coupled with $X^{(m,n)}$. Sometimes, we will use the notation $\mathbb{P}^{(n)}_w$, $w\in W_n$, which should be understood as the law $\mathbb{P}^{(m,n)}_{\nu_w}$, where $\nu_w$ is a probability measure on $\tilde L_{m,n}$ defined as $\nu_w(A):=k^{-2m}\mu_{m+n}(\tilde{\Theta}_{m,n}^{-1}(w)\cap A)$ for each $A\subset \tilde{L}_{m,n}$. So $(\Omega,\mathbb{P}^{(n)}_w,X^{(n)})$ is a random walk with respect to the filtration generated by $X^{(n)}_i,i\geq 0$.

\begin{lemma}\label{lemma54}
	$X^{(n)}$ is reversible with respect to $\pi_n$. In addition, regarding $p_1^{(n)}$ as an operator $p^{(n)}_1:l(W_n)\to l(W_n)$ by
	\[p_1^{(n)}f(w):=\sum_{v\in W_n}p_1^{(n)}(w,v)f(v),\quad \forall f\in l(W_n), \forall w\in W_n,\]
	we have $\mcD_{n}(f)=\langle f-p_1^{(n)}f,f\rangle_{l^2(W_n,\pi_{n})}$ where $\langle f,g\rangle_{l^2(W_n,\pi_n)}:=\sum_{w\in W_n}f(w)g(w)\pi_{n}(w)$.
\end{lemma}

\section{Quick oscillations}\label{sec6}
In this section, for $n\geq 1$, we prove an interesting phenomenon considering the random walk $X^{(n)}$ hitting a pair of boundary faces $\tilde{F}_n:=\tilde{F}_{n,1,0}$ and $\tilde{F}_{n,1,1}$.  \vspace{0.2cm}

\noindent(\textbf{\emph{Hitting time}}). For $B\subset W_n$, we write
$\tau^{(n)}_B:=\min\{i\geq 0:X^{(n)}_i\in B\}$
for the \textit{hitting time}; similarly, for $B\subset\tilde{L}_{m,n}$, we write $\tau^{(m,n)}_B:=\min\{i\geq 0:X^{(m,n)}_i\in B\}$. \vspace{0.2cm}

\begin{definition}\label{def61}
	Define a sequence of stopping times $\mathcal{T}_j, j\geq 1$ as follows. First, we let
	\[\mathcal{T}_1=\min\{i\geq 0:X^{(n)}_i\in \tilde{F}_n\}.\]
	For $l\geq 1$, we define
	\[
	\begin{cases}
		\mathcal{T}_{2l}=\min\{i\geq \mathcal{T}_{2l-1}:X^{(n)}_i\in \tilde{F}_{n,1,1}\},\\
		\mathcal{T}_{2l+1}=\min\{i\geq \mathcal{T}_{2l}:X^{(n)}_i\in \tilde{F}_n\}.
	\end{cases}
	\]
	In particular, we take $\min\emptyset=\infty$ in the above definitions.
\end{definition}

We also introduce the notations $I_n,I_{n,+},I_{n,-}$ for the (discrete) middle layer and half spaces of $W_n$.
\[
\begin{aligned}
	&I_n:=\{w\in W_n:\Psi_wK\cap H_{1,1/2}\neq \emptyset\},\quad n\geq 1,\\
	&I_{n,+}:=\big\{w\in W_n:\Psi_wK\cap\bigcup_{1/2\leq s\leq 1}H_{1,s}\neq \emptyset\big\}, \quad n\geq 1,\\
	&I_{n,-}:=\big\{w\in W_n:\Psi_wK\cap\bigcup_{0\leq s\leq 1/2}H_{1,s}\neq \emptyset\big\}, \quad n\geq 1.
\end{aligned}
\]
The main result of this section will indicate that if $\mbEn_{w}(\mathcal{T}_1)\ll\lambda_n$ for some $w\in I_{n,+}$, then the random walk has the inertance to oscillate quickly between $\tilde{F}_n$ and $\tilde{F}_{n,1,1}$ for a period of time. The proof is based on a symmetry argument.

\begin{proposition}\label{prop62}
	There exist an increasing function $\vartheta:[0,1]\to [0,1]$ and a constant $C>0$ depending only on $K$ such that $\vartheta(0)=0$, $\vartheta$ is continuous at $0$ and
	\[\mbEn_{\nu}(\mathcal{T}_2-\mathcal{T}_1)\leq \vartheta\big(\frac{\mbEn_\nu(\mathcal{T}_1)}{C\lambda_n}\big)\cdot C\lambda_n\]
	holds for any $n\geq 1$ and any probability measure $\nu$ supported on $I_{n,+}$.
\end{proposition}

\subsection{Proof of Proposition \ref{prop62}} We prove Proposition \ref{prop62} in this subsection. First, let's show an upper bound of $\mathbb{E}(\mathcal{T}_1)$, which is based on a standard argument.

\begin{lemma}[Nash inequality]\label{lemma63}
	There exist $\kappa>0$ and a constant $C>0$, both depending only on $K$, such that
	\[
	\|f\|_{l^2(W_n,\pi_n)}^{2+4/\kappa}\leq C\cdot \big(\lambda_n\mcD_n(f)+\|f\|_{l^2(W_n,\pi_n)}^2\big)\cdot N^{-2n/\kappa}\|f\|_{l^1(W_n,\pi_n)}^{4/\kappa}
	\]
	holds for any $n\geq 1$ and $f\in l(W_n)$.
\end{lemma}
\begin{proof}
	Noticing that $\mu_n\leq \pi_n\leq 8\mu_n$, we consider the $l^p$-norm with respect to $\mu_n$ instead for convenience.
	
	For $0\leq m\leq n$, define $f_m\in l(W_n)$ by $f_m(w\cdot v)=\fint_{w\cdot W_{n-m}}fd\mu_n$ for each $w\in W_m,v\in W_{n-m}$. In particular, $f_n=f$ and $f_0\equiv \fint_{W_n}fd\mu_n$. Then, we have for some $C_1>0$,
	\begin{equation}\label{eqn61}
		\begin{aligned}
			\|f\|_{l^2(W_n,\mu_n)}^2
			&=\|f-f_m\|_{l^2(W_n,\mu_n)}^2+\|f_m\|_{l^2(W_n,\mu_n)}^2\\
			&\leq \lambda_{n-m}\mcD_n(f)+N^{m-n}\|f_m\|_{l^1(W_n,\mu_n)}^2\\
			&\leq C_1k^{-2m}\big(\lambda_n\mcD_n(f)\big)+k^{md_H}N^{-n}\|f\|_{l^1(W_n,\mu_n)}^2,
		\end{aligned}
	\end{equation}
	where in the second line we use the observation
	\[
	\|f-f_m\|^2_{l^2(W_n,\mu_n)}= \sum_{w\in W_m}\|f-f_m\|^2_{l^2(w\cdot W_{n-m},\mu_n)}\leq \lambda_{n-m}\sum_{w\in W_m}\mcD_{n,w\cdot W_{n-m}}(f),
	\]
	and in the third line we use the fact that $\lambda_n\gtrsim k^{2m}\lambda_{n-m}$ as shown in Proposition \ref{prop32}.
	
	The desired inequality then follows by considering two possible cases.\vspace{0.2cm}
	
		\textit{Case 1}. $N^{-n}\|f\|^2_{l^1(W_n,\mu_n)}\leq \lambda_n\mcD_n(f)$. In this case, {we notice that $\|f\|^2_{l^1(W_n,\mu_n)}\gtrsim \mcD_n(f)$}, so we fix some  $\kappa\geq d_H$ such that $k^{(2+\kappa)n}\gtrsim N^n\lambda_n$ (which is plausible by Proposition \ref{prop32} and Lemma \ref{lemma33}), and we choose
	\[
	m=\Big[\frac{\log\big(\lambda_n\mcD_n(f)\big)-\log(N^{-n}\|f\|^2_{l^1(W_n,\mu_n)})}{(2+\kappa)\log k}\Big]\wedge n.
	\]
	Plug the choice $m$ into (\ref{eqn61}), we see the proposition holds for this case, with the same $\kappa$. \vspace{0.2cm}
	
	\textit{Case 2}.  $N^{-n}\|f\|^2_{l^1(W_n,\mu_n)}\geq \lambda_n\mcD_n(f)$.  In this case,
	\[
	\begin{aligned}
		\|f\|_{l^2(W_n,\mu_n)}^2&= \big\|f-\fint_{W_n} fd\mu_n\big\|_{l^2(W_n,\mu_n)}^2+N^n\big|\fint_{W_n}fd\mu_n\big|^2\\
		&\leq \lambda_n\mcD_n(f)+N^{-n}\|f\|^2_{l^1(W_n,\mu_n)}\leq 2 N^{-n}\|f\|^2_{l^1(W_n,\mu_n)}.
	\end{aligned}
	\]
	So the desired inequality holds trivially.
\end{proof}

For convenience, for each operator $p:l(W_n)\to l(W_n)$ and $s,t\in [1,\infty]$, we write
\[
\|p\|_{l^s\to l^t}=\sup_{f\in l(W_n)\setminus\{0\}}\frac{\|pf\|_{l^t(W_n,\pi_n)}}{\|f\|_{l^s(W_n,\pi_n)}}.
\]
Then, as a consequence of Lemma \ref{lemma63} and Lemma \ref{lemma54}, following standard arguments (see  \cite[Theorem 4.1]{CKS} or  \cite[Theorem 4.3]{B1} for example), we have the following estimate.
\begin{corollary}\label{coro64}
	There exists $C_1>0$ independent of $n$ such that $\|\pn_{[\lambda_n]+1}\|_{l^2\to \l^\infty}\leq C_1N^{-n/2}$.
\end{corollary}

\noindent(\textbf{\emph{Dirichlet transition kernels on $W_n$}}). For $i\geq 0$ and $B\subset W_n$, we define
\[
\pn_{B,i}(w,v)=\mbEn_w(X^{(n)}_i=v,\tau_B>i), \quad\forall w,v\in W_n.
\]
For $f\in l(W_n)$, we write $\pn_{B,i} f(w)=\sum_{v\in W_n}\pn_{B,i}(w,v)f(v)$.
In particular, we can see that $\pn_{B,1}f=1_{W_n\setminus B}\cdot \pn_1(1_{W_n\setminus B}\cdot f)$, and the operator $p^{(n)}_{B,1}$ is self-adjoint, where $1_A$ is the characteristic function of $A$. In addition, as operators from $l(W_n)$ to $l(W_n)$, $\pn_{B,i}=(\pn_{B,1})^{\circ i}$ for $i\geq 1$.

\begin{lemma}\label{lemma65}
	$\|\pn_{\tilde{F}_n,1}\|_{l^2\to\l^2}\leq 1-\frac{C_2}{\lambda_n}$ for some $C_2>0$ depending only on $K$.
\end{lemma}
\begin{proof}
Let $f\in l(W_n)$, we define $|f|$ by $|f|(w)=|f(w)|,\forall w\in W_n$, and we let $g=|f|\cdot 1_{W_n\setminus \tilde{F}_n}$. Then, we can find constants $C_2,C_3>0$ depending only on $K$ such that
\[
\begin{aligned}
	\langle g-\pn_{\tilde{F}_n,1}g,g\rangle_{l^2(W_n,\pi_n)}&=
	\langle g-\pn_1g,g\rangle _{l^2(W_n,\pi_n)}\\
	&=\mcD_n(g)\geq C_3 \lambda_n^{-1}\|g\|^2_{l^2(W_n,\mu_n)}\geq C_2 \lambda_n^{-1}\|g\|^2_{l^2(W_n,\pi_n)},
\end{aligned}
\]
where the second equality is due to Lemma \ref{lemma54}, the first inequality is due to the facts that $N^n|\fint_{W_n}gd\mu_n|^2\lesssim\lambda_n\mcD(g)$ by Proposition \ref{prop41}, and $\|g-\fint_{W_n} gd\mu_n\|^2_{l^2(W_n,\mu_n)}\leq \lambda_n\mcD_n(g)$ by the definition of $\lambda_n$. It follows that
\[\begin{aligned}
\langle \pn_{\tilde{F}_n,1}f,f\rangle_{l^2(W_n,\pi_n)}&\leq \langle \pn_{\tilde{F}_n,1}|f|,|f|\rangle _{l^2(W_n,\pi_n)}=\langle \pn_{\tilde{F}_n,1}g,g\rangle_{l^2(W_n,\pi_n)}\\
&\leq (1-C_2\lambda_n^{-1})\|g\|^2_{l^2(W_n,\pi_n)}\leq (1-C_2\lambda_n^{-1})\|f\|^2_{l^2(W_n,\pi_n)}.
\end{aligned}\]	
Noticing that the inequality holds for any $f\in l(W_n)$ and $\pn_{\tilde{F}_n,1}$ is self-adjoint, the lemma follows.
\end{proof}

\begin{proposition}\label{prop66}
	$\mbEn_w(\mathcal{T}_1)\leq C\lambda_n,\forall w\in W_n$ for some constant $C>0$ depending only on $K$.
\end{proposition}
\begin{proof}
	The claim is an easy consequence of the fact that
	\[
	\mbEn_w(\mathcal{T}_1)=\sum_{i=1}^\infty \pn_{\tilde{F}_n,i}1_{W_n\setminus \tilde{F}_n}(w),\quad\forall w\in W_n,
	\]
	and the fact that
	\begin{equation}\label{eqn62}
		\|\pn_{\tilde{F}_n,i}1_{W_n\setminus\tilde F_n}\|_{l^\infty(W_n,\pi_n)}\leq
		\begin{cases}
			1,&\text{ if }1\leq i\leq \lambda_n,\\
			C_3(1-C_2\lambda_n^{-1})^{i-\lambda_n},&\text{ if }i>\lambda_n,
		\end{cases}
	\end{equation}
	where $C_2$ is the same constant in Lemma \ref{lemma65} and $C_3>0$ is a constant depending only on $K$. In fact, the first case of (\ref{eqn62}) is due to the fact $\|\pn_{\tilde{F}_n,i}\|_{l^\infty\to l^\infty}\leq 1$, and the second case of (\ref{eqn62}) is a consequence of Corollary \ref{coro64}, Lemma \ref{lemma65}, the fact that $\|1_{W_n\setminus\tilde{F}_n}\|_{l^2(W_n,\pi_n)}\lesssim N^{n/2}$, and the fact that $\pn_if\geq \pn_{\tilde{F}_n,i}f$ for any non-negative $f\in l(W_n)$.
\end{proof}

By a same proof, we can see the following result, which will be used in Section \ref{sec7}.
\begin{proposition}\label{prop67}
	Let $B$ be the same set in Example \ref{example42} (b), which connects $I_n$ and $\tilde F_{n,1,1}$. Then $\mathbb{E}^{(n)}_w(\tau_B)\leq C\lambda_n,\forall w\in W_n$ for some $C>0$ depending only on $K$.
\end{proposition}\vspace{0.2cm}

\begin{proof}[Proof of Proposition \ref{prop62}]
	The proposition is a consequence of the reflection symmetry and Proposition \ref{prop66}. In particular, we let $C$ be the constant in Proposition \ref{prop66}, so $\mbEn_{\nu}(\mathcal{T}_1)\leq C\lambda_n$ and $\mbEn_{\nu}(\mathcal{T}_2-\mathcal{T}_1)\leq C\lambda_n$ for any probability measure $\nu$ on $W_n$. \vspace{0.2cm}
	
	For simplicity, we write $A_n=\tilde{F}_n,B_n=\tilde{F}_{n,1,1}$. \vspace{0.2cm}
	
\textit{Claim. There exists $C_1\in (0,1)$ depending only on $K$ such that for each $w\in I_n$,
\begin{eqnarray}
\label{eqn63}&C_1\mbEn_w(\tau_{B_n})\leq \mbEn_w(\tau_{A_n}),\\
\label{eqn64}&\mbPn_w(\tau_{A_n}<\tau_{B_n})\geq C_1.
\end{eqnarray}}

In fact, if $\Psi_w(H_{1,1/2})=H_{1,1/2}$, then $\mbEn_w(\tau_{B_n})=\mbEn_w(\tau_{A_n})$ and $\mbPn_w(\tau_{A_n}<\tau_{B_n})=\frac{1}{2}$; otherwise, there is $w'$ symmetric to $w$ about the plane $H_{1,1/2}$, so by symmetry we have
\[
\begin{aligned}
	\mbEn_w(\tau_{A_n})&\geq p^{(n)}_1(w,w')\mbEn_{w'}(\tau_{A_n})=p^{(n)}_1(w,w')\mbEn_{w}(\tau_{B_n}),\\
	\mbPn_w(\tau_{A_n}<\tau_{B_n})&\geq p^{(n)}_1(w,w')\mbPn_{w'}(\tau_{A_n}<\tau_{B_n})=p^{(n)}_1(w,w')\mbPn_{w}(\tau_{B_n}<\tau_{A_n}).
\end{aligned}
\]
Hence, $C_1=\min_{w\sn w'}\frac{p^{(n)}_1(w,w')}{1+p^{(n)}_1(w,w')}\geq \frac{1}{27}$ works for the claim.\vspace{0.2cm}

Next, we define another sequence of hitting times $\mathcal{S}_j, j\geq 1$ as follows. First, we let
	\[\mathcal{S}_1=\min\{i\geq 0:X^{(n)}_i\in I_n\}.\]
	Then, for $l\geq 1$, we define
	\[
	\begin{cases}
		\mathcal{S}_{2l}=\min\{i\geq \mathcal{S}_{2l-1}:X^{(n)}_i\in A_n\cup B_n\},\\
		\mathcal{S}_{2l+1}=\min\{i\geq \mathcal{S}_{2l}:X^{(n)}_i\in I_n\}.
	\end{cases}
	\]
	In particular, we take $\min\emptyset=\infty$ in the above definitions.
	
	Define a sequence of events: $E_l:=\{\mathcal S_{2},\mathcal S_4,\cdots,\mathcal S_{2l}\in B_n\}$ for $l\geq 1$, and let $E_0=\Omega$ be the whole sample space. Noticing that $\nu$ is supported on $I_{n,+}$, $E_l=\{\mathcal S_{2l}<\mathcal{T}_1\}$ for $l\geq 1$. Also, by (\ref{eqn64}) and the strong Markov property, $\mathbb{P}_\nu(E_l)\leq (1-C_1)^l$. For each $L\geq 0$, we decompose $\Omega=\big(\bigcup_{l=1}^{L}(E_{l-1}\setminus E_{l})\big)\bigcup E_L$ and see that
	\[
	\begin{aligned}
		\mbEn_\nu(\mathcal{T}_2-\mathcal{T}_1)&=\sum_{l=1}^L \mbEn_\nu(\mathcal{T}_2-\mathcal{T}_1;E_{l-1}\setminus E_l)+\mbEn_\nu(\mathcal{T}_2-\mathcal{T}_1;E_L)\\
		&\leq \sum_{l=1}^L \mbEn_\nu(\mathcal{T}_2-\mathcal{T}_1;E_{l-1})+\mbEn_\nu(\mathcal{T}_2-\mathcal{T}_1;E_L)\\
		&=\sum_{l=1}^L \mbEn_\nu\Big(\mbEn_{X^{(n)}_{\mathcal S_{2l-1}}}\big(\tau_{B_n}-\tau_{A_n}; \tau_{B_n}>\tau_{A_n}\big);{E_{l-1}}\Big)+\mbEn_\nu\big(\mbEn_{X^{(n)}_{\mathcal{T}_1}}(\tau_{B_n});E_L\big)\\
		&\leq \sum_{l=1}^L \mbEn_\nu\Big(\mbEn_{X^{(n)}_{\mathcal S_{2l-1}}}(\tau_{B_n});E_{l-1}\Big)+C\lambda_n(1-C_1)^L\\
		&\leq \sum_{l=1}^L \mbEn_\nu\Big(C_1^{-1}\mbEn_{X^{(n)}_{\mathcal S_{2l-1}}}(\tau_{A_n});E_{l-1}\Big)+C\lambda_n(1-C_1)^L\\
		&=C_1^{-1}\sum_{l=1}^L\mbEn_\nu(\mathcal{T}_1-\mathcal S_{2l-1};E_{l-1})+C\lambda_n(1-C_1)^L\\&\leq C_1^{-1}L\mbEn_\nu(\mathcal{T}_1)+C\lambda_n(1-C_1)^L,
	\end{aligned}
	\]
	for some constant $C>0$ depending only on $K$, where we use Proposition \ref{prop66} and the fact $\mathbb{P}(E_l)\leq (1-C_1)^{l}$ in the fourth line, and we use (\ref{eqn63}) in the fifth line. Hence, we can take
	\[\vartheta(t):=\inf_{L\geq 0}\big\{C_1^{-1}Lt+(1-C_1)^L\big\},\]
	which clearly satisfies the desired requirement in the proposition.
\end{proof}

\subsection{A consequence for $X^{(m,n)}$} Proposition \ref{prop62} will be useful combining with the coupling argument Lemma \ref{lemma53}. The quick oscillation implies quick hitting of $\tilde{F}_{m+n}\cap \tilde{L}_{m,n}$.

\begin{remark}
We remark that for $A\subset W_n$, $\tau^{(n)}_A=\tau^{(m,n)}_{\tilde{\Theta}_{m,n}^{-1}(A)}$, with the coupling setting of Lemma \ref{lemma53}.
\end{remark}

\begin{lemma}\label{lemma69}
There exists $C\in (0,1)$ such that for any $w'\cdot w\in \tilde{L}_{m,n}$ with $w'\in \tilde F_{m,2,0}$, $w\in W_n$,  we have
\[
\mbPmn_{w'\cdot w}(\tau^{(m,n)}_{\tilde{F}_{m+n}}\leq \mathcal{T}_{k^m+1})\geq C^{k^m+1}.
\]
\end{lemma}
\begin{proof}
This is an easy consequence of symmetry. For $l=0,1,\cdots,k^m$, we define
\[A_l=\big\{\kappa\in \tilde{L}_{m,n}:\Psi_{\kappa}\square\cap H_{1,l/k^m}\neq \emptyset\big\},\]
which can be viewed as horizontal layers of $(m+n)$-words in $\tilde L_{m,n}$ spaced about $k^{-m}$ apart. In particular, $A_0=\tilde{F}_{m+n}\cap \tilde{L}_{m,n}$ and $A_{k^m}=\tilde{F}_{m+n,1,1}\cap \tilde{L}_{m,n}$. Moreover, one can see that $\tilde{\Theta}_{m,n}^{-1}(\tilde{F}_n)=\bigcup_{j=0}^{\lfloor k^m/2\rfloor}A_{2j}$ and $\tilde{\Theta}_{m,n}^{-1}(\tilde{F}_{n,1,1})=\bigcup_{j=1}^{\lceil k^m/2\rceil}A_{2j-1}$. By a same argument as the proof of (\ref{eqn64}), we can see that
\[
\begin{aligned}
\mathbb{P}_{w'\cdot w}^{(m,n)}\big(X^{(m,n)}_{\tau^{(n)}_{\tilde{F}_n}}\in A_{j-1}\big)\geq C,\qquad &\forall\text{ odd }j,\ w'\cdot w\in A_j,\\
\mathbb{P}_{w'\cdot w}^{(m,n)}\big(X^{(m,n)}_{\tau^{(n)}_{\tilde{F}_{n,1,1}}}\in A_{j-1}\big)\geq C,\qquad &\forall\text{ nonzero even }j,\ w'\cdot w\in A_j,
\end{aligned}
\]
where we can take $C=\frac{\alpha}{1+\alpha}$ with $\alpha=\frac 13\cdot \frac {1}{2\cdot 9}=\frac{1}{54}$ (recall (\ref{eqn51})). The lemma then follows from the strong Markov property.
\end{proof}

\begin{corollary}\label{coro610}
	For each $m\geq 0$, there exists a bounded increasing function $\vartheta_m:\mathbb{R}_+\to \mathbb{R}_+$ such that $\vartheta_m(0)=0$, $\vartheta_m$ is continuous at $0$ and
	\[\mbEmn_{\nu}(\tau^{(m,n)}_{\tilde{F}_{m+n}})\leq \vartheta_m\big(\frac{\mbEn_{\tilde{\Theta}_{m,n}^*\nu}(\mathcal{T}_1)}{\lambda_n}\big)\cdot\lambda_n\]
	for any probability measure $\nu$ on $\tilde{L}_{m,n}$ such that $\tilde{\Theta}_{m,n}^*\nu$ is supported on $I_{n,+}$.
\end{corollary}
\begin{proof}
	By Lemma \ref{lemma69}, we know that for some $C_1\in (0,1)$ depending only on $m$,
	\[
	\mbPmn_\nu\big(\tau^{(m,n)}_{\tilde{F}_{m,n}}\geq \mathcal{T}_{l\cdot (k^{m}+1)}\big)\leq C_1^l,\quad \forall l\geq 1.
	\]
	Hence, by letting $C_2$ be the constant in Proposition \ref{prop62}, we have
	\[
	\begin{aligned}
		\mbEmn_{\nu}(\tau^{(m,n)}_{\tilde{F}_{m+n}})&\leq \sum_{l=1}^\infty C_{1}^{l-1}\cdot \mbEn_{\tilde{\Theta}_{m,n}^*\nu}(\mathcal{T}_{l\cdot (k^{m}+1)})\\
		&\leq \sum_{l=1}^\infty C_1^{l-1}\sum_{i=0}^{l(k^m+1)-1}\vartheta^{\circ i}\big(\mbEn_{\tilde{\Theta}_{m,n}^*\nu}(\frac{\mathcal{T}_1}{C_2\lambda_n})\big)\cdot C_2\lambda_n\\
		&\leq\Big(\sum_{i=0}^\infty C_2\cdot (\sum_{l=\lceil\frac{i+1}{k^m+1}\rceil}^{\infty}C_1^{l-1})\cdot\vartheta^{\circ i}\big(\mbEn_{\tilde{\Theta}_{m,n}^*\nu}(\frac{\mathcal{T}_1}{C_2\lambda_n})\big)\Big)\cdot\lambda_n.
	\end{aligned}
	\]
	The corollary follows, noticing that $\vartheta:[0,1]\to [0,1]$ and $\vartheta$ is continuous at $0$.
\end{proof}

\section{Lower bound estimate of the mean hitting time}\label{sec7}
In this section, combining the results of Sections \ref{sec4}--\ref{sec6}, we prove the following lower bound estimate of the mean hitting time $\mbEn_w(\mathcal{T}_1),\forall w\in \tilde F_{n,1,1}$, in contrast to Proposition \ref{prop66}.

\begin{theorem}\label{thm71}
There is a constant $C>0$ depending only on $K$ such that $\mbEn_w(\mathcal{T}_1)\geq C\cdot\lambda_n$ for any $w\in \tilde{F}_{n,1,1}$.
\end{theorem}
\begin{proof} For convenience, we let $\phin\in l(W_n)$ be defined as $\phin(w):=\mbEn_w(\mathcal{T}_1),\forall w\in W_n$. In the following, we fix $m,n\geq 1$, and write
	\[t=\min\big\{\frac{\phin(w)}{\lambda_n}:w\in \tilde F_{n,1,1}\big\}.\]\vspace{0.2cm}
	
	\noindent\textit{Claim 1. There exists a path $w^{(0)},w^{(1)},\cdots,w^{(L)}$ such that $w^{(0)}\in \tilde{F}_{n,1,1}$, $w^{(L)}\in I_n$, and  $w^{(l)}\in I_{n,+}$,  $\phin(w^{(l)})\leq t\lambda_n,\forall 0\leq l\leq L$. }
	\begin{proof}[Proof of Claim 1]
	 The claim is inspired by \cite{B1}. First, by the definition of $t$, we can find $w^{(0)}\in \tilde{F}_{n,1,1}$ such that $\phin(w^{(0)})=t\lambda_n$. Next, since $\phin$ is superharmonic in $W_n\setminus \tilde{F}_n$, by the minimal principle,  we can find a path $w=w^{(0)},w^{(1)},\cdots,w^{(L')}$ such that $w^{(L')}\in \tilde{F}_n$, $\phin(w^{(l)})\leq t\lambda_n$ for any $0\leq l\leq L'$. Choose $L$ such that $w^{(L)}$ is the first element of the path satisfying $w^{(L)}\in I_n$.
	\end{proof}
	
	\noindent\textit{Claim 2. There exists a constant $C_1>0$ depending only on $K$ such that for each $w'\cdot w\in \tilde{L}_{m,n}$ with $w'\in \tilde F_{m,2,0}$, $w\in W_n$,
		$\mbEmn_{w'\cdot w}(\tau^{(m,n)}_{\tilde{F}_{m+n}})\leq \big(C_1+\vartheta_m(t)\big)\lambda_n$.}
	
	\begin{proof}[Proof of Claim 2]
		We let $A=\{w^{(0)},w^{(1)},\cdots,w^{(L)}\}$ as introduced in Claim 1, and let $C_1$ be the constant in Proposition \ref{prop67}. Recall that we have coupled $X^{(m,n)}$ and $X^{(n)}$ so that $\tilde{\Theta}_{m,n}(X^{(m,n)}_i)=X^{(n)}_i$. Then,
		\[
		\begin{aligned}
			\mbEmn_{w'\cdot w}\big(\tau^{(m,n)}_{\tilde{F}_{m+n}}\big)&=\mbEmn_{w'\cdot w}\big(\tau^{(n)}_A\wedge \tau^{(m,n)}_{\tilde{F}_{m+n}}\big)+\mbEmn_{w'\cdot w}\big(\tau^{(m,n)}_{\tilde{F}_{m+n}}-\tau^{(n)}_A; \tau^{(m,n)}_{\tilde{F}_{m+n}}>\tau^{(n)}_A\big)\\
			&\leq C_1\lambda_n+\mbEmn_{w'\cdot w}\big(\mbEmn_{X^{(m,n)}_{\tau^{(n)}_A}}( \tau^{(m,n)}_{\tilde{F}_{m+n}})\big)\\
			&\leq C_1\lambda_n+\vartheta_m(t)\lambda_n,
		\end{aligned}
		\]
		where we use Proposition \ref{prop67} and the strong Markov property in the second line, and we use Corollary \ref{coro610} in the last line.
	\end{proof}

	\noindent\textit{Claim 3. Define
		\[
		\bar{\mathscr{R}}_{m+n}:=\sup\big\{\frac{(\fint_{\tilde{F}_{m+n, 2, 0}}fd\mu_{m+n})^2}{\mcD_{m+n}(f)}:f\in l(W_{m+n}),f|_{\tilde{F}_{m+n}}\equiv 0\big\}.
		\]
		Then there exist $C_2,C_3>0$ depending only on $K$ such that }
	\[
	\bar{\mathscr{R}}_{m+n}\leq 2\big(C_2C_1+C_2\vartheta_m(t)+C_3\big)N^{-n}\lambda_n\cdot k^{-2m}.
	\]
	\begin{proof}[Proof of Claim 3]
		We let $\phi^{(m,n)}\in\l(\tilde L_{m,n})$ be defined by $\phi^{(m,n)}(w'\cdot w)=\mbEmn_{w'\cdot w}(\tau^{(m,n)}_{\tilde{F}_{m+n}})$ for $w'\in \tilde F_{m,2,0}$, $w\in W_n$, and denote $\mathcal{D}_{m,n}(f,g)=\langle f-p^{(m,n)}_1f,g\rangle_{l^2(\tilde L_{m,n}, \pi_{m,n})}$ for $f,g\in  l(\tilde L_{m,n})$. We can see that
		\[
		\mathcal{D}_{m,n}(\phi^{(m,n)},g)=\langle 1,g\rangle_{l^2(\tilde L_{m,n}, \pi_{m,n})}
		\]
		for all $g\in l(\tilde{L}_{m,n})$ satisfying $g|_{\tilde L_{m,n}\cap\tilde{F}_{m+n}}\equiv 0$. It follows that
		\begin{equation}\label{eqn71}
			\frac{(\fint_{\tilde{L}_{m,n}} \phi^{(m,n)}d\pi_{m,n})^2}{\mcD_{m,n}(\phi^{(m,n)},\phi^{(m,n)})}=\sup\big\{\frac{(\fint_{\tilde{L}_{m,n}}fd\pi_{m,n})^2}{\mcD_{m,n}(f,f)}:f\in l(\tilde{L}_{m,n}),f|_{\tilde L_{m,n}\cap \tilde{F}_{m+n}}\equiv 0\big\}.
		\end{equation}
		On the other hand, we also have
		\begin{equation}\label{eqn72}
			\begin{aligned}
				\frac{(\fint_{\tilde{L}_{m,n}} \phi^{(m,n)}d\pi_{m,n})^2}{\mcD_{m,n}(\phi^{(m,n)},\phi^{(m,n)})}&=\frac{(\fint_{\tilde{L}_{m,n}} \phi^{(m,n)}d\pi_{m,n})^2}{\int_{\tilde{L}_{m,n}} \phi^{(m,n)}d\pi_{m,n}}\leq \big(\pi_{m,n}(\tilde{L}_{m,n})\big)^{-1}\|\phi^{(m,n)}\|_{l^\infty(\tilde L_{m,n},\pi_{m,n})}.
			\end{aligned}
		\end{equation}
	   Combining (\ref{eqn71}) and (\ref{eqn72}), we see that for some constant $C_2>0$ depending only on $K$,
		\begin{equation}\label{eqn73}
			\begin{aligned}
				&\sup\big\{\frac{(\fint_{\tilde{L}_{m,n}}fd\mu_{m+n})^2}{\mcD_{m+n}(f)}:f\in l(W_{m+n}),f|_{\tilde{F}_{m+n}}\equiv 0\big\}\\
				\leq&\sup\big\{\frac{(\fint_{\tilde{L}_{m,n}}fd\mu_{m+n})^2}{\mcD_{m+n,\tilde{L}_{m,n}}(f)}:f\in l(\tilde{L}_{m,n}),f|_{\tilde{F}_{m+n}}\equiv 0\big\}\\
				\leq&C_2\big(C_1+\vartheta_m(t)\big)N^{-n}\lambda_n\cdot k^{-2m}.
			\end{aligned}
		\end{equation}
	   In fact, the first inequality of (\ref{eqn73}) is trivial as $\mcD_{m+n}(f)\geq \mcD_{m+n,\tilde{L}_{m+n}}(f),\forall f\in l(W_{m+n})$; the second inequality of (\ref{eqn73}) is due to (\ref{eqn71}), (\ref{eqn72}), the fact that $\pi_{m,n}(\tilde{L}_{m,n})\asymp k^{2m}N^n$, and the fact that $\|\phi^{(m,n)}\|_{l^\infty(\tilde{L}_{m,n},\pi_{m,n})}\leq \big(C_1+\vartheta_m(t)\big)\lambda_n$ by Claim 2. Next, we apply Proposition \ref{prop41} locally to see		
		\[\begin{aligned}
		\big|\fint_{\tilde{L}_{m,n}}fd\mu_{m+n}-\fint_{\tilde{F}_{m+n,2,0}}fd\mu_{m+n}\big|&\leq\frac{1}{k^{2m}} \sum_{w'\in\tilde F_{m,2,0}}\big|\fint_{w'\cdot W_n}fd\mu_{m+n}-\fint_{w'\cdot \tilde F_{n,2,0}}fd\mu_{m+n}\big|\\
		&\leq \sqrt{C_3N^{-n}\lambda_n\cdot k^{-2m}\cdot \mcD_{m+n}(f)},\quad\forall f\in l(W_{m+n}),
		\end{aligned}
		\]
		where $C_3>0$ is a constant depending only on $K$. Combining the above estimate with (\ref{eqn73}), we get the claim.
	\end{proof}

	\noindent\textit{Claim 4. $\mathscr{R}_{m+n}\leq 2\cdot\bar{\mathscr{R}}_{m+n}$.}
	\begin{proof}[Proof of Claim 4]
		We pick $f\in l(W_{m+n})$ such that
		\[\mathscr{R}_{m+n}=\frac{\big|\fint_{\tilde{F}_{m+n,2,0}}fd\mu_{m+n}-\fint_{\tilde{F}_{m+n,1,0}}fd\mu_{m+n}\big|^2}{\mcD_{m+n}(f)},\]
		and we can assume that $f$ is anti-symmetric with respect to the reflection symmetry that interchanges $\tilde{F}_{m+n,2,0}$ and $\tilde{F}_{m+n,1,0}$. More precisely, we define $\tilde{\Gamma}_{[1,2]}:W_{m+n}\to W_{m+n}$ with the relation $\Gamma_{[1,2]}(\Psi_w K)=\Psi_{\tilde{\Gamma}_{[1,2]}(w)}K$, where $\Gamma_{[1,2]}$ was defined at the beginning of Section \ref{sec2}. Then we can require $f=-f\circ \tilde{\Gamma}_{[1,2]}$ (since otherwise we can take $\frac{f-f\circ\tilde{\Gamma}_{[1,2]}}{2}$ as the function).
		
		Let $A'=\big\{w\in W_{m+n}:d(\Psi_w K,F_{2,0})\leq d(\Psi_wK,F_{1,0})\big\}$ and define $g\in l(W_{m+n})$ as
		\[g(w)=\begin{cases}
			f(w),&\text{ if }w\in A',\\
			0,&\text{ if }w\in W_{m+n}\setminus A'.
		\end{cases}\]
		Then $g|_{\tilde{F}_{m+n,1,0}}\equiv 0$, $g|_{\tilde{F}_{m+n,2,0}}=f|_{\tilde{F}_{m+n,2,0}}$ and $\mcD_{m+n}(g)\leq \frac{1}{2}\mcD_{m+n}(f)$ by the anti-symmetry of $f$. Hence,
		\[\bar{\mathscr{R}}_{m+n}\geq \frac{\big|\fint_{\tilde{F}_{m+n,2,0}}gd\mu_{m+n}\big|^2}{\mcD_{m+n}(g)}\geq \frac{1}{4}\cdot 2\cdot \mathscr{R}_{m+n}.\]
	\end{proof}
	
	Finally, by Claims 3 and 4, we see that for some $C_4,C_5>0$ depending only on $K$,
	\[
	\begin{aligned}
		\mathscr{R}_{m+n}&\leq 4(C_2C_1+C_2\vartheta_m(t)+C_3) N^{-n}\lambda_n\cdot k^{-2m}\\
		&\leq \big(C_4+C_5\vartheta_m(t)\big)\cdot N^{-m-n}\lambda_{m+n}\cdot (N^mk^{-4m}),
	\end{aligned}
	\]
	where we use Proposition \ref{prop32} in the second line. On the other hand, by Propositions \ref{prop45}, we also have $\mathscr{R}_{m+n}\geq C_6N^{-m-n}\lambda_{m+n}$ for some $C_6>0$ depending only on $K$.  Hence,
	\[C_4+C_5\vartheta_m(t)\geq C_6k^{4m}N^{-m}.\]
	We fix $m=m_0$, which depends only on $K$ (since $C_4,C_6$ depends only on $K$), such that  $2C_4\leq C_6k^{4m_0}N^{-m_0}$. Then,
	\[\vartheta_{m_0}(t)\geq C_4C^{-1}_5.\]
	This implies that $t\geq C_7$ for some $C_7>0$ depending only on $K$, since $\vartheta_{m_0}(0)=0$ and $\vartheta_{m_0}$ is continuous at $0$. The theorem then follows immediately from the definition of $t$.
\end{proof}

As a consequence of Theorem \ref{thm71}, we get the following face-to-face resistance estimate.

\begin{corollary}\label{coro72}
	Denote  the effective resistance between $\tilde{F}_{n,1,1}$ and  $\tilde{F}_n:=\tilde{F}_{n,1,0}$ by
	\[R_{n, F}:=\big(\inf\{\mcD_n(f):f|_{\tilde{F}_{n,1,1}}\equiv 1,f|_{\tilde{F}_{n}}\equiv 0\}\big)^{-1}.\]
    Then, there is a constant $C>0$ independent of $n$ such that
	\[R_{n,F}\geq C\cdot N^{-n}\lambda_n.\]
\end{corollary}
\begin{proof}
	Let $\phi^{(n)}\in l(W_n)$ be defined as $\phi^{(n)}(w)=\mbEn_w(\mathcal{T}_1)$ as in the proof of Theorem \ref{thm71}. By Lemma \ref{lemma54} and Propositions \ref{prop66}, we notice that
	\[
	\mathcal{D}_n(\phi^{(n)})=\langle{\phi^{(n)}},1\rangle_{l^2(W_n,\pi_n)}\leq C_1N^n\lambda_n.
	\]
	for some $C_1>0$ depending only on $K$. Let $C_2$ be the constant in Theorem \ref{thm71}, and we let $f=\big((C_2\lambda_n)^{-1}\phi^{(n)}\big)\wedge 1$.  Then clearly $f|_{\tilde{F}_{n,1,1}}\equiv 1,f|_{\tilde{F}_{n}}\equiv 0$, and we have the estimate
	\[
	\mcD_n(f)\leq (C_2\lambda_n)^{-2}\cdot (C_1N^n\lambda_n)=C_1C_2^{-2}N^n\lambda_n^{-1}.
	\]
	This gives the desired lower bound of the resistance $R_{n,F}$.
\end{proof}

\section{A capacity estimate}\label{sec8}
In this section, our main aim is to prove the following lower bound estimate of $R_n$. In some contents, for example \cite{AB,BCM,BM,GHL2,M1,M2}, the reverse of the effective resistance is also called the $0$-capacity.

\begin{theorem}\label{thm81}
  There is $C > 0$ independent of $n$ so that $R_n\geq C\cdot N^{-n} \lambda_n$ for any $n\geq 1$. In addition, there is $\rho>1$ such that $R_n\asymp N^{-n}\lambda_n\asymp N^{-n}\sigma_n\asymp \rho^{-n}$.
\end{theorem}

\begin{remark}\label{remark82}
Together with Proposition \ref{prop32} (in particular, (\ref{eqn33})) and Lemma \ref{lemma33}, the estimate  $R_n \geq CN^{-n}\lambda_n$ implies that $N^nR_n\asymp\lambda_n\asymp \sigma_n$, and in addition, there is a \textit{resistance renormalization factor} $\rho>0$ so that $R_n\asymp \rho^{-n}$. Furthermore, \[1<\rho\leq N/k^2,\] where $\rho\leq N/k^2$ follows from (\ref{eqn32}), and  $\rho>1$ will follow from the later Lemma \ref{lemma81}. This will play a key role in constructing a limit (self-similar) form on a $\mathcal{USC}^{(3)}$ in the next two sections.
\end{remark}

Recall that by Kusuoka and Zhou in \cite{KZ}, for a limited range of local symmetric fractals including the classical Sierpinski carpets, a slightly different version of Theorem \ref{thm81} was verified by the ``Knight move'' method due to Barlow and Bass \cite{BB1}. This approach heavily relies on the local symmetry of the fractal, and there are very limited results about side moves and corner moves in the $\mathcal {USC}$ setting due to the more flexible geometry caused by allowing cells to live off the grids. In an earlier work,  for planar $\mathcal {USC}$, the authors \cite{CQ1} developed a purely analytic  ``building brick'' technique (constructing functions with ``linear'' boundary values and controllable energy estimates) which overcomes the difficulty caused by the worse geometry.  In this section, we will see the technique is also applicable in the $\mathcal{USC}^{(3)}$ setting. See also \cite{CQW} for an extension to more flexible planar Sierpinski like fractals which include many irrationally ramified fractals.\vspace{0.2cm}

We will apply Corollary \ref{coro72} to inductively construct a function in $l(W_n)$ which is ``linear'' on each face of $W_n$, whose energy is controlled by a multiple of $N^{n}\lambda_n^{-1}$ from above.\vspace{0.2cm}

For each $o\in \{1,2,3\}$, we denote $\tilde{\Gamma}_{[o]}:W_*\to W_*$ so that $\Gamma_{[o]}(\Psi_wK) = \Psi_{\tilde{\Gamma}_{[o]}(w)}K$ for any $w\in W_*$ and denote $\tilde{\Gamma}_{[o,o']}$ in a similar way. As in the proof of Lemma \ref{lemma28}, for $o\in\{1,2,3\}$, we denote $P_{o}(x) = x_{o}$ for $x = (x_1,x_2,x_3)\in \mathbb{R}^3$, and for each $w\in W_*$ we write \[a_w = P_1\circ\Psi_w(0,0,0), \quad b_w = P_1\circ\Psi_w(1,0,0),\quad c_w=P_1\circ\Psi_w(\frac 12,0,0).\]

\begin{lemma}\label{lemma81}
There are $\alpha\in(0,1)$ and $C>0$ depending only on $K$ such that  
\[
\lambda_{m}\leq C\cdot\alpha^nN^n\cdot\lambda_{m-n},\quad\forall m>n\geq 1.
\]
\end{lemma}
\begin{proof}
In fact, we take $\alpha=\frac{k}{4k-4}$ in the lemma, and will show that $R_{n,F}\leq \alpha^n$, where $R_{n,F}$ is defined in Corollary \ref{coro72}. The lemma then follows from the fact that $\lambda_m\lesssim \lambda_n\cdot \lambda_{m-n}$ by Proposition \ref{prop32} and Lemma \ref{lemma33} and the fact that $\lambda_{n}\lesssim N^n R_{n,F}$ by Corollary \ref{coro72}.
	
  Denote $A=\bigcup_{o\in \{2,3\},s\in \{0,1\}}\tilde{F}_{1,o,s}, B = \{w=w_1\cdots w_n\in W_n:w_i\in A \text{ for any }1 \leq i \leq n\}$ and $B_s = B\cap \tilde{F}_{n,1,s}$ for $s \in \{0,1\}$. Let $P_{2,3}$ be the orthogonal projection defined as $P_{2,3}(x) = (x_2,x_3)$ for $x = (x_1,x_2,x_3)\in \mathbb{R}^3$. For each $w\in B_0$, we denote $B_w =\big \{v\in B:P_{2,3}(\Psi_vK) = P_{2,3}(\Psi_wK)\big\}$. Clearly $B_w\cap B_1 = \big\{\tilde{\Gamma}_{[1]}(w)\big\}$ for each $w\in B_0$. Then, for any $w\in B_0$ and $g\in l(B_w)$ with $g(w) = 0,g(\tilde{\Gamma}_{[1]}(w)) = 1$, we have $\mathcal{D}_{n,B_w}(g) \geq k^{-n}$. Noticing that $B_w\cap B_v = \emptyset$ for any $w\neq v\in B_0$, for any $f\in l(W_n)$ with $f|_{\tilde{F}_{n,1,0}} = 0,f|_{\tilde{F}_{n,1,1}} = 1$, we then have $\mathcal{D}_{n}(f) \geq \sum_{w\in B_0}\mathcal{D}_{n,B_w}(f|_{B_w})\geq \#B_0\cdot k^{-n} = \big(\frac{4k-4}{k}\big)^n=\alpha^{-n}$. This implies the desired inequality $R_{n,F}\leq \alpha^n$.
\end{proof}

\begin{lemma}\label{lemma82}
There is a sequence of functions $\big\{g_m\in l(\tilde{F}_{1,3,0}\cdot W_{m-1})\big\}_{m\geq 1}$ satisfying the following requirements:

{(a).} for $m\geq 1$, $g_m = g_m\circ \tilde{\Gamma}_{[2]}$, $0 \leq g_m \leq 1$, and
\[g_m(w) = \begin{cases}
           0, & \mbox{if } w\in \tilde{F}_{1,3,0}\cdot W_{m-1}\cap \{a_w = 0\}, \\
           1, & \mbox{if } w\in \tilde{F}_{1,3,0}\cdot W_{m-1}\cap \{b_w = 1\}; \\
         \end{cases}\]

{(b).} for $ m\geq 1$, $\mathcal{D}_{m,\tilde{F}_{1,3,0}\cdot W_{m-1}}(g_m) \leq CN^{m}\lambda_m^{-1}$ for some $C>0$ independent of $m$;

{(c).} for $ m\geq 2$, $g_{m}(i\cdot w) = a_i + \frac{g_{m-1}(w)}{k},\forall i\in \tilde{F}_{1,3,0},\forall w\in \tilde{F}_{1,3,0}\cdot\tilde{F}_{m-2,3,1}$;

{(d).} for $ m\geq 2$,
\[g_m(w) =
           \frac lk,  \mbox{if } w\in \tilde{F}_{2,3,0}\cdot W_{m-2}\cap \{a_w = \frac lk\textit{ or } b_w = \frac lk\} \textit{ for } l=1,\cdots, k-1.
      \]
\end{lemma}

\begin{figure}[h]
	\centering
	\includegraphics[width=6cm]{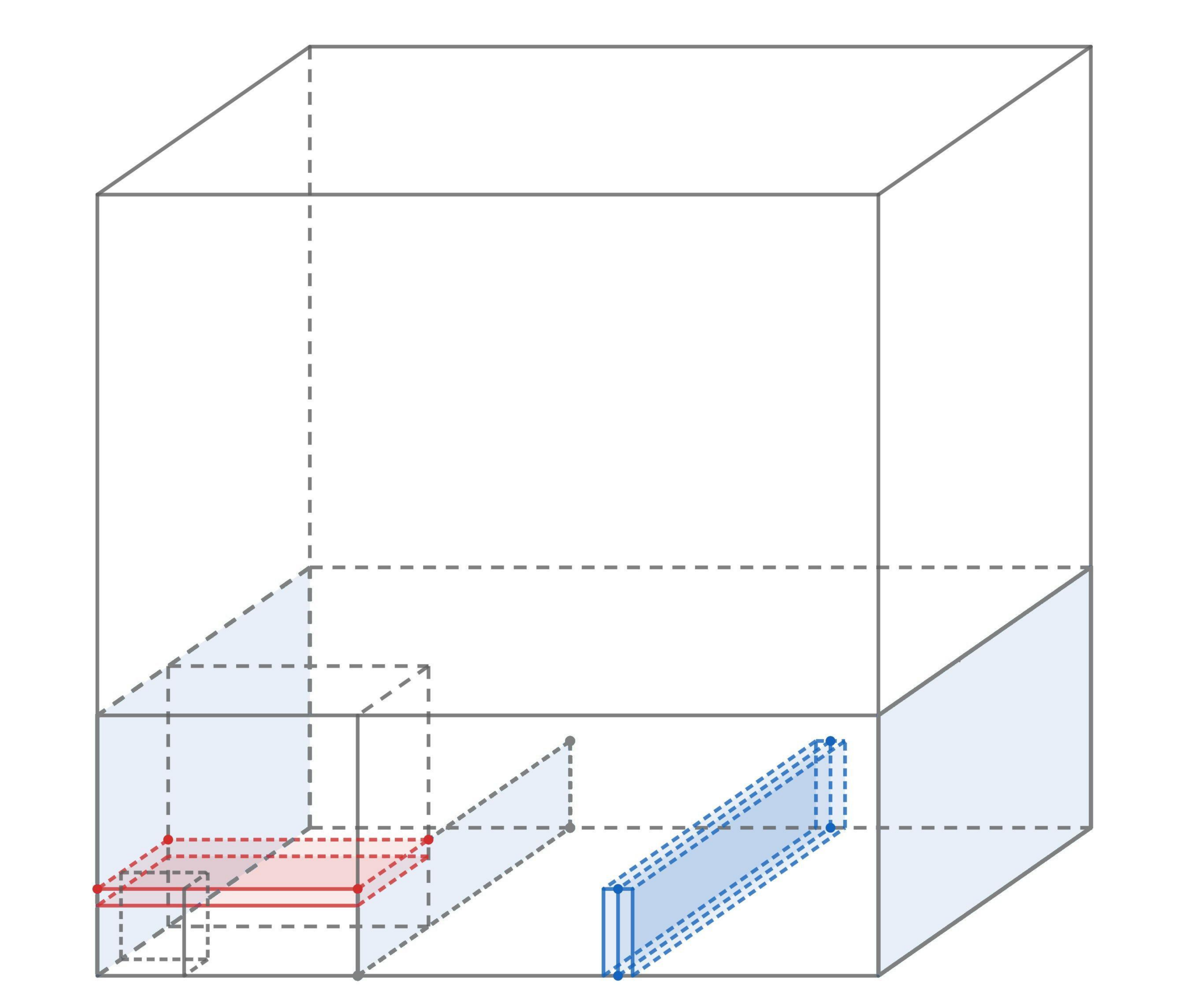}
	\begin{picture}(0,0)
	\put(-225,20){$\tilde F_{1,3,0}\cdot W_{m-1}$}\put(-100,80){$W_m$}
	\end{picture}
	\begin{center}
		\caption{The red part represents $i\cdot \tilde F_{1,3,0}\cdot\tilde F_{m-2,3,1}$ in (c), the blue part represents the part mentioned in (d).}
	\end{center}
\end{figure}
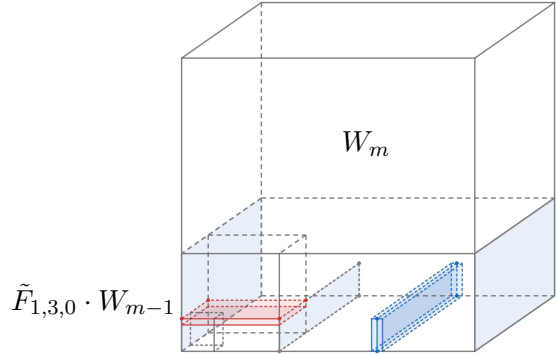

\begin{proof}
  First, for each $m\geq 1$, let $h_m,h'_m$ be two functions in $ l(W_m)$, so that $h_m|_{\tilde{F}_{m,1,0}} = 0,h_m|_{\tilde{F}_{m,1,1}} = 1$, $\mathcal{D}_{m}(h_m) = R_m(\tilde{F}_{m,1,0},\tilde{F}_{m,1,1})^{-1}$; and $h'_m|_{\tilde{F}_{1,3,0}\cdot W_{m-1}} = 0,h'_m|_{\tilde{F}_{m,3,1}} = 1$, $\mathcal{D}_{m}(h'_m) = R_{m}(\tilde{F}_{1,3,0}\cdot W_{m-1},\tilde{F}_{m,3,1})^{-1}$.
  We define $g_1\in l(\tilde F_{1,3,0})$ by $g_1 = h_1|_{\tilde F_{1,3,0}}$, and for $m\geq 2$, define $g_m\in l(\tilde{F}_{1,3,0}\cdot W_{m-1})$ by
  \[g_m(i\cdot w) = h_m(i\cdot w)\cdot h'_{m-1}(w) + (a_i + \frac {h_{m-1}(w)}{k})\cdot(1 - h'_{m-1}(w)),\quad\forall i\in \tilde{F}_{1,3,0},\forall w\in W_{m-1}.\]

Next,  the requirement {(a)} follows from $h_m = h_m\circ\tilde{\Gamma}_{[2]}, h'_m = h'_m\circ\tilde{\Gamma}_{[2]}$, $0 \leq h_m,h'_m \leq 1$, and the boundary values of $h_m$. Also, it is direct to check that the requirements {(c), (d)} are satisfied. So it only remains to estimate the energy of $g_m$. Note that by Corollary \ref{coro72}, we have $\mathcal{D}_m(h_m) \leq C'N^{m}\lambda_m^{-1}$ and $\mathcal{D}_{m}(h'_m) \leq C'N^m\lambda_m^{-1}$ for some $C'>0$ independent of $m$. Thus by $\lambda_m\asymp\lambda_{m+1}$ and $0 \leq h_m,h'_m \leq 1$, we have
\[
\mathcal{D}_{m,\tilde{F}_{1,3,0}\cdot W_{m-1}}(g_m)\leq 4C'\big(N^m\lambda_m^{-1} + N^{m-1}\lambda_{m-1}^{-1} + 2k^2N^{m-1}\lambda_{m-1}^{-1}\big)\leq CN^m\lambda_m^{-1}
\]
  for some $C > 0$ independent of $m$, which gives the requirement {(b)}.
\end{proof}

\begin{lemma}\label{lemma83}
For $n\geq 1$, there is a function $f_n\in l(W_n)$ so that $\mathcal{D}_{n}(f_n) \leq CN^{n}\lambda_n^{-1}$ for some $C>0$ independent of $n$, and $f_n(w)=c_w$ for all $w\in \partial W_n$.
\end{lemma}
\begin{proof}
  First, we define a function $f'_n\in l(W_n)$ by induction. First let $f'_{n,0}\in l(W_n)$ with $f'_{n,0}|_{\tilde{F}_{n,1,0}} = 0, f'_{n,0}|_{\tilde{F}_{n,1,1}} = 1$ and $\mathcal{D}_{n}(f'_{n,0}) = R_{n,F}^{-1}$. Now for each $1\leq m\leq n$, we define $f'_{n,m}$ inductively by
  \[f'_{n,m}(w\cdot \tau) = \begin{cases}
      a_w + \frac{g_{n-m+1}(\tau)}{k^{m-1}}, & \mbox{if }w\in \tilde{F}_{m-1,3,0} \mbox{ and }\tau\in \tilde{F}_{1,3,0}\cdot W_{n-m}, \\
      f'_{n,m-1}(w\cdot \tau), & \mbox{otherwise},
    \end{cases}\]
  where $g_m$'s are the same in Lemma \ref{lemma82}. At the end, we define $f'_n = \frac{k^n - 1}{k^n}\cdot f'_{n,n} + \frac{1}{2k^{n}}$.

  For each $1 \leq m \leq n$, we have
  \[\begin{aligned}&\mathcal{D}_n(f'_{n,m}) \leq\mathcal{D}_n(f'_{n,m-1}) + \mathcal{D}_{n,\tilde{F}_{m,3,0}\cdot W_{n-m}}(f'_{n,m}) \\
  \leq &\mathcal{D}_n(f'_{n,m-1}) + k^{2(m-1)}\cdot\mathcal{D}_{n-m+1,\tilde{F}_{1,3,0}\cdot W_{n-m}}\big(\frac{g_{n-m+1}}{k^{m-1}}\big) \leq \mathcal{D}_n(f'_{n,m-1}) + C_1N^{n-m+1}\lambda_{n-m+1}^{-1}
  \end{aligned}\]
  for some $C_1 > 0$ independent of $m,n$, where the first inequality follows from the requirements {(a),(c),(d)} of $\{g_m\}_{m\geq 1}$ and the last inequality follows from the requirement {(b)} of $\{g_m\}_{m\geq 1}$. By summing up the above estimate over $1\leq m\leq n$, and using Corollary \ref{coro72}, we have
  \[
  \mathcal{D}_{n}(f_{n,n}') \leq \mathcal{D}_n(f_{n,0}') + C_1\sum_{m = 1}^{n}N^{n-m+1}\lambda_{n-m+1}^{-1} \leq C_2\sum_{m = 1}^{n}N^{m}\lambda_{m}^{-1}
\]
  for some $C_2 > 0$ independent of $n$. On the other hand, by Lemma \ref{lemma81}, we have $N^m\lambda_m^{-1}\leq C_3\alpha^{n-m}N^n\lambda_n^{-1},\forall 1\leq m<n$ for some $C_3 > 0,\alpha\in(0,1)$ independent of $n$. Combining the above two estimates, we have
\begin{equation}\label{eqn81}
\mathcal{D}_n(f'_n) \leq \mathcal{D}_n(f'_{n,n})\leq C_4N^{n}\lambda_n^{-1}
\end{equation}
 for some $C_4>0$ independent of $n$.

  Next, note that $f'_n|_{\tilde{F}_{n,1,0}} = \frac{1}{2k^{n}},f'_n|_{\tilde{F}_{n,1,1}} = 1 - \frac{1}{2k^n}$ and for any $w\in \tilde{F}_{n-2,3,0}$,
  \begin{equation}\label{eqn82}
  a_w - \frac{1}{2k^n} \leq a_w + \frac{1}{k^n}(\frac{1}{2} - a_w) \leq f'_n|_{w\cdot \tilde{F}_{1,3,0}\cdot W_1} \leq b_w + \frac{1}{k^n}(\frac{1}{2} - b_w) \leq b_w + \frac{1}{2k^n}.
  \end{equation}
   We define $f''_n\in l(W_n)$ by
  \[f''_n(v) = \begin{cases}
      f'_n(v), & \mbox{if } v\notin \tilde{F}_{n,3,0}, \\
      c_v, & \mbox{otherwise}.
    \end{cases}\]
  Since for each $v = w\cdot w'\in \tilde F_{n,3,0}$ with $w\in \tilde F_{n-2,3,0}$, it satisfies $a_w \leq c_v \leq b_w$, together with (\ref{eqn82}), we have
    \[
    \mathcal{D}_n(f''_n)\leq\mathcal D_n(f_n')+\sum_{v\in\tilde F_{n,3,0}}\sum_{v'\sn v}(f''_n(v')-c_v)^2 \leq \mathcal{D}_n(f'_n) + C_5\sum_{v\in \tilde{F}_{n,3,0}}(\frac{2}{k^{n-2}})^2,
    \]
    for some $C_5>0$ independent of $n$.
    It then follows from (\ref{eqn81}),
    \[\mathcal{D}_n(f''_n)\leq C_4N^n\lambda_n^{-1} + 4C_5k^4\leq C_6N^n\lambda_n^{-1}\]
    for some $C_6>0$ independent of $n$ since $N^{n}\lambda^{-1}_n\geq C_3\alpha^{-n}\gg 1$.

    Finally, for each $o\in \{2,3\},s\in \{0,1\}$ we denote $A_{n,o,s} = \big\{w\in W_n:\mathrm{dist}(\Psi_wK,F_{o,s}) = \mathrm{dist}(\Psi_wK,\bigcup_{o\in\{2,3\},s\in\{0,1\}}F_{o,s})\big\}$, and define $f_n\in l(W_n)$ by
    \[f_n(w) = \begin{cases}
                 f''_n(w), & \mbox{if } w\in A_{n,3,0}, \\
                 f''_n\circ\tilde{\Gamma}_{[3]}(w), & \mbox{if } w\in A_{n,3,1}, \\
                 f''_n\circ\tilde{\Gamma}_{[2,3]}(w), & \mbox{if } w\in A_{n,2,0},\\
                 f''_n\circ\tilde{\Gamma}_{[2,3]}\circ\tilde{\Gamma}_{[2]}(w), & \mbox{if } w\in A_{n,2,1},
               \end{cases}\]
   which gives the desired function.
\end{proof}

\begin{proof}[Proof of Theorem \ref{thm81}. ]
 For $n\geq 1$, let $f_{n,1}=f_n$, $f_{n,2}=f_{n}\circ \tilde \Gamma_{[1,2]}$ and  $f_{n,3}=f_{n}\circ \tilde \Gamma_{[1,3]}$, where $f_n$ is the same function in Lemma \ref{lemma83}. Let $g_{n,i}=1-f_{n,i}$ for $i=1,2,3$.
  Let $c_*, L_*$ be the same constants in (\textbf{A3}).

 Now for $w\in W_*$ with $|w|=m$, for each $o\in \{1,2,3\} $, $s\in\{0,1\}$, we define $\varphi_{w,n,o},\psi_{w,n,o}\in l(W_{m + n})$, so that for each $v\in W_{m}$ and $\tau\in W_n$,
  \[\begin{aligned}&\varphi_{w,n,o}(v\cdot \tau) =(-1)\vee\frac{\sqrt{3}}{c_*}\big(f_{n,o}(\tau) + k^{m}\cdot P_o\circ\Psi_v(0,0,0) - k^{m}\cdot P_o\circ\Psi_w(0,0,0)\big),\\
  &\psi_{w,n,o}(v\cdot \tau) =(-1)\vee\frac{\sqrt{3}}{c_*}\big(g_{n,o}(\tau) - k^{m}\cdot P_o\circ\Psi_v(0,0,0) + k^{m}\cdot P_o\circ\Psi_w(0,0,0)\big).
  \end{aligned}\]
 Let\[\varphi_{w,n} =0\wedge \min\big\{\varphi_{w,n,o},\psi_{w,n,o}:o\in \{1,2,3\}\big\} + 1.\]
  Then by condition (\textbf{A3}), it is direct to check that $\varphi_{w,n}$ is supported on $\mathcal{N}_w\cdot W_n$ and $\varphi_{w,n}|_{w\cdot W_n} = 1$.  Also by Lemma \ref{lemma83},
  \[\mathcal{D}_{m+n}(\varphi_{w,n}) \leq 6\big(\frac{3}{c_*^2}(2L_*+1)^3\mathcal{D}_n(f_n) + 27(2L_*+1)^3\cdot 9k^{2n}\cdot(k^{m+n})^{-2}\big) \leq CN^n\lambda_n^{-1}\]
  for some $C' > 0$ independent of $n$, where $(2L_*+1)^3$ is an upper bound of $\#\mathcal N_w$. This gives that $R_n(w\cdot W_n,\mathcal{N}_w^c\cdot W_n) \geq C'^{-1}N^{-n}\lambda_n$ for any $w\in W_*$ and $n\geq 1$, and the estimate  $R_n \geq CN^{-n} \lambda_n$ follows.
  
  Finally, the estimate that $R_n\asymp N^{-n}\lambda_n\asymp N^{-n}\sigma_n\asymp \rho^{-n}$ with $\rho>1$ follows from Remark \ref{remark82} and Lemma \ref{lemma81}.
\end{proof}

\section{A limit form}\label{sec9}
In this section, we define a limit form on a $\USC^{(3)}$ $K$. The proof is based on studies on equivalence characterizations of sub-Gaussian heat kernel estimates \cite{AB,BB6,GHL2,HS,M2}, and in particular, Murugan's work \cite{M2} on simplification  for the $d_H<d_W+1$ case. For convenience (to see all the estimates are uniform), we will rescale all graphs $(W_m,E_m)$, and embed them into a same infinite graph. 

Throughout the following context,  for $m\geq 1$ we write $\Psi_1^{\circ(m)}$, $\Psi_1^{\circ(-m)}$ the $m$-times composition of $\Psi_1$, $\Psi_1^{-1}$ respectively. As before, we will always denote $d_H=\frac{\log N}{\log k}$ the Hausdorff dimension of $K$, and write $d_W:=-\frac{\log \rho}{\log k}+d_H$ (with $\rho$ being the same constant in Theorem \ref{thm81}), which is named as \textit{walk dimension} in various contents \cite{B, GHL1}. Note that \[2\leq d_W<d_H<d_W+1,\]
since $1<\rho\leq N/k^2<k$ (see Remark \ref{remark82}).
\vspace{0.2cm}

\noindent(\textbf{\emph{An infinite graph}}). Let
\[
V_m:=\big\{\Psi_1^{\circ(-m)}\circ \Psi_w(\frac{1}{2},\frac{1}{2},\frac{1}{2}):w\in W_m\big\},\qquad\forall m\geq 0,
\]
and write  $V:=\bigcup_{n=0}^\infty V_m$. It is easy to check that $V_0\subset V_1\subset V_2\subset\cdots$. For each $x,y\in V$, we say $x\sim y$ if and only if there is $m\geq 1$ and $w\sm w'$ such that $x=\Psi_1^{\circ(-m)}\circ \Psi_w(\frac{1}{2},\frac{1}{2},\frac{1}{2})$ and $y=\Psi_1^{\circ(-m)}\circ \Psi_{w'}(\frac{1}{2},\frac{1}{2},\frac{1}{2})$. We write $G=(V,\sim)$ the induced \textit{infinite graph} and denote $d_G$ the associated \textit{graph distance}, i.e.,  for $x,y\in V$,
\[d_{G}(x,y)=\min\{L: x=x_0, y=x_L \text{ for some } x_1, \cdots, x_{L-1} \text{ in } V \text{ such that }x_{i-1}\sim x_{i}, 1\leq i\leq L\}.\]
It is clear that $(V_m,\sim)$ is isomorphic to $(W_m,E_m)$ for each $m\geq 1$. For $x\in V$, $r>0$, we write $B_G(x,r):=\{y\in V: d_G(x,y)<r\}$ the \textit{ball} centered at $x$ with radius $r$.\vspace{0.2cm}

For $A\subset V$, we define the \textit{volume} of $A$ by
\[\varrho_G(A)=\sum_{x\in A}\deg(x),\]
where $\deg(x)=\#\{y\in V:y\sim x\}$. Finally, for each $f\in l(V)$ and $A\subset V$, we define
\[\mcD_A(f)=\frac{1}{2}\sum_{x\sim y:x,y\in A}\big(f(x)-f(y)\big)^2.\]

\begin{lemma}\label{lemma91}
There is some constant $C>0$ depending only on $K$, such that

(a). $G=(V,\sim)$ is Ahlfors regular, i.e.,
\[
C^{-1}r^{d_H}\leq \varrho_G\big(B_G(x,r)\big)\leq Cr^{d_H},\quad\forall x\in V,r\geq 1,
\]

(b). the Poincar\'e inequality holds, i.e., for each $f\in l(V)$, $x\in V$ and $r\geq 1$, we have
\[
\sum_{y\in B_G(x,r)}\big(f(y)-\fint_{B_G(x,r)} fd\varrho_G\big)^2\varrho_G(y)\leq Cr^{d_W}\mcD_{B_G(x,2r)}(f),
\]

(c). the capacity upper bound estimate holds, i.e., for each $x\in V$ and $r\geq 1$, there is $\psi\in l(V)$ such that $\psi|_{B_G(x,r)}\equiv 1$, $\psi|_{V\setminus B_G(x,2r)}=0$ and $\mcD_V(\psi)\leq Cr^{d_H-d_W}$.
\end{lemma}
\begin{proof}

For $A\subset V$, we write $\diam_G (A)=\sup\{d_G(x,y): x,y\in A\}$ the \textit{diameter} of $A$. It is easy to check that $k^n\leq \diam_G(V_n) \leq N+3N\sum_{i=1}^{n-1} k^{i}<C_1 k^n$ for $C_1:=\frac{3N}{k-1}$.

(a) follows from the observation that for $n\geq 0$, $k^n\leq r<k^{n+1}$, it always have $\varrho_G\big(B_G(x,r)\big)\asymp N^n$, $\forall x\in V$. 

To show (b) and (c), we introduce some notations as follows.   For each $w\in W_*$, $m\geq 0$, let $V_{m,w}=\Psi_1^{\circ (-(|w|+m))}\circ\Psi_w\circ \Psi_1^{\circ m}(V_m)\subset V$, and for $n\geq 0$, $A\subset W_n$, let $V_{m,A}:=\bigcup_{w'\in A}V_{m,w'}$. Noticing that $(V_{m,A},\sim)$ is isomorphic to $(A\cdot W_m,\stackrel{n+m}{\sim})$ in the graph sense with comparable volumes, all the conclusions from previous sections can be used here.

(b). If suffices to show the result for large $r$. Let us assume $r\geq 2C_1k$, and fix $m\geq 1$ so that $C_1k^m\leq r/2<C_1k^{m+1}$. We choose $n$ large enough so that $B_G(x,2r)\subset V_{n+m}$, and let $A=\{w\in W_n:V_{m,w}\cap B_G(x,r)\neq\emptyset\}$. Clearly, $V_{m,A}\subset B_G(x,3r/2)$ and $\#A\leq C_2$ for some $C_2>0$ depending only on $K$. Then,
\[
\begin{aligned}
&\sum_{y\in B_G(x,r)}\big(f(y)-\fint_{B_G(x,r)} fd\varrho_G\big)^2\varrho_G(y)\leq \sum_{y\in V_{m,A}}\big(f(y)-\fint_{V_{m,A}} fd\varrho_G\big)^2\varrho_G(y)\\
=&\sum_{w\in A}\varrho_G(V_{m,w})\big(\fint_{V_{m,w}} f\varrho_G-\fint_{V_{m,A}} fd\varrho_G\big)^2+\sum_{w\in A}\sum_{y\in V_{m,w}}\big(f(y)-\fint_{V_{m,w}} f\varrho_G\big)^2\varrho_G(y)\\
\leq& C_3(\sigma_m+\lambda_m)\mcD_{B_G(x,2r)}(f),
\end{aligned}
\]
for some $C_3>0$ depending only on $K$, where in the third line we use the definition of $\sigma_m$ and Lemma \ref{lemma28} (any two cells in $A$ are connected by a path contained in $\{w\in W_n:w\stackrel{n}\sim A\}$) to estimate the first term, and the definition of $\lambda_m$ to estimate the second term. The desired  estimate then follows from Theorem \ref{thm81}.

(c). Similarly to (b), for large $r$, we fix $m$ so that $C_1(L_*+1)k^m\leq r<C_1(L_*+1)k^{m+1}$, choose $n$ large enough so that $B_G(x,2r)\subset V_{n+m}$, and let $A=\{w\in W_n:V_{m,w}\cap B_G(x,r)\neq\emptyset\}$. Clearly, $\bigcup_{w\in A}V_{m,\mathcal N_w}\subset B_G(x,2r)$. For each $w\in A$, define $\psi_w\in l(V)$ as the tent function such that $0\leq \psi_w\leq 1$, $\psi_w|_{V_{m,w}}\equiv 1$, $\psi_w|_{V\setminus V_{m,\mathcal N_w}}\equiv 0$ and $\mcD_{V}(\psi_w)\leq R_m^{-1}$. Then $\psi:=\max\{\psi_w:w\in A\}$ is a desired cut-off function, noticing that $R_m\asymp r^{d_W-d_H}$ by Theorem \ref{thm81}.
\end{proof}

\noindent(\textbf{\emph{Cable systems}}). For each $x,y\in \mathbb{R}^3$, denote  $\overline{x,y}=\big\{(1-t)x+ty:t\in [0,1]\big\}$. Define
\[\Xi_m=\bigcup_{x,y\in V_m, x\sim y}\overline{x,y},\quad\forall m\geq 1,\]
and let $\Xi=\bigcup_{m=1}^\infty \Xi_m$. We assign a \textit{measure} $\varrho_\Xi$ to $\Xi$ such that $\varrho_\Xi\big(\{(1-t)x+ty:t\in [a,b]\}\big)=b-a$ for each pair $x\sim y$ and $[a,b]\subset [0,1]$, and define a \textit{metric} $d_\Xi$ on $\Xi$ by $d_\Xi(x,y)=\inf\varrho_\Xi\big(\gamma[0,1]\big)$ for $x,y\in \Xi$, where the infimum is taken over any continuous curve $\gamma:[0,1]\to \Xi$ such that $\gamma(0)=x$ and $\gamma(1)=y$. In particular, it is easy to see that $\varrho_\Xi\asymp\mathcal{H}^1$ and $d_\Xi\asymp d$ (using the property (\textbf{A3})) where $\mathcal H^1$ and $d$ stand for the $1$-dimensional Hausdorff measure and the Euclidean metric on $\mathbb R^3$ as we did before.

For $x\in\Xi$ and $r>0$, we write $B_\Xi(x,r):=\{y\in \Xi:d_\Xi(x,y)<r\}$ for the \textit{ball} centered at $x$ with radius $r$ in the $d_\Xi$ sense. In addition, we let $B_\Xi^{(m)}(x,r)=B_\Xi(x,r)\cap \Xi_m$ for $x\in \Xi_m$.\vspace{0.2cm}

\noindent(\textbf{\emph{Energies on cable systems}}). For each $x\sim y$, we let $\gamma_{x,y}:[0,1]\to\Xi$ be defined as $\gamma_{x,y}(t)=(1-t)x+ty$. We say that a function $f:\Xi\to \mathbb{R}$ is absolutely continuous (or Lipschitz) if for each pair $x\sim y$,
$f\circ \gamma_{x,y}:[0,1]\to\mathbb{R}$ is absolutely continuous (or Lipschitz). If $f$ is absolutely continuous, for each $x\sim y$ we define $|\nabla f|\big(\gamma_{x,y}(s)\big)=\big|(f\circ \gamma_{x,y})'(s)\big|$ for a.e. $s\in [0,1]$, and so $|\nabla f|$ is $\varrho_\Xi$-a.e. defined. Let $\tilde{\mcF}_0$ denote the space of
compactly supported Lipschitz functions on $\Xi$. By Rademacher's theorem, if $f\in \tilde{\mcF}_0$, then $f$ is absolutely continuous, $|\nabla f|$ is uniformly bounded and compactly supported, and therefore $\int_\Xi |\nabla f|^2d\varrho_\Xi<\infty$.  We define a Dirichlet form $(\tilde{\mcE},\tilde{\mcF})$ on $L^2(\Xi,\varrho_\Xi)$ by
\[\tilde \mcE(f,g):=\int_\Xi |\nabla f|\cdot |\nabla g|d\varrho_\Xi<\infty\]
for $f,g\in \tilde\mcF$ and write $\tilde \mcE(f):=\tilde\mcE(f,f)$ for short, where $\tilde\mcF$ is the completion of $\tilde \mcF_0$ with respect to the norm $\|f\|_{\tilde{\mcE}_1}:=\sqrt{\tilde \mcE(f)+\|f\|^2_{L^2(\Xi,\varrho_\Xi)}}$.

For $m\geq 1$, we define $(\tilde{\mcE}^{(m)},\tilde{\mcF}^{(m)})$ on $\Xi_m$ in a same way.

Readers are suggested to refer to \cite{BB6} for formal definitions and detailed discussions about the cable system. \vspace{0.2cm}

The following Theorems \ref{thm92} and \ref{thm94} are now consequences of Lemma \ref{lemma91}, thanks to the deep studies on equivalence characterizations of heat kernel estimates.

\begin{theorem}\label{thm92}
Let $\tilde{p}_t(x,y)$ be the heat kernel associated with $(\tilde\mcE, \tilde \mcF)$. It is  reversible with respect to $\varrho_\Xi$ and satisfies the following non-standard Gaussian-type estimate for some constants $c_1$-$c_4>0$,
\begin{equation}\label{eqn91}
\frac{c_1}{t^{1/2}\vee t^{d_H/d_W}}e^{-c_2\cdot G_\Xi\big(t,d_\Xi(x,y)\big)}\leq \tilde{p}_t(x,y)\leq \frac{c_3}{t^{1/2}\vee t^{d_H/d_W}}e^{-c_4\cdot G_\Xi\big(t,d_\Xi(x,y)\big)},
\end{equation}
where
\[G_\Xi(t,r)=\begin{cases}
	r^2/t,&\text{ if }t<r,\\
	(r^{d_W}/t)^{1/(d_W-1)},&\text{ if }t\geq r.
\end{cases}\]

In addition, for $m\geq 1$ and $t<k^{md_W}$, we have the estimate (\ref{eqn91}) holds for the heat kernel $\tilde{p}^{(m)}_t(x,y)$ associated with $(\tilde{\mcE}^{(m)},\tilde{\mcF}^{(m)})$ as well.
\end{theorem}

\begin{remark}
Let \[\phi(t)=
	\begin{cases}
		t^{1/2}	,&\text{ if }t<1,\\
		t^{1/d_W}, &\text{ if }t\geq 1.
	\end{cases}\]
Then $t^{1/2}\vee t^{d_H/d_W}\asymp \varrho_\Xi(B_\Xi\big(x,\phi(t))\big)$ for any $x\in \Xi,t>0$, also it holds for any $m\geq 1$, $x\in \Xi_m$ and $t\in (0,k^{md_W})$. In addition, $G_\Xi(t,r)$ is properly associated with $\phi(t)$. See \cite[Section 5.1]{GT} or  \cite[Example 3.18]{HS}.
\end{remark}

\noindent(\textbf{\emph{Weak solution}}). Let $I$ be an open interval in $\mathbb{R}$,  and $\Omega$ be an open subset of $\Xi$. We say that a function $u:I\times \Omega\to\mathbb{R}$ is a weak solution to the heat equation (with respect to $(\tilde{\mcE},\tilde{\mcF})$) if  $\frac{\partial u}{\partial t}$ exists in $L^2(\Omega,\varrho_\Xi)$ for each $t\in I$, and for any non-negative $\psi\in \tilde\mcF\cap C_c(\Omega)$,
\[\langle\frac{\partial u}{\partial t},\psi\rangle_{L^2(\Omega,\varrho_\Xi)}+\tilde{\mcE}(u,\psi)=0,\]
 where $\langle f,g\rangle_{L^2(\Omega,\varrho_\Xi)}:=\int_\Omega f\cdot gd\varrho_\Xi$, $\forall f,g\in L^2(\Omega, \varrho_\Xi)$, and $C_c(\Omega)$ is the space of continuous functions on $\Omega$ with compact support.

For $m\geq 1$, we define a weak solution to the heat equation with respect to $(\tilde{\mcE}^{(m)},\tilde{\mcF}^{(m)})$ in a same way.\vspace{0.2cm}

\noindent(\textbf{\emph{Parabolic Harnack inequality (PHI)$_\phi$}}). We say that $(\tilde{\mcE},\tilde{\mcF})$ satisfies a $\phi$-parabolic Harnack inequality if there exists a constant $C>0$ such that, for any real $t>s>0$ and $x\in \Xi$, and any non-negative weak solution $u$ of the heat equation on $Q_{t,x,s}:=(t-s,t)\times B_\Xi\big(x,\phi(s)\big)$, we have
\[\text{ess inf}_{Q_{t,x,s/4}}u\geq C\text{ess sup}_{Q_{t-s/2,x,s/4}}u.\]

Similarly, we say that $(\tilde{\mcE}^{(m)},\tilde{\mcF}^{(m)})$ satisfies a $\phi$-parabolic Harnack inequality if there exists a constant $C>0$ such that, for any real $k^{md_W}>t>s>0$ and $x\in \Xi_m$, and any non-negative weak solution $u$ of the heat equation on $Q^{(m)}_{t,x,s}:=(t-s,t)\times B^{(m)}_\Xi\big(x,\phi(s)\big)$, we have
\[\text{ess inf}_{Q^{(m)}_{t,x,s/4}}u\geq C\text{ess sup}_{Q^{(m)}_{t-s/2,x,s/4}}u.\]

\begin{theorem}\label{thm94}
$(\tilde{\mcE},\tilde{\mcF})$ satisfies (PHI)$_{\phi}$, and for $m\geq 1$, $(\tilde{\mcE}^{(m)},\tilde{\mcF}^{(m)})$ also satisfies (PHI)$_{\phi}$. In addition, the constant $C>0$ in the definition of  (PHI)$_{\phi}$ can be taken to be the same.
\end{theorem}

In fact, by  \cite[Thereom 1.1]{M2} and by Lemma \ref{lemma91}, there is a random walk with a sub-Gaussian heat kernel estimate on $(V,\sim)$. By \cite{BB6}, the cable system satisfies the cut-off Sobolev inequality (\cite[Proposition 3.3]{BB6}) and the Poincar\'e inequality (\cite[Proposition 3.5]{BB6}), whence by \cite{AB} or \cite{GHL2}, we have Theorem \ref{thm92} holds for $(\tilde{\mcE},\tilde{\mcF})$ and $(\tilde{\mcE}^{(m)},\tilde{\mcF}^{(m)})$.

On the other hand, by  \cite[Theorem 5.3]{HS}, the heat kernel estimate (\ref{eqn91}) holds if and only if (PHI)$_\phi$ holds. Then one can easily see that (PHI)$_\phi$ naturally holds on $(\tilde{\mcE},\tilde{\mcF})$ and $(\tilde{\mcE}^{(m)},\tilde{\mcF}^{(m)})$, so Theorem \ref{thm94} holds. \vspace{0.2cm} 

In the remaining of this section, we will take a pointwise (sub-sequential) limit of the renormalized kernels
\begin{equation}\label{eqn92}
\bar p_m(t,x,y):=k^{md_H}\tilde{p}^{(m)}_{k^{md_W}\cdot t}(\Psi_1^{\circ(-m)}x,\Psi_1^{\circ(-m)}y),\quad \forall t\in \mathbb{R}_+,\ x,y\in \Psi_1^{\circ (m)}\Xi_m.
\end{equation}
We will see that $\bar p_m(t,x,y)$ is jointly continuous, and $\Psi_{1}^{\circ(m)}\Xi_m$ converges to $K$ as $m\to \infty$ with respect to the Hausdorff metric. Here, for two compact sets $A,B$ in $\mathbb{R}^3$, we will denote $d_{\mathcal{H}}(A,B)$ the \textit{Hausdorff metric} between them, i.e., $d_{\mathcal H}(A,B):=\inf\big\{r>0:B\subset\bigcup_{x\in A}B(x,r),A\subset\bigcup_{x\in B}B(x,r)\big\}$, where $B(x,r)$ is the Euclidean ball centered at $x$ with radius $r$.

\begin{lemma}\label{lemma94}
Let $m\geq 1$, $t,s\in(0,k^{md_W})$ and $x,y_1,y_2\in \Xi_m$. Also, write $t_0=t\vee s$. Then
\[
\big|\tilde{p}^{(m)}_t(x,y_1)-\tilde{p}^{(m)}_s(x,y_2)\big|\leq \frac{C}{t_0^{1/2}\vee t_0^{d_H/d_W}}\cdot \big(\frac{|t-s|\vee d(y_1,y_2)^{d_W}\vee d(y_1,y_2)^2}{t_0}\big)^\beta,
\]
for some positive constants $C$ and $\beta$.
\end{lemma}
\begin{proof}
The lemma is a consequence of the heat kernel estimate (\ref{eqn91}). For $\delta\in(0,t)$ and $z\in\Xi_m$, we write $Q^{(m)}_{t,z,\delta}=(t-\delta,t)\times B^{(m)}_\Xi\big(z,\phi(\delta)\big)$ and write
\[
\text{Osc}^{(m)}_{t,z}(\delta):=\sup\big\{\big|\tilde{p}^{(m)}_{t_1}(z,z_1)-\tilde{p}^{(m)}_{t_2}(z,z_2)\big|:(t_1,z_1)\in Q^{(m)}_{t,z,\delta},(t_2,z_2)\in Q^{(m)}_{t,z,\delta}\big\}.
\]
Then, it follows from (PHI)$_\phi$ that
\[
\text{Osc}^{(m)}_{t,z}(\delta/4)\leq (1-c)\text{Osc}^{(m)}_{t,z}(\delta),\quad\forall t>\delta>0,
\]
for some constant $c\in (0,1)$ independent of $m,t,\delta,z$.

Let's assume $|t-s|$ is small compared with $t_0$ and $d(y_1,y_2)$ are small compared with $\phi(t_0)$, since otherwise the estimate trivially follows from the upper bound estimate in (\ref{eqn91}). Without loss of generality, we assume $t>s$, and we let $\delta=\max\big\{t-s,\phi^{-1}(d(y_1,y_2))\big\}=|t-s|\vee d(y_1,y_2)^{d_W}\vee d(y_1,y_2)^{2}$. Then
\[
\big|\tilde{p}^{(m)}_t(x,y_1)-\tilde{p}^{(m)}_s(x,y_2)\big|\leq \text{Osc}^{(m)}_{t,y_1}(\delta)\leq (1-c)\text{Osc}^{(m)}_{t,y_1}(4\delta)\leq\cdots\leq (1-c)^{M}\text{Osc}^{(m)}_{t,y_1}(4^M\delta),
\]
where we choose $M\in\mathbb{N}$ so that $t_0/8<4^M\delta\leq t_0/2$, whence
\[
\big|\tilde{p}^{(m)}_t(x,y_1)-\tilde{p}^{(m)}_s(x,y_2)\big|\lesssim (\frac{\delta}{t_0})^\beta\text{Osc}^{(m)}_{t,y_1}(4^M\delta),
\]
for some $\beta>0$. The lemma follows immediately noticing that $\text{Osc}^{(m)}_{t,y_1}(4^M\delta)\lesssim (t_0^{1/2}\vee t_0^{d_H/d_W})^{-1}$ by the heat kernel upper bound estimate in (\ref{eqn91}).
\end{proof}

\begin{corollary}\label{coro95}
For fixed $0<s_1<s_2<\infty$,  there exist positive constants $C,\beta$ depending only $s_1,s_2$ such that
\[
\big|\bar p_m(t,x,y)-\bar p_m(t',x',y')\big|\leq C\big(|t-t'|\vee d(x,x')\vee d(y,y')\big)^\beta,
\]
for any large enough $m\geq 1$, $t,t'\in [s_1,s_2]$ and $x,y,x',y'\in \Psi_1^{\circ(m)}\Xi_m$.
\end{corollary}
\begin{proof}
By Lemma \ref{lemma94}, for large enough $m$, we have
\[
\begin{aligned}
&\big|\bar p_m(t,x,y)-\bar p_m(t',x',y')\big|\\
=&k^{md_H}\big|\tilde p^{(m)}_{k^{md_W}t}(\Psi_1^{\circ(-m)}x,\Psi_1^{\circ(-m)}y)-\tilde p^{(m)}_{k^{md_W}t'}(\Psi_1^{\circ(-m)}x',\Psi_1^{\circ(-m)}y')\big|\\
\leq& \frac{C_1 }{t_0^{d_H/d_W}}\cdot \big(\frac{(|t-t'|\vee d(y,y')^{d_W}\vee \big(k^{(2-d_W)m}d(y,y')^2\big)\vee d(x,x')^{d_W}\vee \big(k^{(2-d_W)m}d(x,x')^2\big)}{t_0}\big)^\beta\\
\leq & C_2 \big({|t-t'|\vee d(y,y')^{d_W}\vee d(y,y')^2\vee d(x,x')^{d_W}\vee d(x,x')^2}\big)^\beta\\
\leq & C\big(|t-t'|\vee d(x,x')\vee d(y,y')\big)^\beta
\end{aligned}
\]
for some positive constants $C_1, C_2, C$, $\beta$, where $t_0=t\vee t'$, noticing that $d_W\geq 2$. 
\end{proof}

Finally, we need a lemma concerning the convergence of the underlying spaces.
\begin{lemma}\label{lemma96}
(a). $\lim_{m\to\infty}d_\mathcal{H}(\Psi_1^{\circ(m)}\Xi_m,K)=0$.

(b). For each $m\geq 1$, define a measure $\varrho_m$ on $\Psi_1^{\circ(m)}\Xi_m$ by
\[\varrho_m(A)=k^{-md_H}\varrho_\Xi(\Psi_1^{\circ(-m)}A),\quad\forall A\subset \Psi_1^{\circ(m)}\Xi_m.\]
Then, there exists $C>0$ such that
\[
	\varrho_m\Rightarrow_{\square} C\mu, \quad \text{ as }m\to\infty,
\]
where $\mu$ is the normalized Hausdorff measure on $K$, and
$\Rightarrow_{\square}$ means the weak convergence of Radon measures on $\square$.
\end{lemma}

\noindent({\textbf{\emph{The convergence ``$\rightarrowtail$''}}}). Let $(B,d)$ be a metric space, let $A_n\subset B,n\geq 1$ and $A\subset B$ be closed subsets, and assume $A_n$ converges to $A$ in the Hausdorff metric. Also, let $f_n\in C(A_n),n\geq 1$ and $f\in C(A)$, we say $f_n\rightarrowtail f$ if $f_n(x_n)\to f(x)$ for any sequence $x_n\in A_n,n\geq 1$ such that $x_n\to x$.

It is not hard to see some useful facts about ``$\rightarrowtail$'' when $B$ is compact (see  \cite[Section 2.1]{C}).

(\textit{\textbf{Fact 1}}). Let $f_n\in C(A_n),n\geq 1$ be a sequence of continuous functions, that is uniformly bounded (i.e. $\sup_{n\geq 1}\|f\|_{L^\infty(A_n)}<\infty$) and is equicontinuous (i.e. there is an increasing function $\rho:[0,\infty)\to [0,\infty)$ such that $\rho(0)=0$, $\rho$ is continuous at $0$ and $|f_n(x)-f_n(y)|\leq \rho\big(d(x,y)\big),\forall n\geq 1,x,y\in A_n$), then we can find a subsequence $\{n_l\}_{l\geq 1}$ and $f\in C(A)$ such that $f_{n_l}\rightarrowtail f$ as $l\to \infty$.

(\textit{\textbf{Fact 2}}). Let $\nu_n$, $n\geq 1$ be a sequence of Radon measures on $A_n$, let $\nu$ be a Radon measure on $A$, and assume $\nu_n\Rightarrow_B \nu$, where $\nu_n\Rightarrow_B \nu$ means the weak convergence of measures on $B$. Then
\[\int_{A_n}f_nd\nu_n\to \int_A fd\nu,\]
as long as $f_n\rightarrowtail f$, $f_n\in C(A_n)$, $f\in C(A)$.

\begin{theorem}\label{thm97}
There is a subsequence $\{m_l\}_{l\geq 1}$ and $\hat p:\mathbb{R}_+\times K^2\to\mathbb{R}$ such that $\bar p_{m_l}\rightarrowtail \hat p$ as $l\to\infty$. In addition, let $C_1$ be the constant appearing in Lemma \ref{lemma96}, and set  $\bar{p}_t(x,y)=C_1\hat p(t,x,y)$. Then $\bar{p}_t(x,y)$ is an integration kernel of a Markov semi-group with respect to $\mu$, satisfying a sub-Gaussian heat kernel estimate
\begin{equation}\label{eqn93}
\frac{c_1}{t^{d_H/d_W}}e^{-c_2(\frac{d(x,y)^{d_W}}{t})^{\frac{1}{d_W-1}}}\leq \bar{p}_t(x,y)\leq \frac{c_3}{t^{d_H/d_W}}e^{-c_4(\frac{d(x,y)^{d_W}}{t})^{\frac{1}{d_W-1}}},
\end{equation}
for $0<t<1$, where $c_1$-$c_4>0$ are constants depending only on $K$.
\end{theorem}
\begin{proof}
 By (\textbf{Fact 1}), Corollary \ref{coro95} and the upper bound estimate (\ref{eqn91}), using a diagonal argument, we can find a subsequence $\{m_l\}_{l\geq 1}$ such that $\bar p_{m_l}\rightarrowtail \hat p$ for some $\hat p\in C(\mathbb{R}_+\times K^2)$.

 In the following, without loss of generality, we assume $\bar p_m\rightarrowtail \hat p$ as $m\to \infty$. Notice that for any $t>0$, $x_m\to x$, $y_m\to y$, $\frac{k^{md_H}}{(k^{md_W}t)^{1/2}\vee (k^{md_W}t)^{d_H/d_W}}$ converges to $\frac{1}{t^{d_H/d_W}}$, and  $G_\Xi\big(k^{md_W}t,d_\Xi(\Psi_1^{\circ(-m)}x_m,\Psi_{1}^{\circ(-m)}y_m)\big)$ converges to $(\frac{d(x,y)^{d_W}}{t})^{\frac{1}{d_W-1}}$. The estimates (\ref{eqn93}) follows from Theorem \ref{thm92} with $c_1$-$c_4$ being the same constants in (\ref{eqn91}).

 Finally, we check that $\bar{p}_t(x,y)\mu(dy)=C_1\hat p(t,x,y)\mu(dy), t>0$ form a Markov semigroup. In fact, by (\textbf{Fact 2}) and Lemma \ref{lemma96} (b), it is easy to see that $\int_K \bar{p}_t(x,y)\mu(dy)=1,\forall t>0,x\in K$, and $\bar{p}_t(x,y)=\int_K \bar{p}_s(x,z)\bar{p}_{t-s}(z,y)\mu(dz),\forall t>s>0,x,y\in K$.
\end{proof}

\begin{remark} $\{\bar{p}_t(x,y)\mu(dy)\}_{t>0}$ is a Feller semigroup (by the H\"older continuity of the kernel), so there is an associated Hunt process (see \cite[Theorem 1.9.4] {BG}). In addition, it is easy to see that it is a diffusion process (see \cite[Proposition 1.9.10]{BG}).  Also, notice that $\bar{p}_t(x,y)$ is symmetric, it is uniquely associated with a conservative local regular Dirichlet form $(\bar{\mcE},\mcF)$ on $K$.
\end{remark}

\section{A self-similar form}\label{sec10}
In this section, by using an approach developed in \cite{CQ2} by the authors, we convert the limit form $(\bar {\mathcal E},\mathcal F)$ on $K$ constructed in the last section into  an SsDF $(\mathcal E,\mcF)$. (recall Section \ref{sec1}).

Before proceeding, we remark here that an SsDF $(\mcE,\mcF)$ on  $K$ is always strongly local. Indeed, $(\mcE,\mcF)$ is local by \cite[Lemma 3.12]{Hino} and the fact that $\{f\circ \Psi_i: i\in\{1,\cdots, N\}\}\subset \mcF$ for any $f\in\mcF\cap C(K)$ by (\ref{eqn1}), (\ref{eqn2}), and the denseness of $\mcF\cap C(K)$ in $(\mcF, \|\cdot\|_{\mcE_1}:=\sqrt{\mcE(\cdot)+\|\cdot\|^2_{L^2(K,\mu)}})$, and thus $(\mcE,\mcF)$ is strongly local as a conservative local Dirichlet form.

\begin{lemma}\label{lemma101}
Let $(\bar{\mcE},\mcF)$ be the Dirichlet form constructed in Section \ref{sec9}. For each $f\in L^2(K,\mu)$ and $r>0$, we define
\[
I_r(f)=r^{-d_H-d_W}\int_K\int_{B(x,r)}\big(f(x)-f(y)\big)^2\mu(dy)\mu(dx).
\]
Then $\mcF=\big\{f\in L^2(K,\mu):\sup_{0<r\leq 1}I_r(f)<\infty\big\}=\big\{f\in L^2(K,\mu): \liminf_{r\to 0}I_r(f)<\infty\big\}$ and
\[
	C_1\sup_{0<r\leq 1}I_r(f)\leq \bar{\mcE}(f)\leq C_2\liminf_{r\to 0}I_r(f), \text{ for any } f\in \mcF
\]
for some  constants $C_1,C_2>0$.
\end{lemma}

The proof of Lemma \ref{lemma101} is the same as  \cite[Lemma 4.2]{CQ2}, following the method of \cite{GHL1}. So we omit it here. The following proof of Theorem \ref{thm1} is also essentially the same as that of  \cite[Theorem 4.1]{CQ2}. However, since the dimension changes, we reproduce the proof for the convenience of readers.

\begin{proof}[Proof of Theorem \ref{thm1}]
Recall that $\rho=k^{d_H-d_W}$ is the resistance renormalization factor appearing in Theorem \ref{thm81}.

Throughout the proof, we take the setting that $\bar{\mcE}(f)=\infty$ if $f\notin\mcF$. For $n\geq 0$, define
\[
\mcE_n(f):=\sum_{w\in W_n}\rho^{-n}\bar{\mcE}(f\circ \Psi_w),\quad\forall f\in L^2(K,\mu).
\]
In addition, for each $f,g\in \mcF$, we define $\mcE_n(f,g):=\sum_{w\in W_n}\rho^{-n}\bar{\mcE}(f\circ \Psi_w,g\circ \Psi_w)$. 
We will show that $(\mcE_n,\mcF)$ is a Dirichlet form, and $C_1\bar{\mcE}(f)\leq \mcE_n(f)\leq C_2\bar{\mcE}(f),\forall f\in \mcF$ for some constants $C_1,C_2>0$ independent of $n$. Moreover, for the verification of (\ref{eqn1}), we will show that $\mcF_{n,c}=\mcF_{c}$, where $\mcF_{n,c}:=\{f\in C(K):f\circ \Psi_w\in \mcF,\forall w\in W_n\}$ and $\mcF_{c}:=\mcF\cap C(K)$.\vspace{0.2cm}
	
	We first show $\mcE_n(f)\leq C_2\bar{\mcE}(f),\forall f\in L^2(K,\mu)$. For any small $r$, we have
	\[
	\begin{aligned}
		I_r(f)\geq \sum_{w\in W_n}r^{-d_H-d_W}\int_{\Psi_wK}\int_{\Psi_wK\cap B(x,r)}\big(f(x)-f(y)\big)^2\mu(dy)\mu(dx)=\rho^{-n}\sum_{w\in W_n}I_{k^nr}(f\circ \Psi_w).
	\end{aligned}
	\]
	The inequality $\mcE_n(f)\leq C_2\bar{\mcE}(f)$ follows immediately by Lemma \ref{lemma101} and letting $r\to 0$. It is also clear that $\mcF_c\subset\mcF_{n,c}$.\vspace{0.2cm}
	
	The other inequality $\mcE_n(f)\geq C_1\bar{\mcE}(f)$ needs more care. For $f\in\mcF_{n,c}$, we need to estimate  $I_r(f)$ (for small $r$) across the cells:
	\begin{equation}\label{eqn101}
		I_r(f)-\rho^{-n}\sum_{w\in W_n}I_{k^nr}(f\circ \Psi_w):=\sum_{w\neq w'\in W_n}I_{r,w,w'}(f),
	\end{equation}
	where
	\[
	I_{r,w,w'}(f)=r^{-d_H-d_W}\int\int_{\{(x,y)\in \Psi_wK\times \Psi_{w'}K:d(x,y)<r\}}\big(f(x)-f(y)\big)^2\mu(dx)\mu(dy).
	\]\vspace{0.2cm}
	
\noindent\textit{Claim 1.
Let $w,w'\in W_n$ such that $\Psi_wK\cap \Psi_{w'}K=\{z\}$ for some $z\in K$. We pick $w''\in W_n$ such that $\Psi_{w''}K\cap \Psi_wK=\infty, \Psi_{w''}K\cap \Psi_{w'}K=\infty$. Then $I_{r,w,w'}(f)\leq C_3\big(I_{2r,w,w''}(f)+I_{2r,w',w''}(f)\big),\forall f\in L^2(K,\mu)$, for some constant $C_3>0$ independent of $w,w',n,r$.}
\begin{proof}[Proof of Claim 1]
We define
\[I'_{r,w,w'}(f):=r^{d_H-d_W}\fint_{\Psi_wK\cap B(z,r)}\fint_{\Psi_{w'}K\cap B(z,r)}\big(f(x)-f(y)\big)^2\mu(dx)\mu(dy),\]		
and define $I'_{r,w,w''}(f)$ and $I'_{r,w',w''}(f)$ in a same way, noticing that $z\in \Psi_{w''}K$ as well. Then noticing that $\mu\big(\Psi_wK\cap B(z,r)\big)\asymp r^{d_H}$, by the Minkowski inequality, we can see
\[
\begin{aligned}
	\sqrt{I_{r,w,w'}(f)}\lesssim &\sqrt{I'_{r,w,w'}(f)}\\=&\sqrt{r^{d_H-d_W}\fint_{\Psi_wK\cap B(z,r)}\fint_{\Psi_{w''}K\cap B(z,r)}\fint_{\Psi_{w'}K\cap B(z,r)}}\\&\overline{\Big(\big(f(x)-f(x')\big)+\big(f(x')-f(y)\big)\Big)^2\mu(dx)\mu(dx')\mu(dy)}\\
	\leq&\sqrt{I'_{r,w,w''}(f)}+\sqrt{I'_{r,w',w''}(f)}\\
	\lesssim&\sqrt{I_{2r,w,w''}(f)}+\sqrt{I_{2r,w',w''}(f)}.
\end{aligned}
\]
The claim follows immediately.
\end{proof}

\noindent\textit{Claim 2. For $n\geq 1$ and $r$ small enough, we have for some constant $C_4>0$ independent of $n,r$ such that	}
\[
\sum_{w\neq w'\in W_n,\#(\Psi_wK\cap \Psi_{w'}K)=\infty}I_{r,w,w'}(f)\leq C_4\rho^{-n}\sum_{w\in W_n}\bar{\mcE}(f\circ \Psi_w), \quad\forall f\in C(K).
\]
\begin{proof}
We fix a pair $(w,w')\in W_n\times W_n$ with $\#(\Psi_wK\cap \Psi_{w'}K)=\infty$, and let $L_{w,w'}=\Psi_wK\cap \Psi_{w'}K$. We now consider two possible cases.
	
Firstly, we assume that $L_{w,w'}$ is a line segment. We let $\nu$ be the Lebesgue measure (length) on $L_{w,w'}$, and we choose large enough $C_5>1$ and let $r'=C_5r$. Define for $l\geq 0$,
	\[
	\begin{aligned}
		I'_{r,w,w',l}&:=r^{d_H-d_W}\int_{L_{w,w'}}\frac{\nu(dz)}{r}\fint_{\Psi_wK\cap B(z,2^{-l}r')}\fint_{\Psi_{w'}K\cap B(z,2^{-l}r')}\big(f(x)-f(y)\big)^2\mu(dx)\mu(dy),\\
		D'_{r,w,w',l}&:=(2^{-l}r)^{d_H-d_W}\int_{L_{w,w'}}\frac{\nu(dz)}{2^{-l}r}\fint_{\Psi_wK\cap B(z,2^{-l}r')}\fint_{\Psi_wK\cap B(z,2^{-l+1}r')}\big(f(x)-f(x')\big)^2\mu(dx)\mu(dx').
	\end{aligned}
	\]
	Here we need to explain the reason that we choose the measure $\nu(dz)/r$ (or $\nu(dz)/(2^{-l}r)$) in the integral: for any $x,y$ close enough, and when $r$ is small enough, we have $\nu\big(\{z:x\in B(z,r'),y\in B(z,r')\}\big)\approx r$. In this case, by denoting $A_{w,l}(z)=\Psi_wK\cap B(z,2^{-l}r') $ and $A_{w',l}(z)=\Psi_{w'}K\cap B(z,2^{-l}r') $ for short, and by using the Minkowski inequality, we can check that
	\[
	\begin{aligned}
		&\sqrt{I'_{r,w,w',l}}\\=&\sqrt{r^{d_H-d_W}\int_{L_{w,w'}}\frac{d\nu(z)}{r}\fint_{A_{w,l}(z)}\fint_{A_{w',l}(z)}\fint_{A_{w,l+1}(z)}\fint_{A_{w',l+1}(z)}}\\&\overline{\Big(\big(f(x)-f(x')\big)+\big(f(x')-f(y')\big)+\big(f(y')-f(y)\big)\Big)^2\mu(dx)\mu(dy)\mu(dx')\mu(dy')}\\
		\leq&2^{-(l+1)(d_W-d_H+1)/2}\sqrt{D'_{r,w,w',l+1}}+\sqrt{I'_{r,w,w',l+1}}+2^{-(l+1)(d_W-d_H+1)/2}\sqrt{D'_{r,w',w,l+1}}.
	\end{aligned}
	\]
	Summing up the above inequality over $l$ and noticing that $f\in C(K)$, we get
	\[
	\sqrt{I'_{r,w,w',0}}\leq\sum_{l=1}^\infty 2^{-l(d_W-d_H+1) /2}\big(\sqrt{D'_{r,w,w',l}}+\sqrt{D'_{r,w',w,l}}\big).
	\]
	Then by using the Cauchy-Schwarz inequality, we can see that
	\begin{equation}\label{eqn102}
		I'_{r,w,w',0}\leq 2\big(\sum_{l=1}^\infty 2^{-l(d_W-d_H+1)/2}\big)\cdot\big(\sum_{l=1}^\infty 2^{-l(d_W-d_H+1)/2}(D'_{r,w,w',l}+D'_{r,w',w,l})\big).
	\end{equation}

	Secondly, we assume that $L_{w,w'}$ is a rectangle. We let $\nu$ be the Lebesgue measure (area) on $L_{w,w'}$, and we still choose large enough $C_5>1$ and let $r'=C_5r$. Define for $l\geq 0$,
	\[
   \begin{aligned}
	I'_{r,w,w',l}&:=r^{d_H-d_W}\int_{L_{w,w'}}\frac{\nu(dz)}{r^2}\fint_{\Psi_wK\cap B(z,2^{-l}r')}\fint_{\Psi_{w'}K\cap B(z,2^{-l}r')}\big(f(x)-f(y)\big)^2\mu(dx)\mu(dy),\\
	D'_{r,w,w',l}&:=(2^{-l}r)^{d_H-d_W}\int_{L_{w,w'}}\frac{\nu(dz)}{2^{-2l}r^2}\fint_{\Psi_wK\cap B(z,2^{-l}r')}\fint_{\Psi_wK\cap B(z,2^{-l+1}r')}\big(f(x)-f(x')\big)^2\mu(dx)\mu(dx').
    \end{aligned}
    \]
	Here we choose the measure $\nu(dz)/r^2$ (or $\nu(dz)/(2^{-2l}r^2)$) in the integral since for any $x,y$ close enough, and when $r$ is small enough, we have $\nu\big(\{z:x\in B(z,r'),y\in B(z,r')\}\big)\approx r^2$. In this case, similarly as above, we can check that
	\begin{equation}\label{eqn103}
		I'_{r,w,w',0}\leq \big(2\sum_{l=1}^\infty 2^{-l(d_W-d_H+2)/2}\big)\cdot\big(\sum_{l=1}^\infty 2^{-l(d_W-d_H+2)/2}(D'_{r,w,w',l}+D'_{r,w',w,l})\big).
	\end{equation}

	In both cases, if $C_5>1$ is chosen suitably, we have for any $r$ small enough (much smaller than the size of $L_{w,w'}$),
	\begin{equation}\label{eqn104}
		\begin{cases}
			I'_{r,w,w',0}\geq C_6I_{r,w,w'}(f),\\
			\sum\limits_{w'\in W_n\setminus \{w\}, \#(\Psi_{w'}K\cap \Psi_wK)=\infty}	D'_{r,w,w',l}\leq  C_6\cdot \rho^{-n}I_{2^{-l+2}k^nr'}(f\circ \Psi_w),
		\end{cases}
	\end{equation}
	where $C_6>0$ is a constant depending only on the fractal $K$. By combining (\ref{eqn102}), (\ref{eqn103}), (\ref{eqn104}), and Lemma \ref{lemma101}, noticing that $d_W-d_H+1>0$, the claim holds.
\end{proof}
	
Now combining Claim 1, Claim 2, (\ref{eqn101}) and Lemma \ref{lemma101}, we finally see
\[C_1\bar{\mcE}(f)\leq\mcE_n(f),\quad\forall f\in \mcF_{n,c},\]
for some $C_1>0$. Thus it is clear that $\mcF_{n,c}\subset\mcF_c$, whence $\mcF_c=\mcF_{n,c}$. \vspace{0.2cm}

\noindent\textit{Claim 3. $(\mcE_n,\mcF)$ is a Dirichlet form and $C_1\bar{\mcE}(f)\leq \mcE_n(f)\leq C_2\bar{\mcE}(f)$ for each $f\in \mcF$.}
\begin{proof}[Proof of Claim 3]
For each $f\in \mcF$, we take a Cauchy sequence $f_m\to f$ with $f_m\in\mcF_c$ with respect to the $\bar{\mcE}_1$-norm  $\|\cdot\|_{\bar{\mcE}_1}:=\sqrt{\bar{\mcE}(\cdot)+\|\cdot\|^2_{L^2(K,\mu)}}$. Then, $f_m$ converges in $L^2(K,\mu)$ to $f$, and it is straightforward to see that $\mcE_n(f_m)\to \mcE_n(f)$, so $C_1\bar{\mcE}(f)\leq \mcE_n(f)\leq C_2\bar{\mcE}(f)$ holds. The fact that $(\mcE_n,\mcF)$ is a Dirichlet form follows easily.
\end{proof}
	
We proceed to finish the construction following the idea of Kusuoka-Zhou \cite{KZ}. Let $\hat{\mcF}$ be a $\mathbb{Q}$-vector subspace of $\mcF$ with countably many elements which is dense in $\mcF$ with respect to the norm $\|\cdot\|_{\bar{\mcE}_1}$. To achieve this, one can simply choose a dense $\mathbb{Q}$-vector subspace $H$ of $L^2(K,\mu)$ with countably many elements, and let $\hat{\mcF}=U_1(H)$, where $U_1$ is the resolvent operator associated with $\bar{\mcE}_1$, i.e. $\bar{\mcE}_1(U_1f,g)=\int_{K}fgd\mu$, for any $f\in L^2(K,\mu)$ and $g\in \mcF$. Then by a diagonal argument, there is a subsequence $\{n_l\}_{l\geq 1}$ such that the limit
\[
{\mcE}(f):=\lim\limits_{l\to\infty}\frac{1}{n_l}\sum_{m=1}^{n_l}\mcE_m(f)
\]
	exists for any $f\in\hat{\mcF}$. Furthermore, for a general $f\in \mcF$, we can also prove that the limit ${\mcE}(f):=\lim\limits_{l\to\infty}\frac{1}{n_l}\sum_{m=1}^{n_l}\mcE_m(f)$ exists. Indeed, for any $\varepsilon>0$, we choose $g\in \hat{\mcF}$ such that $\|f-g\|_{\bar{\mathcal{E}}_1}<\varepsilon$. Then, by using the inequality $\big|\sqrt{\frac{1}{n_l}\sum_{m=1}^{n_l}\mcE_m(f)}-\sqrt{\frac{1}{n_l}\sum_{m=1}^{n_l}\mcE_m(g)}\big|^2\leq\frac{1}{n_l}\sum_{m=1}^{n_l}\mcE_m(f-g)\leq C_2\bar{\mcE}(f-g), \forall l\geq 1$, we have
	\begin{align*}
		&\limsup_{l,l'\to\infty}\Big|\sqrt{\frac{1}{n_l}\sum_{m=1}^{n_l}\mcE_m(f)}-\sqrt{\frac{1}{n_{l'}}\sum_{m=1}^{n_{l'}}\mcE_m(f)}\Big|\\
		\leq &\limsup_{l,l'\to\infty}\Big(\Big|\sqrt{\frac{1}{n_l}\sum_{m=1}^{n_l}\mcE_m(f)}-\sqrt{\frac{1}{n_l}\sum_{m=1}^{n_l}\mcE_m(g)}\Big|+\Big|\sqrt{\frac{1}{n_{l}}\sum_{m=1}^{n_{l}}\mcE_m(g)}-\sqrt{\frac{1}{n_{l'}}\sum_{m=1}^{n_{l'}}\mcE_m(g)}\Big|\\
		&\qquad\quad+\Big|\sqrt{\frac{1}{n_{l'}}\sum_{m=1}^{n_{l'}}\mcE_m(g)}-\sqrt{\frac{1}{n_{l'}}\sum_{m=1}^{n_{l'}}\mcE_m(f)}\Big|\Big)\\\leq
		& 2\sqrt{C_2}\varepsilon+\limsup_{l,l'\to 0}\Big|\sqrt{\frac{1}{n_{l}}\sum_{m=1}^{n_{l}}\mcE_m(g)}-\sqrt{\frac{1}{n_{l'}}\sum_{m=1}^{n_{l'}}\mcE_m(g)}\Big|=2\sqrt{C_2}\varepsilon.
	\end{align*}
	This implies that $\sqrt{\frac{1}{n_l}\sum_{m=1}^{n_l}\mcE_m(f)},l\geq 1$ is a Cauchy sequence, and so its limit exists. By the inequality $C_1\bar{\mcE}(f)\leq \mcE_m(f)\leq C_2\bar{\mcE}(f),\forall m\geq 1,f\in \mathcal{F}$, we know that
	\begin{equation}\label{eqn4x}
		C_1\bar{\mcE}(f)\leq {\mcE}(f)\leq C_2\bar{\mcE}(f),\quad\forall f\in\mcF.
	\end{equation}
	
	Immediately, the functional $\mcE$ induces a regular symmetric Dirichlet form $(\mcE,\mcF)$ on $L^2(K,\mu)$ which is conservative, i.e., satisfies ${\mcE}(\bm1)=0$. The self-similarity (\ref{eqn1}) of the domain $\mcF$ follows from the fact $\mcF_c=\mcF_{1,c}$, and the self-similarity (\ref{eqn2}) of $\mcE$ is immediate from the construction. The irreducibility of $(\mcE,\mcF)$ follows from Lemma \ref{lemma101}, (\ref{eqn4x}) and \cite[Theorem 2.1.11]{CF}.

    Finally, by (\ref{eqn4x}) and Theorem \ref{thm97}, the sub-Gaussian heat kernel estimate holds by standard arguments (for example, see \cite{AB,BCM,BM,GHL2} for various results about stability of Harnack inequalities, and also other equivalent characterizations that are stable under perturbation of the energies).
\end{proof}





\bibliographystyle{amsplain}

\end{document}